\documentclass[a4paper,12pt]{amsart}

\usepackage{geometry}
\usepackage[comma, numbers]{natbib}
\usepackage[inline]{enumitem}
\usepackage{graphicx}
\usepackage[textfont=it,tableposition=above,font=small]{caption}		

\usepackage{framed}					
\usepackage[foot]{amsaddr}	
\usepackage{url}						

\usepackage{amssymb}				
\usepackage{mathtools}			

\RequirePackage[colorlinks,citecolor=blue,urlcolor=blue]{hyperref}

\newtheorem{theorem}{Theorem}

\newtheorem{lemma}[theorem]{Lemma}

\theoremstyle{definition}

\newcommand{\twowidth}{0.4}             
\newcommand{\onewidth}{0.7}             

\newcommand{\R}{\mathbb R}																			
\newcommand{\N}{\mathbb N}																			
\newcommand{\Prob}[1][\R^d]{\mathcal P\left({#1}\right)}				
\newcommand{\Fsets}{\mathcal F}																	
\newcommand{\PP}{\mathsf P}																			
\newcommand{\as}{\mathrm {a.s.}}																
\newcommand{\HD}{hD}																						
\newcommand{\dd}{\,\mathrm{d}\,}																
\newcommand{\vol}[2][d]{\operatorname{vol}_{#1}\left(#2\right)}	
\newcommand{\asa}[2][]{\operatorname{asa}_{#1}\left(#2\right)}		
\newcommand{\B}[1][d]{B^{#1}}																		
\newcommand{\tr}{^\mathsf{T}}																		
\newcommand{\eqd}{\stackrel{\mathclap{\mbox{\tiny{\emph{d}}}}}{=}}	
\newcommand{\ID}{D_{\alpha_n}}                                      
\newcommand{\IDa}{D_\alpha}																					
\newcommand{\RHD}{RhD_\alpha}																		
\newcommand{\co}[1]{\operatorname{co}\left(#1\right)}													
\newcommand{\Ill}{\mathcal I}
\newcommand{\Ell}[1][\mu,\Sigma]{\mathcal E_{#1}}
\newcommand{\distg}{\mathsf{d}} 
\newcommand{\Haus}{\distg_H}		                                
\newcommand{\dist}[1][\Sigma]{\distg_{#1}}						

\author{Stanislav Nagy}
\author{Ji\v{r}\'i Dvo\v{r}\'ak}

\email{nagy@karlin.mff.cuni.cz}

\address{
	Charles University, Prague,
	Faculty of Mathematics and Physics,
	Department of Probability and Math. Statistics,
	Czech Republic
}

\title{Illumination depth}
\date{\today}

\begin{document}

\begin{abstract}
The concept of illumination bodies studied in convex geometry is used to amend the halfspace depth for multivariate data. The proposed notion of illumination enables finer resolution of the sample points, naturally breaks ties in the associated depth-based ordering, and introduces a depth-like function for points outside the convex hull of the support of the probability measure. The illumination is, in a certain sense, dual to the halfspace depth mapping, and shares the majority of its beneficial properties. It is affine invariant, robust, uniformly consistent, and aligns well with common probability distributions.
\end{abstract}

\maketitle

%
%

\section{Introduction}

Halfspace depth is a well known statistical tool that allows to define orders, ranks, and quantiles for multivariate datasets. Recent discoveries of connections between the depth and \emph{floating bodies} \cite{Brunel2019, Nagy_etal2018s} uncovered a vast body of knowledge on depth-like procedures in geometry and related fields. That section of mathematics, collected over the past 70~years, is little known in mathematical statistics. In this paper we focus on the paradigm of \emph{illumination} intimately connected to the depth, yet never studied with respect to its statistical applications. In convex geometry, illumination is known to be dual to the floating bodies (and, by extension, to the halfspace depth). We introduce it as a tool complementary to the halfspace depth, explore its statistical properties, and outline applications. We show that halfspace depth in conjunction with illumination allows to devise a nonparametric methodology similar to the depth, with many advantageous properties: \begin{enumerate*}[label=(\roman*)] \item conceptual and computational simplicity; \item full affine invariance; \item excellent robustness and large sample properties; \item the capacity of na\-tu\-ral\-ly breaking ties in data orderings; \item it can be used for the estimation of extreme quantile regions with efficiency comparable to the state-of-the-art approaches; \item it is well adjusted to elliptically symmetric distributions; and \item is powerful in applications such as classification.\end{enumerate*}

In Section~\ref{section:illumination bodies} we introduce illumination and motivate our research by drawing connections between illumination, floating bodies, and halfspace depth. The definition of the depth illumination and its properties are provided in Section~\ref{section:illumination depth}. The special case of elliptically symmetric distributions is treated in Section~\ref{section:elliptical distributions}. In Section~\ref{section:applications} we apply the new procedures to tie-breaking in depth-induced orderings, the estimation of extremal depth regions and in the classification task. Additional technical details, proofs and supplementary results from the simulation studies are gathered in the appendix.

%
%

\section{Illumination of convex bodies}	\label{section:illumination bodies}

Since its proposal by John W.~Tukey \cite{Tukey1975, Donoho1982}, the concept of the halfspace depth has occupied a prominent place in multivariate statistics. Its main idea is to rank the points in a $d$-dimensional Euclidean space $\R^d$, $d \geq 1$, according to their centrality as recognized with respect to (w.r.t.) a Borel probability measure $P$ on $\R^d$. The higher the depth of $x$ w.r.t. $P$ is, the more centrally located $x$ is within the probability mass of $P$. Points that maximize the depth over $\R^d$ generalize medians, and the loci of points whose depth exceeds given thresholds form equivalents of the inter-quantile regions from univariate inference. A remarkable array of applications of the depth can be found in \cite{Liu_etal1999, Zuo_Serfling2000, Liu_etal2006} and the references therein.

Suppose that all random variables are defined on a probability space $\left(\Omega, \Fsets, \PP\right)$ and denote the set of Borel probability measures on $\R^d$ by $\Prob$. The \emph{halfspace depth} of $x \in \R^d$ w.r.t. $X \sim P \in \Prob$ is the minimum probability of a halfspace that contains $x$
	\[	\HD\left(x;P\right) = \inf_{u \in \R^d} \PP\left( u\tr X \leq u\tr x \right).	\]
For $P$ uniform on a convex body $K$ (a convex compact subset of $\R^d$ with non-empty interior) the map $\HD\left(\cdot;P\right)$ closely relates to the \emph{floating bodies} of $K$, a concept used in geometry since the 19th century; for an extensive bibliography on the topic see \cite{Nagy_etal2018s}.

The (convex) floating body $K_\delta$ of a convex body $K$ with $\delta \geq 0$ is defined as the intersection of all halfspaces whose defining hyperplanes cut off a set of volume $\delta$ from $K$
	\[	K_\delta = \bigcap_{\vol{K \cap H^-} = \delta} H^+,	\]
where $H^+$ and $H^-$ are the two halfspaces with a boundary hyperplane $H \subset \R^d$ \cite{Schutt_Werner1990}. When $P \in \Prob$ is uniform on $K$ of unit volume,
	\[	K_\delta = \left\{ x \in \R^d \colon \HD\left(x;P\right) \geq \delta \right\}.	\]
For general $P\in\Prob$ and $\delta \in [0,1]$ the latter set is called the \emph{central region} of $P$ corresponding to $\delta$. In the sequel it will be denoted by $P_\delta$. 

In geometry, a particular collection of bodies somewhat dual to the floating bodies is known as the \emph{illumination bodies}. Floating bodies $K_\delta$ form subsets of $K$ and fill in $K$ from the inside as $\delta$ decreases to zero. In contrast, illumination bodies $K^\delta$ of $K$ are supersets of $K$ and approximate $K$ from the outside as $\delta \to 0$. Let $K$ be a convex body in $\R^d$ and let $\delta\geq 0$. The illumination body $K^\delta$ is the collection of all points whose volume of the convex hull with $K$ does not exceed the volume of $K$ by more than $\delta$
	\[	K^\delta = \left\{ x \in \R^d \colon \vol{\co{K \cup \left\{ x \right\}}} \leq \vol{K} + \delta \right\}.	\]

\begin{figure}[htpb]
\includegraphics[width=\twowidth\textwidth]{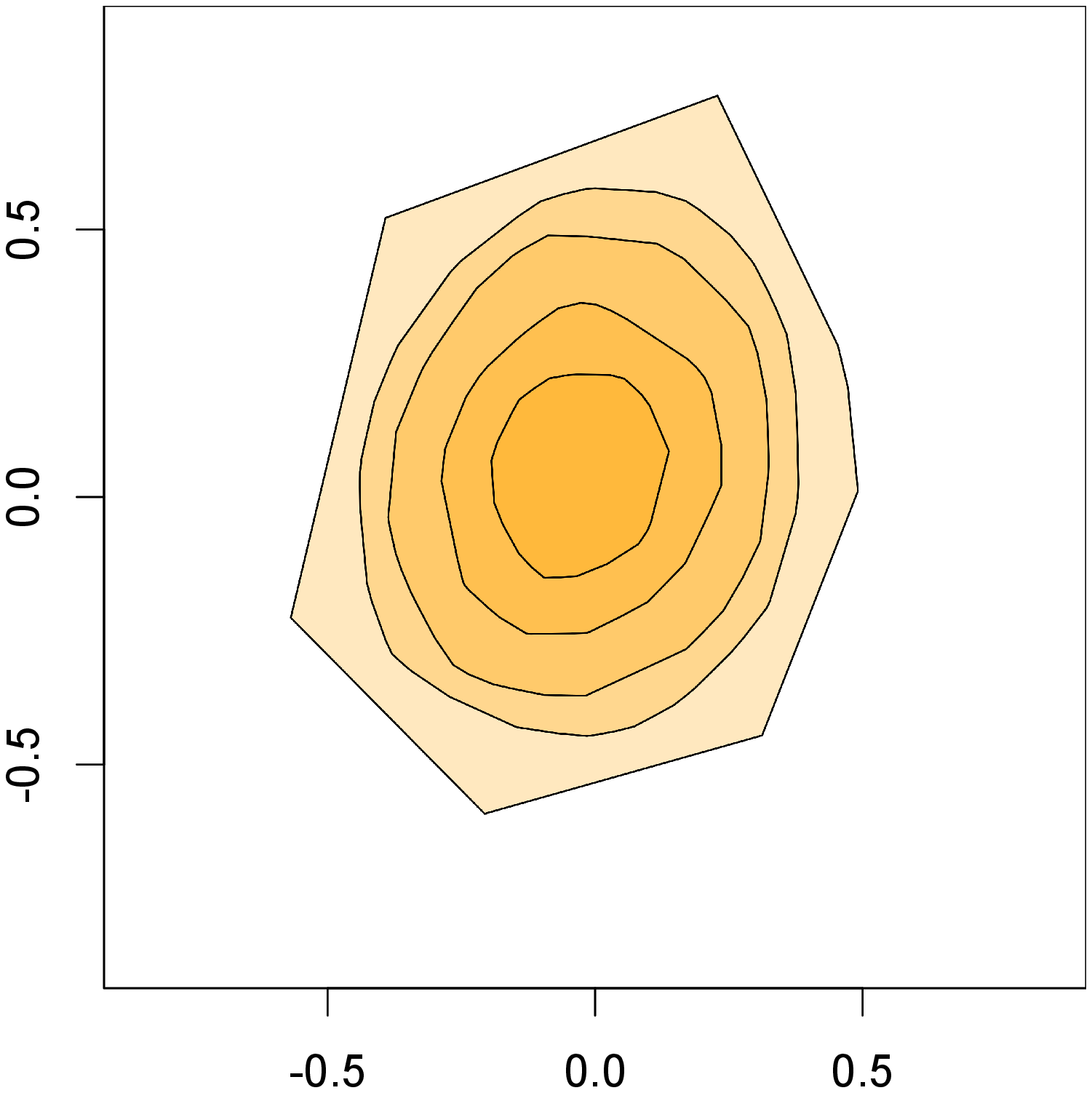} \quad \includegraphics[width=\twowidth\textwidth]{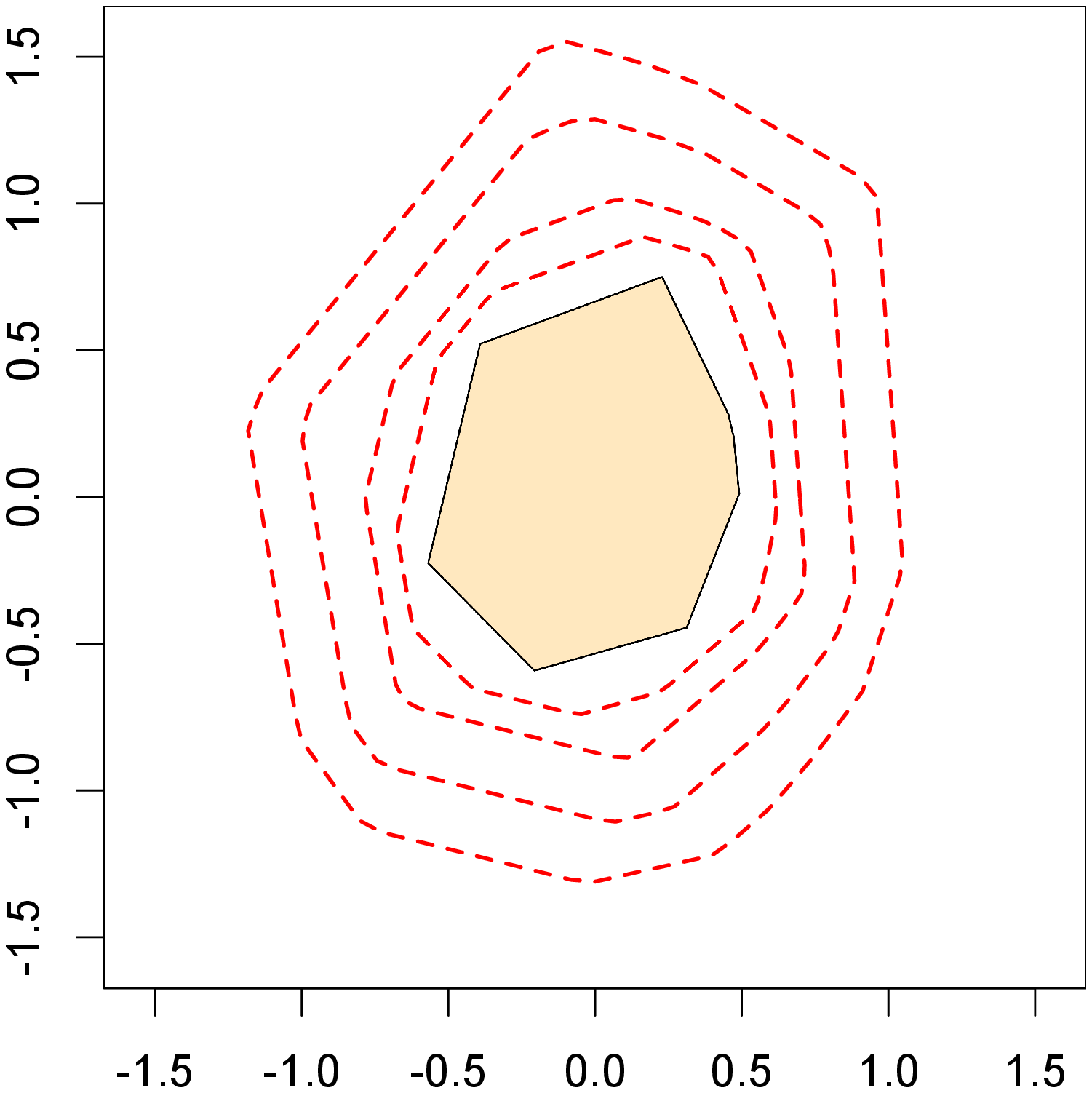}
\caption{Several floating bodies (left panel) and illumination bodies (right panel) of a polygon in $\R^2$ of unit volume with $\delta \in \left\{ 0.05, 0.1, 0.2, 0.3 \right\}$.}
\label{fig:floating and illumination}
\end{figure}

Illumination bodies were proposed by \citet{Werner1994} who found that several important properties of the floating bodies have analogues also when recast in terms of illumination. Just as the floating bodies, the illumination bodies are \begin{enumerate*}[label=(\roman*)] \item convex bodies; \item affine equivariant; \item ellipsoids if $K$ is an ellipsoid; and \item they converge to $K$ at the same rate as $K_\delta$ with $\delta$ decreasing to zero (see also the discussion in Section~\ref{section:duality} below).\end{enumerate*} Further important properties of illumination bodies can be found in \cite{Werner1996, Werner1997, Werner2006, Schutt1997, Schutt_Werner1992}. 

All these characteristics make the illumination bodies of great interest in statistics. Convexity of the upper level sets of depths is a trait that is often recognized as desirable \cite{Dyckerhoff2004, Serfling2006}. As argued by Donoho \cite{Donoho1982,Donoho_Gasko1992}, affine invariance in connection with robustness is the most valuable characteristic of the halfspace depth. For $P \in \Prob$ multivariate normal, or more generally, for $P$ elliptically symmetric with a density $f$, the depth central regions $P_\delta$ are known to be ellipsoids with the same centre and orientation as the ellipsoids given by the level sets of $f$.

It may appear that floating bodies and illumination bodies are inverse to each other, i.e. that $(K_\delta)^{\delta^\prime} = (K^{\delta^\prime})_\delta = K$ for $\delta = \delta^\prime$, or at least for $\delta^\prime$ chosen appropriately (see Figure~\ref{fig:floating and illumination}). The latter is true for $K$ an ellipsoid. But, none of those identities holds true generally, as can be seen already for $K$ a polytope. Indeed, $K_\delta$ is always strictly convex, but for $K$ a polytope, $K^\delta$ is again a polytope \cite{Schutt_Werner1990}; see also Figure~\ref{fig:inverse relations}. 

\begin{figure}[htpb]
\includegraphics[width=\twowidth\textwidth]{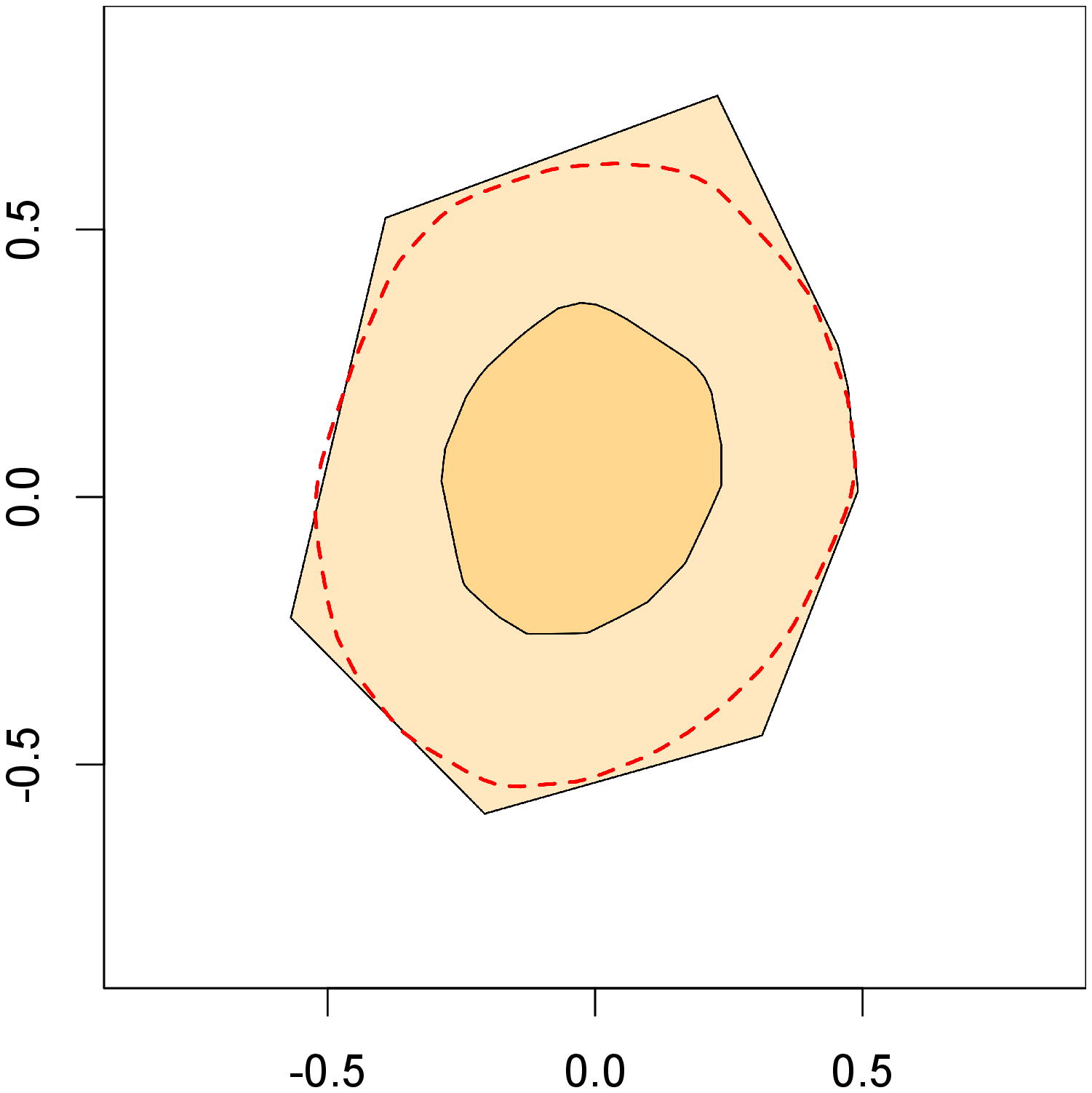} \quad \includegraphics[width=\twowidth\textwidth]{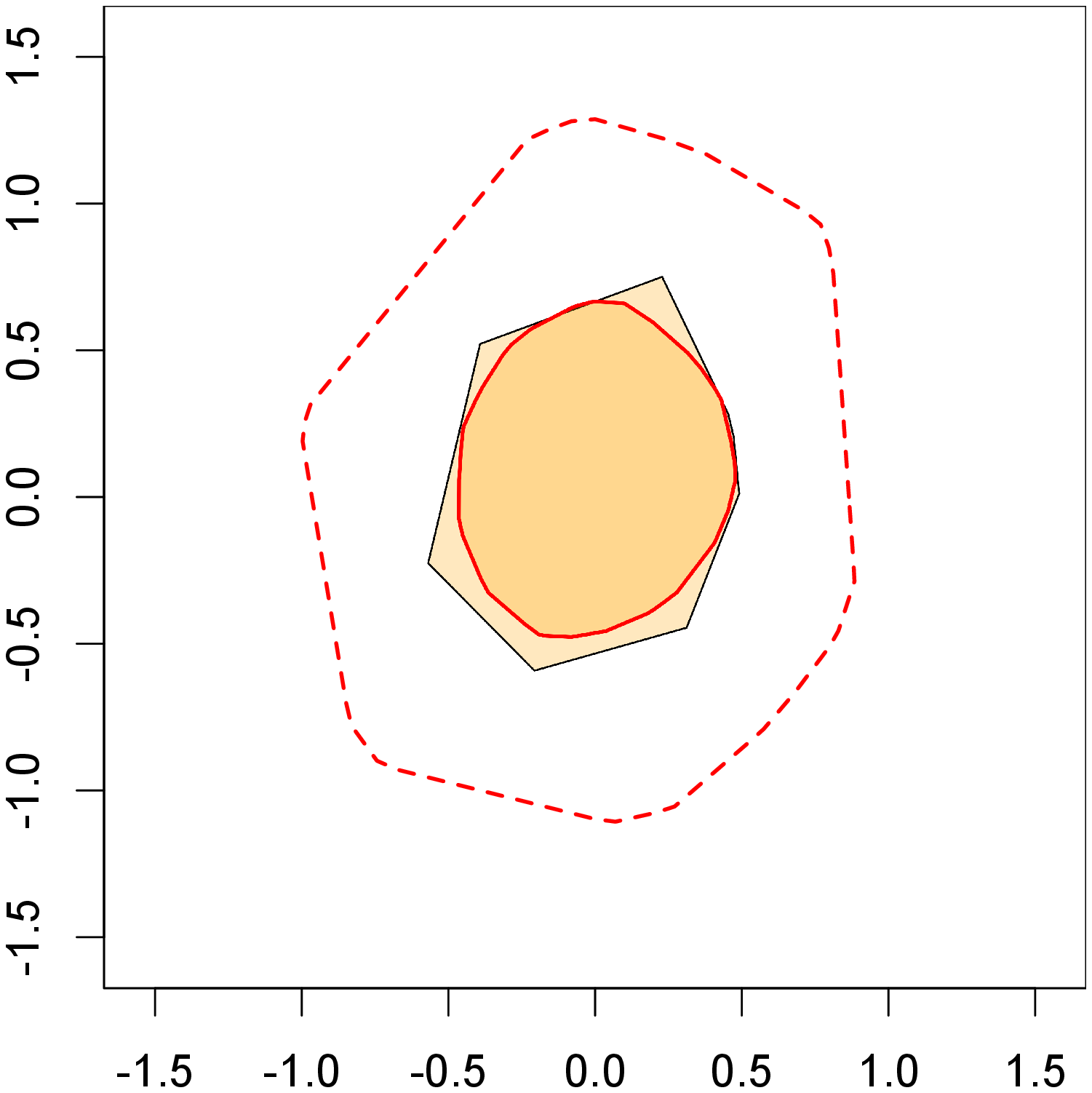}
\caption{Illumination of a floating body $(K_\delta)^{\delta^\prime}$, and a floating body of an illumination body $(K^{\delta^\prime})_\delta$ of a polygon $K$ with unit volume and $\delta = \delta^\prime = 0.2$. In general, the floating body is not a concept inverse to the illumination body. Nevertheless, there exist duality relations between the two constructions, see Section~\ref{section:duality}.}
\label{fig:inverse relations}
\end{figure} 

The open problem of finding an inverse floating body is much more involved; for an overview of some advances in this direction see \citep[Section~8]{Nagy_etal2018s}. For these reasons, it appears that at the current state of the art, illumination is the closest one can get to the inverse mapping of the floating body operator (or, by extension, to the inverse mapping of the halfspace depth).

\subsection{Illumination of ellipsoids}

The set of ellipsoids is central to the theories of illumination, floating bodies and halfspace depth. For instance, ellipsoids form an invariance class for all three transformations. In the sequel, it will thus be important to have exact expressions for the illumination bodies of ellipsoids. Define for a convex body $K$ and $x\in\R^d$ the \emph{illumination} of $x$ w.r.t. $K$ as
	\[	\Ill\left(x;K\right) = \vol{\co{K \cup \left\{ x \right\}}}.	\]
Obviously, $K^\delta = \left\{ x \in \R^d \colon \Ill\left(x;K\right) \leq \vol{K} + \delta \right\}$.

\begin{lemma}	\label{lemma:ellipsoid}
For $\mu\in\R^d$ and a symmetric, positive definite matrix $\Sigma\in\R^{d\times d}$ consider the distance $\dist\left(x,\mu\right) = \sqrt{\left(x - \mu\right)\tr \Sigma^{-1} \left(x - \mu\right)}$ and the ellipsoid given as its unit ball $\Ell = \left\{ x \in \R^d \colon \dist\left(x,\mu\right) \leq 1 \right\}$. For all $x \notin \Ell$ we have
	\[
	\begin{aligned}
	\Ill\left(x; \Ell\right) & = \vol{\Ell} \\
	& + \sqrt{\left\vert \Sigma \right\vert} \frac{\pi^{\frac{d-1}{2}}}{\Gamma\left(\frac{d+1}{2}\right)} \left( \frac{\dist\left(x,\mu\right)}{d} \left( 1 - \frac{1}{\dist\left(x,\mu\right)^2} \right)^{\frac{d+1}{2}} - \int_0^{\arccos\left(1/\dist\left(x,\mu\right)\right)} \sin^d (t) \dd t \right).
	\end{aligned}
	\]
\end{lemma}

The proof of Lemma~\ref{lemma:ellipsoid} can be found in Appendix~\ref{appendix:proofs} along with Lemma~\ref{lemma:g properties} that states some important properties of the function $g_d$ given by
	\begin{equation*}	
	g_d\left( \dist(x,\mu) \right) = \frac{\Ill\left(x; \Ell\right)}{\vol{\Ell}}.	
	\end{equation*}
The function $g_d$ is continuously differentiable and its derivative takes a rather simple form
	\[	g_d^\prime(t) = \frac{\Gamma\left(\frac{d}{2}+1\right)}{\sqrt{\pi}\,\Gamma\left(\frac{d+1}{2}\right)} \frac{1}{d} \left(1 - \frac{1}{t^2}\right)^{(d-1)/2} \quad\mbox{for }t\in(1,\infty).	\]
According to Lemma~\ref{lemma:g properties}, there exists an inverse function $g_d^{-1} \colon [1,\infty) \to [1,\infty)$ to $g_d$, and we can write
	\begin{equation}	\label{MD and illumination}	
	\dist(x,\mu) = g_d^{-1} \left( \frac{\Ill\left(x;\Ell\right)}{\vol{\Ell}} \right).	
	\end{equation}
This result will be of great importance in the subsequent analysis. 

\subsection{Duality considerations}	\label{section:duality}


The main impetus for considering floating/illumination bodies in geometry comes from the calculus of variations where many functionals over subsets of $\R^d$ are minimized by ellipsoids. Such statements are conveniently quantified using the \emph{affine surface area} $\asa{K}$ of a convex body $K$. For $K$ with a sufficiently smooth boundary $\partial K$, $\asa{K}$ is given as a certain integral over $\partial K$ \cite[Section~10.5]{Schneider2014}. Interestingly, it can be written also as the limit
	\begin{equation}	\label{affine surface area}
	\asa{K} = c_d \lim_{\delta \to 0} \frac{\vol{K} - \vol{K_\delta}}{\delta^{2/(d+1)}}	
	\end{equation}
for $c_d > 0$ a known constant. Thus, floating bodies can be used to extend the definition of the affine surface area to arbitrary convex bodies \cite{Schutt_Werner1990}, see also \cite[Section~5.3]{Nagy_etal2018s}. In connection with the motivation from the calculus of variations, by the important affine isoperimetric inequality \citep[Section~10.5]{Schneider2014} ellipsoids are the only maximizers of the affine surface area among all convex bodies with fixed volume.

As shown in \cite{Werner1994}, a definition equivalent to \eqref{affine surface area} can be stated also in terms of the illumination bodies
	\begin{equation*}	
	\asa{K} = b_d \lim_{\delta \to 0} \frac{\vol{K^\delta} - \vol{K}}{\delta^{2/(d+1)}}
	\end{equation*}
for $\asa{K}$ from \eqref{affine surface area}, $b_d>0$ a known constant, and $K$ a convex body. Thus, floating bodies and illumination bodies approach $K$ at the same rate, and in this respect they act dually to each other.

Another duality aspect of floating and illumination bodies was recently studied in \cite{Mordhorst_Werner2017, Mordhorst_Werner2017b}. There it was shown that, under appropriate conditions, the polar of a floating body of $K$ is, in a proper distance, close to an illumination body of the polar of $K$. Therefore, even though the exact correspondence between $\left(K_\delta\right)^\delta$ and $K$ does not hold true, solid evidence from convex geometry suggests that illumination is a concept that is naturally complementary to the floating body (and the halfspace depth). This pairing will be used throughout this paper to define robust, affine invariant extensions of the halfspace depth $\HD$.

%
%

\section{Illumination depth}	\label{section:illumination depth}

Denote the depth of the halfspace median of $P\in\Prob$ by $\Pi(P) = \sup_{x \in\R^d} \HD\left(x;P\right)$. By definition, $\Pi(P) \geq 1/2$ if and only if $P$ is halfspace symmetric \cite{Zuo_Serfling2000b}. Thus, $\Pi(\cdot)$ may be considered a measure of symmetry of probability distributions. For distributions with a density $1/2 \geq \Pi(P) \geq (d+1)^{-1}$ \citep[Section~4]{Nagy_etal2018s}. 

Let $\left\{X_n\right\}_{n=1}^\infty$ be a sequence of independent random variables with distribution $P\in\Prob$. Denote by $P_n \in \Prob$ the empirical measure of the first $n$ variables. We propose to amend the sample halfspace depth $\HD\left(\cdot;P_n\right)$ by considering not only the collection of all depth central regions 
	\[	P_{n,\delta} = \left\{ x \in \R^d \colon \HD\left(x;P_n\right) \geq \delta \right\},	\]
but also the illuminations on these. Let $\left\{ \alpha_n \right\}_{n=1}^\infty \subset [0,\Pi(P))$ be a non-increasing sequence of constants. For a random sample of size $n$ the \emph{halfspace-illumination depth} (or simply the \emph{illumination depth}) of $x \in \R^d$ w.r.t. the empirical measure $P_n$ is 
	\begin{equation}	\label{illumination depth}
		  \ID\left(x;P_n\right) = \begin{pmatrix} \HD\left(x;P_n\right) \\ \Ill\left(x;P_{n, \alpha_n}\right)/\vol{P_{n,\alpha_n}} \end{pmatrix}.	
	 \end{equation}
In Figure~\ref{fig:illumination depth} some level sets of this depth are displayed. Several remarks are in order.
	\begin{enumerate}[label=(R\arabic*), ref=(R\arabic*)]
	\item The depth $\ID$ has two components. The usual halfspace depth is a reliable indicator of centrality if $x$ lies inside the region $P_{n,\alpha_n}$ where there are enough observations to assess its degree of centrality. For such $x$, $\ID(x;P_n) = (\HD(x;P_n),1)\tr$, and the illumination does not affect depth-rankings. In contrast, the illumination evaluates the position of $x$ against the group of central points $P_{n,\alpha_n}$ that represent the main mass of $P_n$. Illumination thus plays a role in the ranking of extremal points whose depth is small, or zero. In the terminology of \cite{Zuo_Serfling2000}, illumination is an outlyingness function.
	\item The illumination is undefined if $\Pi(P_n) < \alpha_n$, or if $\vol{P_{n,\alpha_n}} = 0$. The first situation cannot occur for $n$ large enough since $\Pi(P_n) \xrightarrow[n\to\infty]{\as} \Pi(P)$ \citep[formula~(6.7)]{Donoho_Gasko1992}. Suppose then that $n$ is big enough for $P_{n,\alpha_n}$ to be non-empty. Because $P_{n,\alpha_n}$ is convex, its volume is zero only if that set is contained in a hyperplane. For random samples from continuous distributions, that happens with probability zero.
	\item \label{maximal depth} The maximum depth $\Pi(P)$ is typically unknown. In practice it can be replaced by $\Pi(P_n)$ or, if $P$ is regular enough, by a universal lower bound on $\Pi(P)$ (i.e. $1/2$ for halfspace (or elliptically) symmetric distributions, $\exp(-1)$ for log-concave distributions, or $(s+1)^{-1/s}$ for $s$-concave measures with $-1<s<0$ \citep[Theorem~3]{Nagy_etal2018s}). In any case, $\alpha_n$ should be bounded away from $\Pi(P)$ for the sets $P_{n,\alpha_n}$ to be sufficiently large.
	\item The illumination depth may be parametrized also in terms of probability --- $\alpha_n$ may be chosen as the maximum $\delta \in (0,1/2)$ with the property that $P_{n,\delta}$ contains at least $\left\lceil p_n\, n \right\rceil$ sample points for $\left\{ p_n \right\}_{n=1}^\infty \in (0,1)$ given. In the theoretical treatment of the depth we consider the simpler parametrization by $\alpha$.
	\end{enumerate}
	

\begin{figure}[htpb]
\includegraphics[width=\onewidth\textwidth]{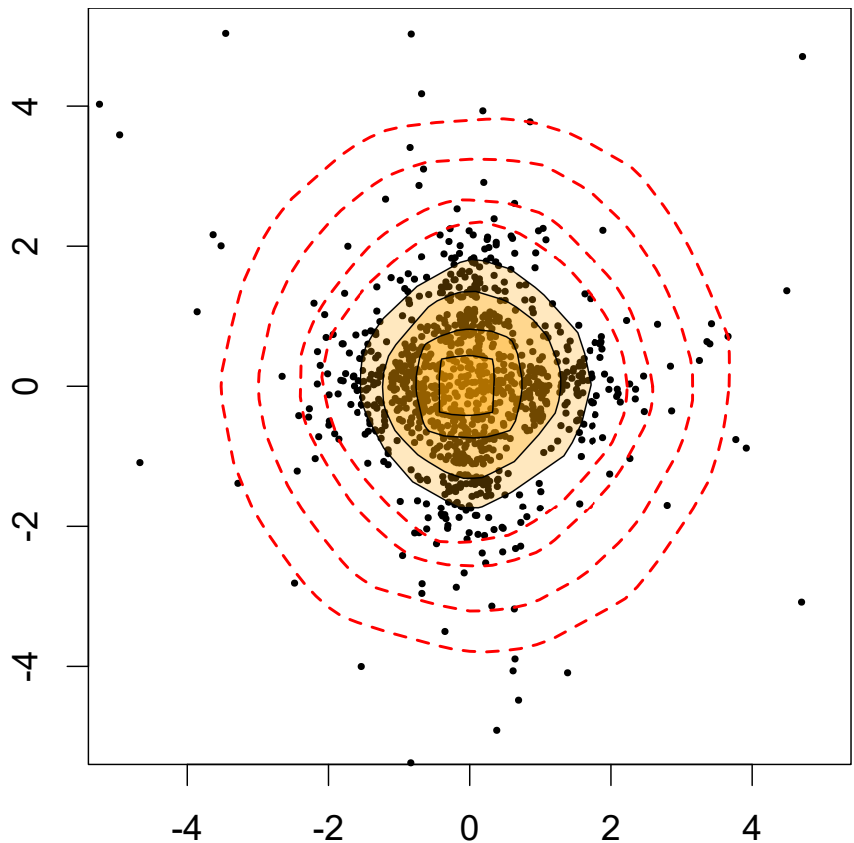}
\caption{Several level sets of the halfspace depth (orange regions) and the illumination with $\alpha = 1/20$ (regions with dashed red boundaries) for a bivariate random sample of size $1000$.}
\label{fig:illumination depth}
\end{figure}

The practical choice of the cut-off levels $\alpha_n$ determines the properties of the depth and should be selected with an application in mind. For estimation of extreme depth quantile regions $P_\delta$ with $\delta$ extremely small, it will be advantageous to take $\alpha_n$ to converge to $0$ slowly enough. That way, one obtains estimators comparable to the approaches taken in \cite{Einmahl_etal2015, He_Einmahl2017}. The disadvantage of this choice of thresholding is its lack of robustness. For a procedure with good robustness properties, a sequence of cut-offs bounded from below is more appropriate. Here we strive for robustness in conjunction with affine invariance. Therefore, we focus mainly on the latter situation; one example of the former scenario will be given in Section~\ref{section:quantile estimation}. For tie-breaking purposes, for a particular point $x$ the cut-off $\alpha_n$ may even be taken to depend on $x$. Then, if the halfspace depths of $x$ any $y$ coincide, illumination on some $P_{n,\alpha_n}$ with $\alpha_n > \HD\left(x;P_n\right)$ may help to decide which of the two points is deeper inside the mass of $P_n$, see Section~\ref{section:tie-breaking}. Without any substantial loss of generality\footnote{All results presented here could be extended in a straightforward way, at the cost of further technicalities.}, in our theoretical treatment we focus on the situation of a constant cut-off sequence $\alpha_n = \alpha < \Pi(P)$ for all $n$. As we will see in Section~\ref{section:robustness}, for $\alpha = \Pi(P)/(1 + \Pi(P))$ we obtain a procedure with excellent robustness properties. 

The mode of the cut-offs $\alpha_n$ affects the population version of $\ID$. If $\alpha_n \to 0$, only $\HD$ in \eqref{illumination depth} is relevant as $n\to\infty$. In this situation the appropriate population version of $\ID$ is the usual halfspace depth $\HD\left(\cdot;P\right)$ and by the standard consistency result for $\HD$ for any $P\in\Prob$ \cite{Donoho_Gasko1992}
	\[	\lim_{n\to\infty} \left\Vert \ID\left(x;P_n\right) - \left( \HD\left(x;P\right), 1 \right)\tr \right\Vert = 0 \quad\mbox{almost surely}.	\]
The situation is different for a constant $\alpha_n = \alpha \in (0,\Pi(P))$. Then we define the population version of $\IDa$ as 
	\begin{equation}	\label{illumination depth population}
	\IDa\left(x;P\right) = \begin{pmatrix} \IDa^{1}\left(x;P\right) \\ \IDa^{2}\left(x;P\right) \end{pmatrix} = \begin{pmatrix} \HD\left(x;P\right) \\ \Ill\left(x;P_{\alpha}\right)/\vol{P_{\alpha}} \end{pmatrix},	
	\end{equation}
which is well defined as soon as $\vol{P_{\alpha}} > 0$. It is immediate that \eqref{illumination depth population} reduces to \eqref{illumination depth} for $P = P_n$ and $\alpha = \alpha_n$. 

The illumination depth $\IDa$ satisfies the desirable properties of a statistical depth suggested in \cite{Zuo_Serfling2000, Serfling2006}.
	
\begin{theorem} \label{theorem:affine invariance}
Let $x \in \R^d$, $X \sim P_X = P\in\Prob$, and $\alpha \in (0,\Pi(P))$ be such that $\vol{P_\alpha}>0$. 
	\begin{enumerate}[label=(\roman*)]
	\item For any $A \in \R^{d\times d}$ non-singular and any $b \in \R^d$ we have $\IDa\left(A x + b;P_{A X + b}\right) = \IDa\left(x;P_{X}\right)$, where $P_{A X + b}$ is the distribution of the random vector $A X + b$.
	\item $P$ is halfspace symmetric around $x \in \R^d$ if and only if $\IDa\left(x;P\right) = \left(c, 1\right)\tr$ for $c \geq 1/2$.
	\item For $x$ that satisfies $\HD\left(x;P\right) = \Pi(P)$ and any $u \in \R^d$, the first and the second component of $\IDa\left(x + u \, t; P\right)$ are a non-increasing and a non-decreasing function of $t \geq 0$, respectively.
	\item All the level sets of the form $\left\{ x \in \R^d \colon \IDa^1\left(x;P\right) \geq c_1 \mbox{ and }\IDa^2\left(x;P\right) \leq c_2 \right\}$ are convex for all $c_1, c_2 \in \R$, and compact for all $c_1 > 0$, $c_2 \in \R$. 
	\item The first element of $\IDa\left(\cdot;P\right)$ is a function that is upper semi-continuous. Its second element is a function that is continuous on $\R^d$. If $P$ has a density, $\IDa\left(\cdot;P\right)$ is continuous on $\R^d$.
	\item As $\left\Vert x \right\Vert \to \infty$, the first element of $\IDa\left(x;P\right)$ converges uniformly to zero. Its second element increases uniformly to $\infty$.
	\end{enumerate}
\end{theorem} 

\subsection{Uniform consistency}

Recall that $P$ is said to have contiguous support if the support of $P$ cannot be separated by a slab between two parallel hyperplanes. Connected support is contiguous. 

\begin{theorem}	\label{theorem:consistency}
Let $P \in \Prob$ be absolutely continuous with contiguous support, and let $\alpha \in (0,\Pi(P))$. Then the illumination is locally uniformly consistent for $P$, that is for any $K \subset \R^d$ compact
	\[	\sup_{x \in K} \left\vert \Ill\left(x;P_{n,\alpha}\right) - \Ill\left(x;P_\alpha\right) \right\vert \xrightarrow[n\to\infty]{\as}0.	\]
If, furthermore, $P$ satisfies Assumptions~1 and~2 from \cite{Brunel2019}, and $K_n$ is a sequence of sets such that $K_n \cup P_\alpha$ is contained in a ball of radius $R_n$, then 
	\[	
	\begin{aligned}
	\sup_{x \in K_n} \left\vert \Ill\left(x;P_{n,\alpha}\right) - \Ill\left(x;P_\alpha\right) \right\vert = \mathcal O_{\PP}\left( \frac{\max\left\{1,R_n^{d-1}\right\}}{\sqrt{n}} \right), \\
	\sup_{x \in K_n} \left\vert \frac{\Ill\left(x;P_{n,\alpha}\right)}{\vol{P_{n,\alpha}}} - \frac{\Ill\left(x;P_\alpha\right)}{\vol{P_\alpha}} \right\vert = \mathcal O_{\PP}\left( \frac{\max\left\{1,R_n^{d-1}\right\}}{\sqrt{n}} \right), \\
	\sup_{x \in K_n} \left\Vert \IDa(x;P_n) - \IDa\left(x;P\right) \right\Vert = \mathcal O_{\PP}\left( \frac{\max\left\{1,R_n^{d-1}\right\}}{\sqrt{n}} \right). \\	
	\end{aligned}
	\]
In particular, if $d=1$ and $R_n>0$, or $d>1$ and $R_n = o\left( n^{1/(2 (d - 1))} \right)$, 
	\[
	\begin{aligned}
	\sup_{x \in K_n} \left\vert \Ill\left(x;P_{n,\alpha}\right) - \Ill\left(x;P_\alpha\right) \right\vert & \xrightarrow[n\to\infty]{\PP}0, \\
	\sup_{x \in K_n} \left\vert \frac{\Ill\left(x;P_{n,\alpha}\right)}{\vol{P_{n,\alpha}}} - \frac{\Ill\left(x;P_\alpha\right)}{\vol{P_\alpha}} \right\vert & \xrightarrow[n\to\infty]{\PP}0, \\
	\sup_{x \in K_n} \left\Vert \IDa\left(x;P_n\right) - \IDa\left(x;P\right) \right\Vert & \xrightarrow[n\to\infty]{\PP}0.
	\end{aligned}
	\]
\end{theorem}

Thanks to the sharp bound on the volume difference of convex bodies devised in Lemma~\ref{lemma:Lipschitz} in the appendix, it is possible to state an explicit deviation inequality such as that from \cite[Theorem~2]{Brunel2019}. We omit this result for brevity. 

Note that the technical Assumptions~1 and~2 from \cite{Brunel2019} are not restrictive at all. They are satisfied, for instance, if $P$ has a density that is bounded away from zero in a large enough superset of $P_\alpha$, if the density of $P$ is continuous, positive and decreases fast enough \citep[Assumption~3]{Brunel2019}, or, for the case of elliptically symmetric distributions, if the density of $P$ is continuous and positive at the boundary of $P_\alpha$. 

Neither illumination nor the illumination depth are consistent uniformly over unbounded sets in $\R^d$. This is illustrated in an example in Section~\ref{section:consistency} in the appendix. This does not limit practical applications. For any sequence of compact sets $K_n$, allowed to increase in size with $n$, Theorem~\ref{theorem:consistency} guarantees uniform consistency.

\subsection{Robustness}	\label{section:robustness}

We now turn to the robustness properties of the illumination. For a data set that corresponds to an empirical measure $P_n \in \Prob$, the addition breakdown point of an estimator $T = T(P_n)$ is defined \cite{Donoho1982} as
	\begin{equation}	\label{breakdown point}
	BP\left(T,P_n\right) = \min\left\{ \frac{m}{n + m} \colon \sup_{Y^{(m)}} \distg\left(T\left(Q_{m+n}\right),T(P_n)\right) = \infty \right\},	
	\end{equation}
where $Y^{(m)}$ is an $m$-tuple of (not necessarily distinct) points in $\R^d$, $Q_{m+n}$ is the empirical measure that assigns probability $1/(m+n)$ to all the data points from $P_n$ and $Y^{(m)}$, and $\distg$ is an appropriate distance in the target space of $T$. For the usual halfspace depth, the finite sample breakdown point of the central region $T_\delta(P_n) = P_{n, \delta}$ for $\delta \in (0, \Pi(P_n))$ can be derived from \cite[Section~3]{Donoho_Gasko1992}. With the Hausdorff distance of compact sets $K, L \subset \R^d$ 
	\begin{equation}	\label{Hausdorff distance}
	\Haus\left(K,L\right) = \max\left\{ \max_{x \in K} \inf_{y \in L} \left\Vert x - y \right\Vert, \max_{x \in L} \inf_{y \in K} \left\Vert x - y \right\Vert \right\}	
	\end{equation} 
in \eqref{breakdown point} in place of $\distg$, from the advances in \cite{Donoho_Gasko1992} (for a formal proof see Section~\ref{proof:breakdown} in the appendix) it can be shown that
	\begin{equation}	\label{halfspace depth BP}
	BP\left( T_\delta, P_n \right) = \frac{\left\lceil \delta/(1-\delta) n \right\rceil}{n + \left\lceil \delta/(1-\delta) n \right\rceil} \quad \mbox{if }\delta \leq \Pi(P_n)/(1 + \Pi(P_n)).	
	\end{equation} 
For $\delta$ close to zero and $n$ large, the finite sample breakdown point of $P_{n,\delta}$ is of order $\delta$. This corroborates the well known fact that the outer, extremal regions of the halfspace depth are not robust. For instance, in any configuration of $n$ points in $\R^d$, to dislocate the largest proper depth region $P_{n,1/n}$ --- the convex hull of the sample points --- it is enough to add a single contaminating observation. On the other hand, for any $\delta \geq \Pi(P_n)/(1 + \Pi(P_n))$, the central region $P_{n,\delta}$ is rather stable with a positive breakdown point not smaller than $\Pi(P_n)/(1 + \Pi(P_n))$. If $P_n$ corresponds to a random sample from $P\in\Prob$, as $n\to\infty$ the latter breakdown point approaches $\Pi(P)/(1+\Pi(P))$ \citep[Proposition~3.3]{Donoho_Gasko1992}. Thus, the more regular $P$ is, the more robust its inner halfspace depth central regions are. 

We now give expressions for the finite sample breakdown point of the illumination.  

\begin{theorem}	\label{theorem:breakdown}
For an empirical measure $P_n\in\Prob$, $\alpha < \Pi(P_n)$, and 
	\begin{equation}	\label{level set of illumination}
	T_{\alpha,\delta}(P_n) = \left\{ x \in \R^d \colon \Ill(x;P_{n,\alpha})/\vol{P_{n,\alpha}} \leq \delta \right\}
	\end{equation}
for $\delta \geq 1$, we have that
	\[	
	\begin{aligned}
	BP\left(T_{\alpha,\delta},P_n\right) & = \frac{\left\lceil \alpha/(1-\alpha) n \right\rceil}{n + \left\lceil \alpha/(1-\alpha) n \right\rceil} & & \mbox{if }\alpha \leq \Pi(P_n)/(1 + \Pi(P_n)), \\
	BP\left(T_{\alpha,\delta},P_n\right) & \geq \frac{\Pi(P_n)}{1 + \Pi(P_n)} & & \mbox{otherwise}.
	\end{aligned}
	\]
If $P_n$ is the empirical measure of a random sample from $P\in\Prob$ of size $n$, then it almost surely holds true that
	\[	\lim_{n \to \infty} BP\left(T_{\alpha,\delta},P_n\right) = 
																													  \begin{cases}
																														\alpha & \mbox{if } \alpha < \Pi(P)/(1 + \Pi(P)),	\\
																														\frac{\Pi(P)}{1 + \Pi(P)} & \mbox{otherwise}.
																														\end{cases}	\]
\end{theorem} 

Theorem~\ref{theorem:breakdown} asserts that the illumination is quite robust --- unlike for the halfspace depth, its breakdown point does not depend on $\delta$. For $\alpha = \Pi(P)/(1 + \Pi(P))$ we have, in view of Remark~\ref{maximal depth}, that for any configuration of $n$ points $P_n$ we have $BP\left(T_{\alpha,\delta},P_n\right) \geq 1/(d+2)$, and the illumination always possesses a strictly positive breakdown point. This simple bound is, however, rather pessimistic. If $P_n$ is a random sample from a log-concave distribution, $\lim_{n\to\infty} BP\left(T_{\alpha,\delta},P_n\right) = 1/(1 + e) \approx 0.27$ almost surely, and for $P$ halfspace symmetric, $\lim_{n\to\infty} BP\left(T_{\alpha,\delta},P_n\right) = 1/3$ almost surely. Overall, for a random sample of size $n$ from $P$ that is regular enough, it takes, with large probability, at least almost $m = n/2$ points to be added to the dataset to disturb the illumination procedure completely. This contrasts sharply with the usual halfspace depth. According to \eqref{halfspace depth BP}, for any distribution $P$, $P_{n,\delta}$ alone will be disrupted completely already if around $m = n \delta /(1 - \delta)$ contaminants are strategically added to the sample. 
For numerical results see Section~\ref{section:applications}.

\subsection{Computational cost}

Illumination is computed in two steps. Given a dataset $P_n$, firstly a single central region $P_{n,\alpha}$ is computed. This set is a convex polytope. In the second step, illumination of $x$ onto $P_{n,\alpha}$ is evaluated by employing algorithms for the computation of convex hulls of points and volumes of convex polytopes. 

Computation of $P_{n,\alpha}$ is generally a demanding task. For a single level $\alpha$ required for the illumination, recent advances made this feasible in dimension $d$ up to five or ten and moderate sample sizes $n$; see \cite{Liu_etal2019} and the references therein. 

Finding the illumination of $x$ is already quite well explored. Both problems of finding a convex hull of a dataset and its volume are standard in computational geometry. A great number of effective algorithms exists in this direction \cite{Bueler_etal2000}; for a more recent contribution see \cite{Emiris_Fisikopoulos2018}. 

In our \texttt{R} implementation we combine tools from package \texttt{TukeyRegion} with the \texttt{R} interface to the \texttt{Qhull}\footnote{\url{http://www.qhull.org/}} toolbox implemented in package \texttt{geometry}. This code, given in Appendix~\ref{appendix:code}, handles hundreds of observations in dimensions $d\leq 5$ without substantial difficulties, see Table~\ref{table:complexity}. More efficient algorithms for the computation of volumes of convex polytopes can be used to speed up the computation. The computation of the single halfspace depth region $P_{n,\alpha}$ is the true bottleneck of this procedure, especially in higher dimensions.

\begin{table}[ht]
\centering
\begin{tabular}{lll |rrr}
  \hline
          \multicolumn{2}{l}{ Setup $\backslash$ \texttt{R} package}        &  & \multicolumn{1}{l}{ \texttt{TukeyRegion}} & \multicolumn{1}{l}{ \texttt{geometry}} & \multicolumn{1}{l}{ \texttt{ddalpha}} \\ 
            &         &  & \multicolumn{1}{l}{                      } & \multicolumn{1}{l}{                   } & \multicolumn{1}{l}{                  } \\ 
   \hline
	$n = 50$  & $d = 2$ &  &                  0.02 &               0.77 &              0.08 \\ 
            & $d = 3$ &  &                  0.03 &               1.03 &              0.05 \\ 
            & $d = 4$ &  &                  0.27 &               9.61 &              0.06 \\ 
  $n = 200$ & $d = 2$ &  &                  0.03 &               0.70 &              0.20 \\ 
            & $d = 3$ &  &                  0.54 &               1.88 &              0.21 \\ 
            & $d = 4$ &  &                295.64 &              70.36 &              0.20 \\ 
  $n = 500$ & $d = 2$ &  &                  0.31 &               0.71 &              0.51 \\ 
            & $d = 3$ &  &                  9.99 &               2.31 &              0.48 \\ 
            & $d = 4$ &  &              72010.61 &             220.21 &              0.50 \\ 
   \hline
\end{tabular}
\caption{Computation times (in seconds) for the evaluation of a single halfspace depth region $P_{n,\alpha}$ that is illuminated on (\texttt{TukeyRegion}); the illumination of $1000$ randomly sampled points onto $P_{n,\alpha}$ (\texttt{geometry}); and the usual halfspace depth of these $1000$ points w.r.t. $P_n$ (\texttt{ddalpha}). In all cases the depth is computed w.r.t. a random sample from a standard $d$-variate normal distribution of size $n$.}
\label{table:complexity}
\end{table}

\section{Illumination for elliptically symmetric distributions}	\label{section:elliptical distributions}

Now we focus our attention to elliptically symmetric distributions, or more generally, to those distributions $P$ whose halfspace depth central regions $P_\delta$ are close to ellipsoids. It may appear that the latter assumption is restrictive. Nonetheless, it is known that any sufficiently regular distribution $P$ possesses central regions $P_{\delta}$ that are bound to have almost ellipsoidal shapes. This was first observed by \citet{Milman_Pajor1989}, see the proposition in the appendix of that paper. There it is shown that for $P$ uniform on a symmetric convex body $K$ every $P_\delta$ is uniformly, up to a known constant, isomorphic to an ellipsoid. References to further extensions of that groundbreaking result to (asymmetric) log-concave or $s$-concave measures with additional discussion can be found in \cite[Section~7]{Nagy_etal2018s}. Thus, even though formally the restriction to elliptical symmetry in this section is real, at least heuristically all these results will hold true more widely.

We start by collecting some useful information about elliptically symmetric distributions. For references to these results see \cite{Serfling_symmetry, Fang_etal1990}. $P \in \Prob$ is said to be spherically symmetric if the measure $P$ is invariant with respect to all orthogonal transformations on $\R^d$. It is elliptically symmetric if it can be represented as an affine image of a spherically symmetric distribution --- we say that $X \sim P \in \Prob$ is elliptically symmetric if $X \eqd \mu + A Z$ for $\mu \in \R^d$, $A \in \R^{d \times k}$, and $Z = \left(Z_1, \dots, Z_k\right)\tr \sim Q \in \Prob[\R^k]$ is spherically symmetric. The symbol $\eqd$ stands for ``is equal in distribution". Note that $Q$ is uniquely characterized by the symmetric\footnote{By symmetry of the distribution function we mean that $F(z) = 1 - F(-z)$ at all points of continuity of $F$.} distribution function $F(z) = \PP\left( Z_1 \leq z \right)$, $z \in \R$.  We also write $X \sim P = EC\left(\mu,\Sigma,F\right)$ for $\Sigma = A A\tr \in \R^{d \times d}$. Because $\mu + c A Z = \mu + A (c Z)$ and $\widetilde{F}(z) = \PP\left( c Z_1 \leq z\right) = F\left(z/c\right)$ for any $c>0$ and $z \in \R$, $EC\left(\mu,c^2 \Sigma, F\right) = EC\left(\mu, \Sigma, \widetilde{F}\right)$. To identify $P$ uniquely, we therefore in this section consider mainly elliptically symmetric distributions whose scatter matrix $\Sigma$ is normalized to have a unit determinant $\left\vert \Sigma \right\vert = 1$.

For $X \sim P = EC\left(\mu,\Sigma,F\right) \in \Prob$ with $\Sigma$ positive definite and $Q = \Sigma^{-1/2}\left(X - \mu\right)$ the spherically symmetric affine image of $P$, affine invariance of $\HD$ gives a simple expression for the halfspace depth of any $x \in \R^d$
	\begin{equation}	\label{elliptical depth}
	\HD\left(x;P\right) = \HD\left(\Sigma^{-1/2}\left(x - \mu\right);Q\right) = F\left( - \left\Vert \Sigma^{-1/2}\left(x - \mu\right) \right\Vert \right) = F\left( - \dist\left(x,\mu\right) \right).	
	\end{equation}
It is also not hard to realise (see, e.g., \citep[Theorem~34]{Nagy_etal2018s}) that $P = EC(\mu,\Sigma,F)$ if and only if all the halfspace depth central regions $P_{\alpha}$ with $0 < \alpha < 1/2$ are ellipsoids of the form
	\begin{equation*}	
	P_{\alpha} = \left\{ x \in \R^d \colon \dist\left(x,\mu\right) \leq - F^{-1}\left(\alpha\right) = F^{-1}\left( 1 - \alpha \right) \right\} = \Ell[\mu,\Sigma \left(F^{-1}\left(1-\alpha\right)\right)^2].
	\end{equation*}
Thus, by the expression for the illumination of ellipsoids \eqref{MD and illumination}, the lower level sets of the illumination \eqref{level set of illumination} with $\delta \geq 1$ are also all ellipsoids
	\[
	\begin{aligned}
	T_{\alpha,\delta}(P) & = \left\{ x \in \R^d \colon \dist[\Sigma \left( F^{-1}\left(1 - \alpha\right) \right)^2]\left(x,\mu\right) \leq g_d^{-1}\left(\delta\right)	\right\} \\
	& = \left\{ x \in \R^d \colon \dist(x,\mu) \leq F^{-1}\left(1 - \alpha\right) g_d^{-1}\left(\delta\right)	\right\}.
	\end{aligned}
	\]
In particular, for $x \notin P_{\alpha}$, 
	\begin{equation}	\label{illumination for measure}
	g_d^{-1}\left(\frac{\Ill\left(x;P_\alpha\right)}{\vol{P_\alpha}}\right) 
	= \dist[\Sigma \left(F^{-1}\left(1 - \alpha\right)\right)^2](x,\mu) 
	= \frac{\dist(x,\mu)}{F^{-1}\left(1 - \alpha\right)},
	\end{equation}
and both the halfspace depth upper level sets and the illumination lower level sets are ellipsoids centred at $\mu$ with the same orientation as the Mahalanobis ellipsoid $\Ell$, or equivalently, the density contours of $P$. A first application of this property is straightforward. For $P$ with a unimodal elliptically symmetric density, the depth-induced centre-outward ordering of the points $x\in\R^d$ is a function of their Mahalanobis distance $\dist(x,\mu)$ from the mode $\mu$ --- the smaller $\dist\left(x,\mu\right)$, the more central $x$ is. To get a robust, affine invariant depth-based estimator of $\dist\left(x,\mu\right)$, employ \eqref{elliptical depth} and \eqref{illumination for measure} to see that
	\begin{equation}	\label{Mahalanobis distance population}
	\dist(x,\mu) =  \begin{cases}
											F^{-1}\left(1 - \HD\left(x;P\right)\right) & \mbox{if }\HD\left(x;P\right) \geq \alpha, \\
											F^{-1}\left(1 - \alpha\right) g_d^{-1}\left(\frac{\Ill\left(x;P_\alpha\right)}{\vol{P_\alpha}}\right) & \mbox{otherwise}.
											\end{cases}
	\end{equation}
Because $F$ is non-decreasing and $g_d^{-1}$ strictly increases on its domain by Lemma~\ref{lemma:g properties}, $\dist(x_1,\mu) \leq \dist(x_2,\mu)$ is equivalent with one of the three possible situations: \begin{enumerate*}[label=(\roman*)] \item either $\HD\left(x_1;P\right) \geq \alpha > \HD\left(x_2;P\right)$; or \item if both depths are high, $\HD\left(x_1;P\right) \geq \HD\left(x_2;P\right) \geq \alpha$; or \item if both depths are low, $\alpha > \max\left\{\HD\left(x_1;P\right),\HD\left(x_2;P\right)\right\}$, and at the same time $\Ill\left(x_1;P_\alpha\right) \leq \Ill\left(x_2;P_\alpha\right)$. \end{enumerate*} This yields the following centre-outwards ranking procedure for points $x_1, \dots, x_m \in \R^d$ (the lowest rank is for the most central position):
	\begin{enumerate}[label=(\roman*)] 
	\item compute the depth $\IDa$ of all points $x_1, \dots, x_m$; 
	\item the $k$ points whose halfspace depth is at least $\alpha$ are ranked as the $k$ most central points $x_{(1)} \preceq \dots \preceq x_{(k)}$ according to their decreasing halfspace depth, i.e. $x_i \preceq x_j$ if $\HD\left(x_i;P\right) \geq \HD\left(x_j;P\right)$;
	\item the $m-k$ remaining points are ranked as the less central points $x_{(k+1)} \preceq \dots \preceq x_{(m)}$ according to their increasing illumination, i.e. $x_i \preceq x_j$ if $\Ill\left(x_i;P_\alpha\right) \leq \Ill\left(x_j;P_\alpha\right)$. 
	\end{enumerate}
This robust ranking can produce ties. However, they are easy to break. If $\HD\left(x_i;P\right) = \HD\left(x_j;P\right) \geq \alpha$, use the illumination and set $x_i \prec x_j$ if $\Ill\left(x_i;P_{\alpha^\prime}\right) < \Ill\left(x_j;P_{\alpha^\prime}\right)$ for some $\alpha^\prime > \HD\left(x_i;P\right)$. If the original tied ranks of $x_i$ and $x_j$ were decided from the illumination, set $x_i \prec x_j$ if $\HD\left(x_i;P\right) > \HD\left(x_j;P\right)$. For $x_1, \dots, x_m$ sampled randomly from a continuous distribution, we identify the ranks uniquely, with no ties, almost surely. The performance of these ranking procedures is demonstrated in Section~\ref{section:tie-breaking}.

Suppose now that we have an estimator $F_n$ of $F$ from $P = EC\left(\mu,\Sigma,F\right)$ at hand. It could be obtained, for instance, by first performing a robust whitening transformation of the random sample $X_1, \dots, X_n$ from $P$, i.e. considering $Z_i = \widehat{\Sigma}^{-1/2}\left(X_i - \widehat{\mu}\right)$, $i=1, \dots, n$, for some robust location and scatter estimators $\widehat{\mu}$ and $\widehat{\Sigma}$ (in accordance with our parametrization, $\left\vert \widehat{\Sigma} \right\vert = 1$). The estimators $\widehat{\mu}$ and $\widehat{\Sigma}$ could be, for instance, the halfspace median (the barycentre of the points that maximize $\HD\left(\cdot;P_n\right)$), and the matrix of unit determinant proportional to the halfspace scatter median matrix (the matrix that maximizes the scatter extension of the halfspace depth, see \cite{Paindaveine_VanBever2018}). In the second step, $F$ is estimated simply by the empirical distribution function $F_n$ of any univariate marginal distribution of $Z_1, \dots, Z_n$. Since $F_n$ estimates a univariate distribution function, it does not suffer from the curse of dimensionality, and can be expected to have decent theoretical properties. More involved estimators of $F$ such as that from \cite{Liebscher2005} can be employed as well.

Because $F$ is symmetric, we can assume that also $F_n$ possesses the symmetry property. From a possibly non-symmetric estimator $\widetilde{F}_n$ of $F$ this can be achieved by symmetrization: set $F_n$ to be the right continuous version of the function $t \mapsto \left(\widetilde{F}_n(t) + 1 - \widetilde{F}_n(-t)\right)/2$. This procedure improves the properties of the basic estimator $\widetilde{F}_n$ if $F$ is symmetric \cite{Schuster1973}.

Finally, assume that $F_n$ is non-decreasing and affine invariant, the latter meaning that $F_n$ based on $X_1, \dots, X_n$ is the same as $F_n$ constructed from $A X_1 + b, \dots, A X_n + b$ for any $A \in \R^{d \times d}$ non-singular and $b \in \R^d$. 
These conditions are natural in our setting and are all satisfied by most reasonable estimators.
 
\subsection{Estimation of the Mahalanobis distance}	\label{section:Mahalanobis distance}

From \eqref{Mahalanobis distance population} we see that the Mahalanobis distance $\dist\left(x,\mu\right)$ can be estimated directly from the illumination depth, given that an estimator of $F$ is at hand. Consider the estimator
	\begin{equation*}	
		 M_\alpha\left(x;P_n\right) = 	\begin{cases}
																		F_n^{-1}\left( 1 - \HD\left(x;P_n\right) \right) & \mbox{if }\HD\left(x;P_n\right) \geq \alpha, \\
																		F_n^{-1}\left(1 - \alpha\right) g_d^{-1}\left(\frac{\Ill\left(x;P_{n,\alpha}\right)}{\vol{P_{n,\alpha}}}\right) & \mbox{otherwise}.
											\end{cases}
	\end{equation*}
Note that $F_n$ needs to be known only in the central part of the distribution; in the more extreme regions, $M_\alpha$ is proportional to a known function of the illumination only. Several desirable properties of $M_\alpha\left(\cdot;P_n\right)$ are summarized in the following theorem.

\begin{theorem}	\label{theorem:M properties}
Let $P_n \in \Prob$ be the empirical measure of a random sample $X_1, \dots, X_n$ from $P\in\Prob$ that is not concentrated in a singleton, let $F_n^{-1}(1 - \alpha) > 0$ with $0<\alpha<1/2$, and let $x \in \R^d$.
	\begin{enumerate}[label=(\roman*), ref=(\roman*)]
	\item \label{M properties i} If $P = EC\left(\mu,\Sigma,F\right)$ and $F_n$ is a Fisher consistent estimator of $F$, then $M_\alpha\left(x;P_n\right)$ is a Fisher consistent estimator of the Mahalanobis distance $\dist(x,\mu)$.
	\item \label{M properties ii} $M_\alpha$ is affine invariant, i.e. $M_\alpha\left(A x + b; P_{A X + b, n}\right) = M_\alpha\left(x;P_n\right)$ for any non-singular matrix $A \in \R^{d \times d}$ and $b \in \R^d$, where $P_{A X + b, n}$ is the empirical measure of the transformed random sample $A X_1 + b, \dots, A X_n + b$. 
	\item \label{M properties iii} For any $\delta \geq F_n^{-1}\left(1 - \Pi(P_n)\right)$, the lower level set
		\begin{equation}	\label{level set of M}
		\left\{ x \in \R^d \colon M_\alpha\left(x;P_n\right) \leq \delta \right\} 
		\end{equation}
	is either the unique halfspace median of $P_n$, or a convex body. 
	\item \label{M properties iv} As $n\to\infty$, let the limiting addition breakdown point \eqref{breakdown point} of the estimator $T(P_n) = F_n^{-1}\left(1 - \alpha\right)$ with the metric $\distg(s,t) = \left\vert \log(s) - \log(t) \right\vert$ for $s,t>0$ be at least $\min\left\{\alpha,1/3\right\}$ almost surely. Then for any $\delta > \lim_{n\to\infty} F_n^{-1}\left( 1 - \Pi\left(P_n\right)\right)$, the limiting addition breakdown point of the level set \eqref{level set of M} with respect to the Hausdorff distance is $\min\left\{\alpha,1/3\right\}$ almost surely.
	\end{enumerate}
\end{theorem}

The choice of the metric in the breakdown point from part~\ref{M properties iv} is natural. For a non-degenerate symmetric distribution function $F$ and its symmetric estimator $F_n$, the quantile $F^{-1}\left(1 - \alpha\right)$ with $0 < \alpha < 1/2$ lies in the positive halfline, and a sequence of estimated quantiles that converges to zero is just as undesirable as that escaping to infinity. Note also that the condition on the breakdown point of the sample quantile is naturally satisfied for any reasonable estimator of $F$ --- already the (symmetrized) empirical distribution function of any univariate random sample obeys it. The additional condition on $\delta$ in part~\ref{M properties iv} guarantees that the level set \eqref{level set of M} is non-empty. If, for instance, $P = EC\left(\mu,\Sigma,F\right)$ is such that $F$ strictly increases in a neighbourhood of $0$, and $F_n$ is an estimator that is strongly uniformly consistent on this neighbourhood, this condition reduces to $\delta > 0$. 

In the following theorem we study the uniform consistency of our robust estimator of the Mahalanobis distance.	

\begin{theorem}	\label{theorem:M consistency}
Let $P = EC\left(\mu,\Sigma,F\right)\in\Prob$ be such that $\Sigma$ is positive definite, $0<\alpha<1/2$, $F$ is continuous at $0$ and strictly increasing on $[0,F^{-1}\left( 1 - \alpha\right)]$, it has a density that is bounded from below in a neighbourhood of $F^{-1}\left(1 - \alpha\right)$, and let $F_n$ satisfy 
	\begin{equation}	\label{F consistency}
	\sup_{t \in \R} \left\vert F_n(t) - F(t) \right\vert \xrightarrow[n\to\infty]{\as}0.	
	\end{equation}
 and 
	\begin{equation}	\label{quantile rate}
	F_n^{-1}\left(1 - \alpha\right) - F^{-1}\left(1 - \alpha\right) = o_{\PP}\left(1/R_n\right),	
	\end{equation}
where $R_n = o\left( n^{1/(2 (d - 1))} \right)$. 	Let $K_n$ be a sequence of sets with $K_n \subset \B\left(\mu,R_n\right)$. Then 
	\[	\sup_{x \in K_n} \left\vert M_\alpha(x;P_n) - \dist(x,\mu) \right\vert \xrightarrow[n\to\infty]{\PP}0.	\]
\end{theorem}

The proof of Theorem~\ref{theorem:M consistency} is technical and can be found in the appendix.

\subsection{Estimation of the halfspace depth}

The close relation of the illumination depth with the Mahalanobis distance will now be used in a robustified definition of the sample halfspace depth based on the idea of illumination. Recall the connections of the halfspace depth with the Mahalanobis distance from \eqref{elliptical depth}. For $\alpha \in [0,1/2)$ we propose the following estimator of the depth $\HD\left(\cdot;P\right)$ of $P = EC(\mu,\Sigma,F)$
	\begin{equation}	\label{refined depth}
		  \RHD\left(x;P_n\right) = \begin{cases}
														\HD\left(x;P_n\right) & \mbox{if }\HD\left(x;P_n\right) \geq \alpha, \\
														F_n\left(g_d^{-1}\left(\frac{\Ill\left(x;P_{n,\alpha}\right)}{\vol{P_{n,\alpha}}}\right) F_n^{-1}\left(\alpha\right)\right) & \mbox{otherwise}.
														\end{cases}	
	\end{equation}
As for the illumination, there are several natural choices of the cut-off $\alpha$. Our main focus is in the robust estimation of the halfspace depth. Thus, we consider mainly constant $\alpha$. In spaces of lower dimensions, $\alpha = 1/3$ guarantees decent stability in combination with affine invariance, and superb robustness properties; details analogous to Theorem~\ref{theorem:M properties} are omitted.
	
\begin{theorem}	\label{theorem:F consistency}
Under the assumptions of Theorem~\ref{theorem:M consistency}
	\[	\sup_{x \in K_n} \left\vert \RHD(x;P_n) - \HD\left(x;P\right) \right\vert \xrightarrow[n\to\infty]{\PP}0.	\]
\end{theorem}

Theorem~\ref{theorem:F consistency} asserts that the robustified depth is a uniformly consistent estimator of the true halfspace depth. Following \cite{Einmahl_etal2015, He_Einmahl2017}, for an estimator of $F_n$ that performs well also in the tails of the distribution, we are able to derive a multiplicative version of the uniform consistency result. By \cite[Remark~1]{Einmahl_etal2015}, this result is much stronger than Theorem~\ref{theorem:F consistency}, and is valuable especially when extreme depth-regions are to be estimated.

\begin{theorem}	\label{theorem:F multiplicative consistency}
Suppose that the assumptions of Theorem~\ref{theorem:M consistency} are satisfied. In addition, let
	\begin{equation}	\label{quantile rate2}
	F_n^{-1}\left(1 - \alpha\right) - F^{-1}\left(1 - \alpha\right) = \mathcal O_{\PP}\left(\xi_n\right) = o_P\left(1/R_n\right),	
	\end{equation}
for a sequence $\xi_n$, and let
	\begin{equation}	\label{F multiplicative consistency}
	\sup_{\left\vert t \right\vert < 2 R_n/\sqrt{\lambda}} \left\vert \frac{F_n(t)}{F(t)} - 1 \right\vert \xrightarrow[n\to\infty]{\PP}0	
	\end{equation}
hold true for $\lambda$ the smallest eigenvalue of $\Sigma$. For $b = F^{-1}\left(1 - \alpha\right)$ and any $c > 0$ let
	\begin{equation}	\label{F multiplicative continuity}
	\lim_{n\to\infty} \sup_{\substack{b/2 < \left\vert s \right\vert < R_n/\sqrt{\lambda} \\ \left\vert s - t \right\vert < c\, \max\left\{ R_n \xi_n, \omega_n\right\}}} \left\vert \frac{F(s)}{F(t)} - 1 \right\vert = 0
	\end{equation}
with $\omega_n = \left(R_n^{d-1}/\sqrt{n}\right)^{2/(d+1)}$. Then 
	\[	\sup_{ x \in K_n} \left\vert \frac{\RHD\left(x;P_n\right)}{\HD\left(x;P\right)} - 1 \right\vert \xrightarrow[n\to\infty]{\PP} 0.	\]
\end{theorem}

We conclude this section by giving several remarks on the assumptions of Theorems~\ref{theorem:M consistency}, \ref{theorem:F consistency}, and \ref{theorem:F multiplicative consistency}. 

	\begin{enumerate}[label=(R\arabic*), ref=(R\arabic*), start=5]
	\item \label{R5} The conditions on the convergence rates of the estimator of the quantile \eqref{quantile rate} and \eqref{quantile rate2} are not restrictive at all. Suppose, for instance, that the parametric convergence rate $F_n^{-1}\left(1 - \alpha\right) - F^{-1}\left(1 - \alpha\right) = \mathcal O_{\PP}\left(n^{-1/2}\right)$ is true. Then, for $d>1$, both conditions \eqref{quantile rate} and \eqref{quantile rate2} are satisfied provided already that $R_n = o\left(n^{1/(2(d-1))}\right)$, as required for the consistency of the illumination.
	\item Condition \eqref{F multiplicative consistency} is also not too stringent. It is satisfied by the refined estimator of the univariate distribution function $F$ studied in \cite[Section~2.1]{Einmahl_etal2015}. There, based on the extreme value theory, an estimator $F_n$ is constructed that, under appropriate assumptions, obeys
	\[	\sup_{\left\vert t \right\vert < F^{-1}\left(1 - \delta_n\right)} \left\vert \frac{F_n(t)}{F(t)} - 1 \right\vert \xrightarrow[n\to\infty]{\PP}0	\]
for an adequate sequence $\delta_n \to 0$ as $n\to\infty$. Given that the sequence $R_n$ in \eqref{F multiplicative consistency} does not grow too fast, i.e. that $R_n \leq \sqrt{\lambda} F^{-1}\left( 1 - \delta_n \right)/2$, we have that $\left\{ t \in \R \colon \left\vert t \right\vert < 2 R_n/\sqrt{\lambda} \right\} \subset \left\{ t \in \R \colon \left\vert t \right\vert < F^{-1}\left(1 - \delta_n\right) \right\}$,
and \eqref{F multiplicative consistency} is valid for the estimator of $F$ from \cite{Einmahl_etal2015}. 
	\item \label{R7} Condition \eqref{F multiplicative continuity} is valid for $R_n$ that increases slowly enough. Suppose, for instance, as in Remark~\ref{R5} that $\xi_n = n^{-1/2}$ in \eqref{quantile rate2}. Then $\max\left\{ R_n \xi_n, \omega_n \right\}$ from \eqref{F multiplicative continuity} reduces to $\omega_n$, and we may bound, with $w_F$ the minimal modulus of continuity of $F$,
	\[	\sup_{\substack{b/2 < \left\vert s \right\vert < R_n/\sqrt{\lambda} \\ \left\vert s - t \right\vert < c\, \omega_n}} \left\vert \frac{F(s)}{F(t)} - 1 \right\vert \leq \sup_{\substack{\left\vert s \right\vert < R_n/\sqrt{\lambda} \\ \left\vert s - t \right\vert < c\, \omega_n}} \frac{\left\vert F(s) - F(t) \right\vert}{F(t)} \leq \frac{w_F\left(c\, \omega_n\right)}{F\left(-R_n/\sqrt{\lambda}\right)}.	\]
If $F$ has a density bounded from above by a constant $M>0$, the mean value theorem gives $w_F\left(h\right) = \sup_{\left\vert s - t\right\vert < h} \left\vert F(s) - F(t) \right\vert \leq M h$. Therefore, for $d>1$, for \eqref{F multiplicative continuity} to be true it is enough that $\omega_n = \left(R_n^{d-1}/\sqrt{n}\right)^{2/(d+1)} = o\left(F\left(-R_n/\sqrt{\lambda}\right)\right)$. If $F(t)$ does not decrease with $t \to -\infty$ at a rate faster than $\left\vert t \right\vert^\gamma$ for some $\gamma<0$, it is not difficult to see that $R_n = o\left( n^{(2(d-1) - \gamma(d+1))^{-1}} \right)$ guarantees \eqref{F multiplicative continuity}. For $F$ with a lighter tail, polynomial rates of $R_n$ may not be sufficiently slow. For $F(t)$ not decreasing with $t\to-\infty$ at a rate faster than $e^{-\left\vert t \right\vert^\gamma}$ for $\gamma>0$, we need to take $R_n$ increasing slower than $\left( \left(\frac{1}{d+1} - \varepsilon\right) \log(n) \right)^{1/\gamma} \sqrt{\lambda}$ for some $\varepsilon > 0$ to get \eqref{F multiplicative continuity}. Likewise, for $F(t)$ not decreasing faster than $\exp\left(-e^{\left\vert t \right\vert^\gamma}\right)$ with $t \to -\infty$ and $\gamma>0$, we take $R_n$ slower than $\left( \log\left( \frac{1}{d+1} - \varepsilon\right) + \log \log n \right)^{1/\gamma} \sqrt{\lambda}$ for any $\varepsilon > 0$ small enough. Note, however, that these estimates are rough, and for particular distribution functions $F$ finer rates of $R_n$ can be deduced directly from \eqref{F multiplicative continuity}. Detailed proofs of these results can be found in Appendix~\ref{appendix:proofs}.
	\end{enumerate}
	
%
%

\section{Applications}	\label{section:applications}

\subsection{Tie-breaking} \label{section:tie-breaking}

For an illustration of the tie-breaking capability of the illumination we consider a sample of $n=1000$ observations from the spherically symmetric standard five-dimensional normal distribution $P$ ($d=5$). The correct ranking $R_c$ of the observations from the most central to the most extreme is given by the decreasing values of the (population) density of $P$. We compare it to the ranking $R_{\HD}$ based on the decreasing values of the sample halfspace depth of the observations, where many ties occur, and the improved ranking $R_{\Ill}$ where the ties are resolved based on the value of $\Ill(x;P_{n,\alpha_n(x)})$, larger values corresponding to more extreme observations, see Section~\ref{section:elliptical distributions}. The cut-off $\alpha_n(x)$ is chosen so that $P_{n,\alpha_n(x)}$ contains one half of observations whose halfspace depth is not smaller than that of $x$.

In Table~\ref{tab:tiebreaking} we report the means and standard deviations of the estimated (Spearman) correlation coefficient between $R_{\HD}, R_c$ and $R_{\Ill}, R_c$, respectively, for observations with $\HD(x;P_n) \leq \delta$ for several values of $\delta$. The reported results are based on $100$ replications of the experiment.

From the $1000$ observations, on average $107$ lies on the boundary of the convex hull of the data and the halfspace depth alone cannot rank them properly. The refined ranking based on illumination is quite successful in this case, see the last row of Table~\ref{tab:tiebreaking} and Figure~\ref{fig:ties}.

\begin{table}[htpb]
\begin{tabular}{l|c|c|c}
                   & observations & $\mathrm{cor}_S(R_{\HD},R_c)$ & $\mathrm{cor}_S(R_{\Ill},R_c)$ \\ \hline
  $\delta = 0.5$   & 1000 {\footnotesize (0.)} & 0.989 {\footnotesize (0.002)} & 0.991 {\footnotesize (0.002)} \\
	$\delta = 0.05$  & 733 {\footnotesize (11.)}  & 0.975 {\footnotesize (0.004)} & 0.981 {\footnotesize (0.004)} \\
	$\delta = 0.01$  & 363 {\footnotesize (12.)}  & 0.895 {\footnotesize (0.016)} & 0.933 {\footnotesize (0.012)} \\
	$\delta = 0.005$ & 253 {\footnotesize (11.)}  & 0.806 {\footnotesize (0.029)} & 0.905 {\footnotesize (0.017)} \\
	$\delta = 0.001$ & 107 {\footnotesize (10.)}  & ---   & 0.923 {\footnotesize (0.023)}
\end{tabular}
\caption{For different values of $\delta$ the table shows the mean and standard deviation (in brackets) of the number of observations (out of $1000$) with $\HD(x;P_n) \leq \delta$, and for these observations the mean and standard deviation (in brackets) of the estimated (Spearman) correlation coefficient between their correct ranking $R_c$ and the rankings based on the halfspace depth $R_{\HD}$ and the improved ranking $R_{\Ill}$ based on the illumination, respectively. The last row corresponds to the observations lying on the boundary of the convex hull of the data --- these all have the same halfspace depth, the same $R_{\HD}$ rank, and hence $\mathrm{cor}_S(R_{\HD},R_c)$ is not defined. Based on $100$ replications of the experiment.}
\label{tab:tiebreaking}
\end{table}

\begin{figure}[htpb]
\includegraphics[width=0.7\textwidth]{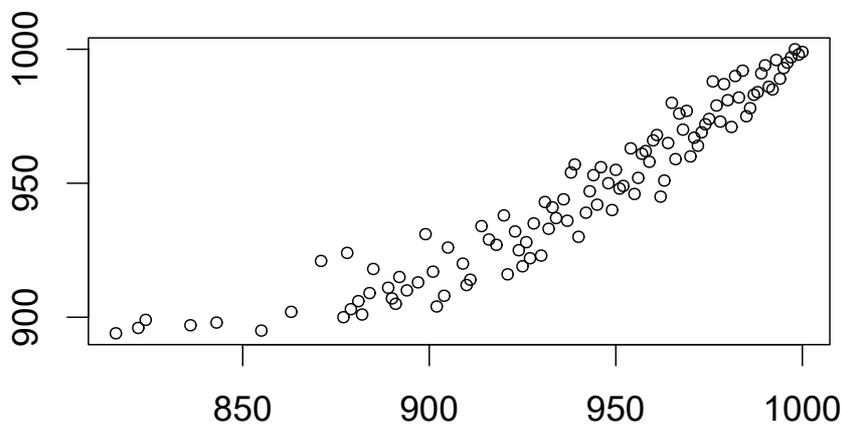}
\caption{Tie-breaking capability of illumination: improved ranks $R_{\Ill}$, based on illumination, of the observations lying on the boundary of the convex hull of the data (vertical axis) plotted against their correct ranks $R_{c}$, based on the values of the probability density function (horizontal axis). The halfspace depth of all these observations is zero.}
\label{fig:ties}
\end{figure}

\subsection{Estimation of extreme central regions}	\label{section:quantile estimation}

Consider the problem of estimation of the region $P_\delta$ for very small values of $\delta$, based on the sample of size $n$ from distribution $P$. The approach described in \cite{Einmahl_etal2015} is based on the so-called refined halfspace depth and finding the set $S=P_{n,k/n}$ for an appropriate value of $k \in \N$. The region $P_\delta$ is then estimated by the inflated set $E_R = c \,S = \left\{ c\,x \colon x \in S \right\}$, where $c=\left( \frac{k}{n \delta} \right)^{1/ \hat{\alpha}}$, and $\hat{\alpha}$ is the estimated tail index of $P$. Note the implicit assumptions of homothety of the depth contours, and that of the halfspace median of $P$ being the origin.

Assume now that $S=P_{n,k/n}$ is an ellipsoid given by $\left\{ x \in \R^d \colon \dist\left(x,0\right)\leq 1\right\}$ (which is a relevant approximation for elliptical distributions). Let $x$ be a point on the boundary of $S$ and $x^*= c \, x$ a point on the boundary of $c \, S$. It holds  that $\dist(x^*,0) = c \, \dist(x,0) = c$ and using \eqref{MD and illumination} we get $g_d(c) = g_d(\dist(x^*,0)) = \Ill(x^*;S)/\vol{S}$.

Instead of the inflation of the set $S$, our approach is based on finding the set $E_{\Ill} = \{ x \in \mathbb{R}^d: \Ill(x;S) \leq g_d(c) \}$. This is related to the approach from \cite{Einmahl_etal2015} but more robust in the sense that our procedure is less sensitive to errors in estimation of $S=P_{n,k/n}$ and the tail index $\alpha$. Figure~\ref{fig:quantile region} shows the estimates of the central region $P_{1/n}$ based on a sample of size $n=500$ from the spherically symmetric bivariate Cauchy distribution where, in agreement with \cite{Einmahl_etal2015}, we take $k=75$ for finding $S=P_{n,k/n}$ and the tail index $\alpha$.

We repeated the experiment 100 times and computed the two Hausdorff distances $\Haus(E_{\Ill},P_{1/n})$ and $\Haus(E_R,P_{1/n})$, respectively. The boxplots of these distances are given in Figure~\ref{fig:quantile boxplot}. We remark that in all 100 replications of the experiment we observed $\Haus(E_{\Ill},P_{1/n}) < \Haus(E_R,P_{1/n})$, making the illumination-based approach more successful in the estimation of $P_{1/n}$. This appears to be justified by a more spherical shape of $E_{\Ill}$, obtained by the illumination of $S$, compared to $E_R$, obtained by the inflation of $S$.

\begin{figure}[htpb]
\includegraphics[width=0.7\textwidth]{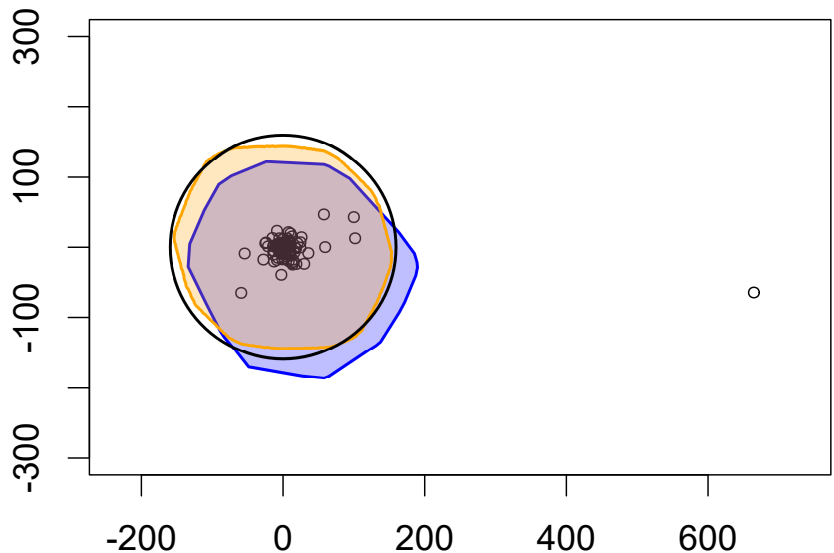}
\caption{Estimation of an extreme central region. Based on a sample of size $n=500$ from the spherically symmetric bivariate Cauchy distribution (note that a single observation with coordinates $(-1677, -1691)\tr$ is not plotted). Black circle: boundary of the true region $P_{1/n}$; light blue region: estimate $E_R$ of $P_{1/n}$ based on the refined depth \cite{Einmahl_etal2015}; light orange region: estimate $E_\Ill$ of $P_{1/n}$ based on the illumination.}
\label{fig:quantile region}
\end{figure}

\begin{figure}[htpb]
\includegraphics[width=0.7\textwidth]{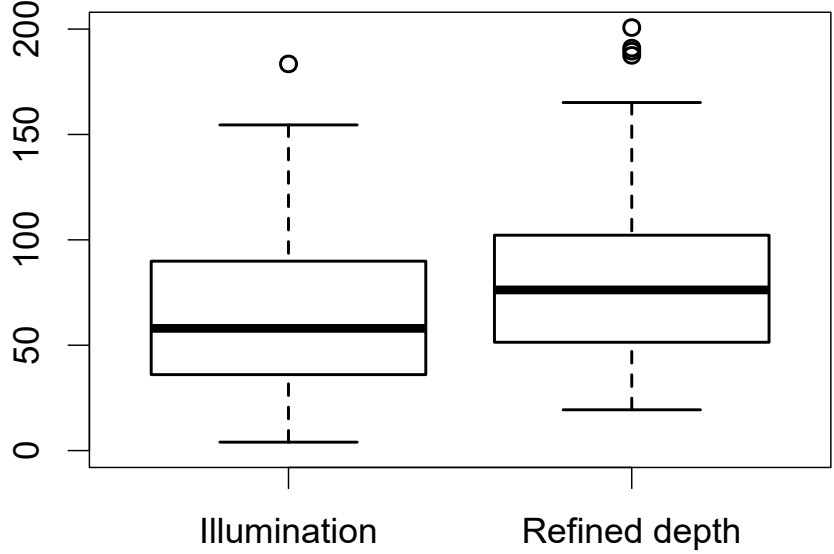}
\caption{Estimation of the extreme central region. Hausdorff distances of the estimates $E_{\Ill}$ (based on illumination) and $E_R$ (based on the refined depth \cite{Einmahl_etal2015}) from $P_{1/n}$.}
\label{fig:quantile boxplot}
\end{figure}

\subsection{Robust classification} 

Illumination can be used to devise a robust and affine invariant version of the quadratic discriminant analysis (QDA) classification rule, whose population version is optimal. Suppose, for simplicity, that we have two independent $d$-variate random samples from normal distributions $P^{(j)}$ with unknown mean vectors $\mu_j$ and unknown variance matrices $\Sigma_j$ for $j=1,2$, respectively. A new observation $x$ is sampled from $P^{(j)}$ with a known probability $\pi_j$, $\pi_1 + \pi_2 = 1$. Our task is to determine from which distribution $x$ was sampled. For the illumination-based QDA, let $0 < \delta < 1/2$ be a fixed parameter. We suggest to assign $x$ into $P^{(1)}$ if and only if
	\begin{equation}	\label{classification rule}
	2 \log\left(\frac{\pi_1}{\vol{P_\delta^{(1)}}}\right) - \dist[\Sigma_1]\left(x,\mu_1\right)^2 > 2 \log\left(\frac{\pi_2}{\vol{P_\delta^{(2)}}}\right) - \dist[\Sigma_2]\left(x,\mu_2\right)^2.	
	\end{equation}
In the population case, this simple classification rule is equivalent with the QDA, i.e. it is optimal in our setting. At the same time, it is affine invariant, highly robust for $\delta$ large enough, and entirely depth-based, as we saw in Section~\ref{section:Mahalanobis distance} that $\dist[\Sigma_j](x,\mu_j)$ can be consistently estimated by $M_\delta\left(x;P_n^{(j)}\right)$ with $F_n = \Phi$ the distribution function of the univariate standard normal variable, and $\vol{P_{n,\delta}^{(j)}}$ almost surely approaches $\vol{P_{\delta}^{(j)}}$ as $n\to\infty$. The proof of the following result can be found in Appendix~\ref{appendix:proofs}.
	
\begin{theorem}	\label{theorem:QDA}
For $P^{(j)}$ as above the illumination-based QDA classification rule \eqref{classification rule} is optimal, i.e. for any $\delta \in (0,1/2)$ it coincides with the classical quadratic discriminant rule. Furthermore, for any $K \subset \R^d$ bounded and $j=1,2$
	\[	\sup_{x \in K} \left\vert 2 \log\left(\frac{\pi_j}{\vol{P_{n,\delta}^{(j)}}}\right) - M_\delta\left(x;P_n^{(j)}\right)^2 - \left( 2 \log\left(\frac{\pi_j}{\vol{P_\delta^{(j)}}}\right) - \dist[\Sigma_j]\left(x,\mu_j\right)^2 \right) \right\vert \xrightarrow[n\to\infty]{\PP} 0,	\]
where $\Phi$ is used in place of $F_n$ in $M_\delta\left(x;P_n^{(j)}\right)$.
\end{theorem}
	
Note that in Theorem~\ref{theorem:QDA} we deal with normal distributions. For other elliptically symmetric distributions $EC\left(\mu_j,\Sigma_j,F\right)$ analogous results are straightforward to derive; for details see Appendix~\ref{appendix:classification}.

To illustrate the performance of the robust QDA classification approach we consider two simulation experiments. Another two scenarios are given in Appendix~\ref{appendix:simulations}.

\subsubsection{Bivariate normal distribution, location and scale difference}\label{subsubsec:normal location scale}

Let $P_X$ be the standard bivariate normal distribution and denote $P^{(1)} = P_X, P^{(2)} = P_{2X+(4,4)\tr}$. The training sets consist of $500$ observations from $P^{(1)}$ and $P^{(2)}$, respectively; the testing sets consist of $1000$ points sampled from $P^{(1)}$ and $P^{(2)}$, respectively. We consider the illumination-based QDA procedure given above and compare it to the classical QDA method and the method based on the refined halfspace depth \cite{Einmahl_etal2015}.

For the illumination we choose $\delta$ so that the probability content of $P_{X,\delta}$ is $0.5$. This results in $\delta = 1 - \Phi(\sqrt{2 \log 2})$. For computing the refined depth we take $k=75$ in agreement with \cite{Einmahl_etal2015}.

The experiment was repeated $100$ times. Figure~\ref{fig:classification normal location scale} (left panel) shows boxplots of the misclassification rates for different methods. The illumination-based approach and the classical QDA on average achieve the optimal (Bayes) error rate while the misclassification rates of the method based on the refined depth tend to be slightly higher.

\begin{figure}[htpb]
\includegraphics[width=0.45\textwidth]{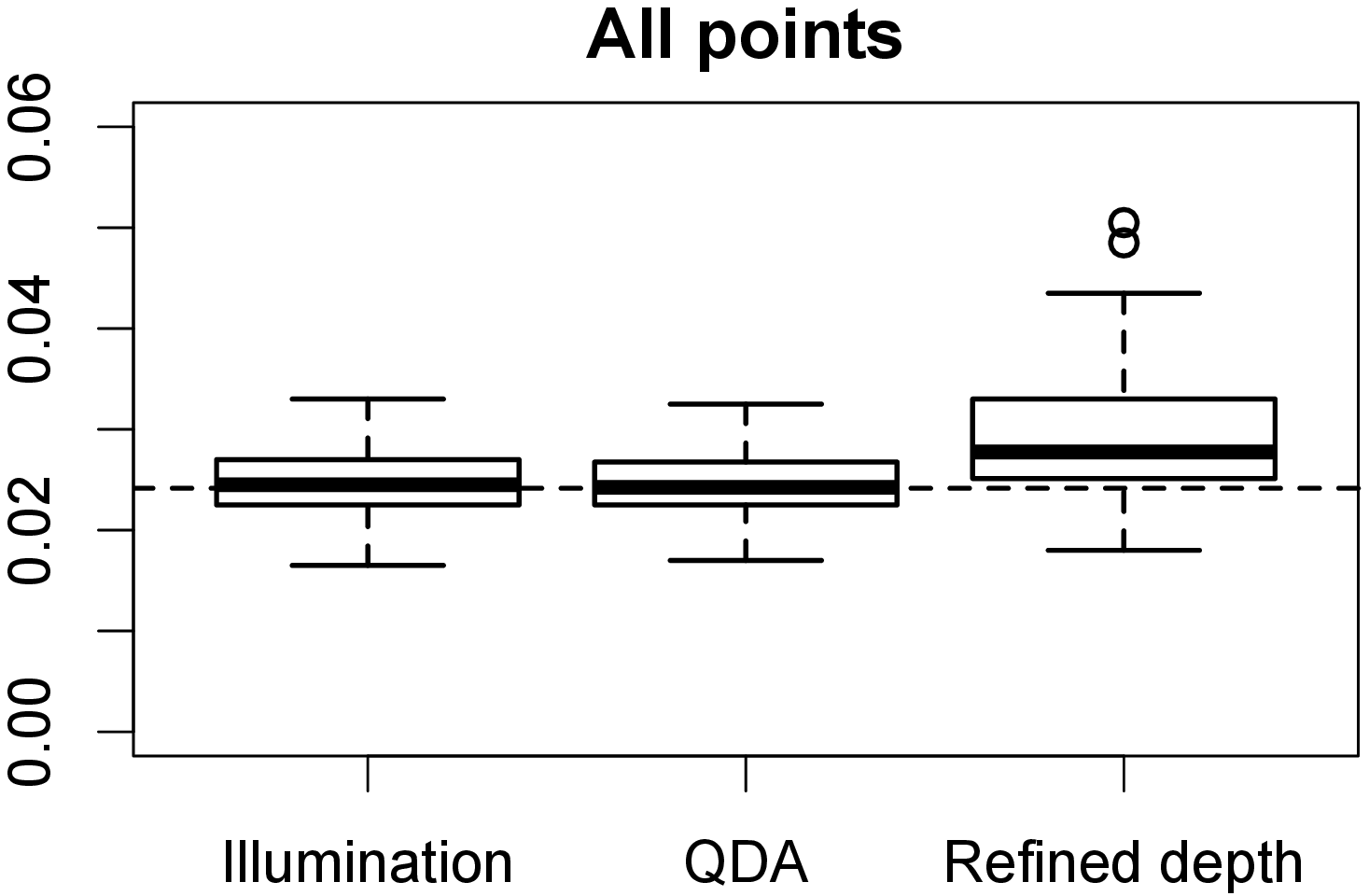}
\includegraphics[width=0.45\textwidth]{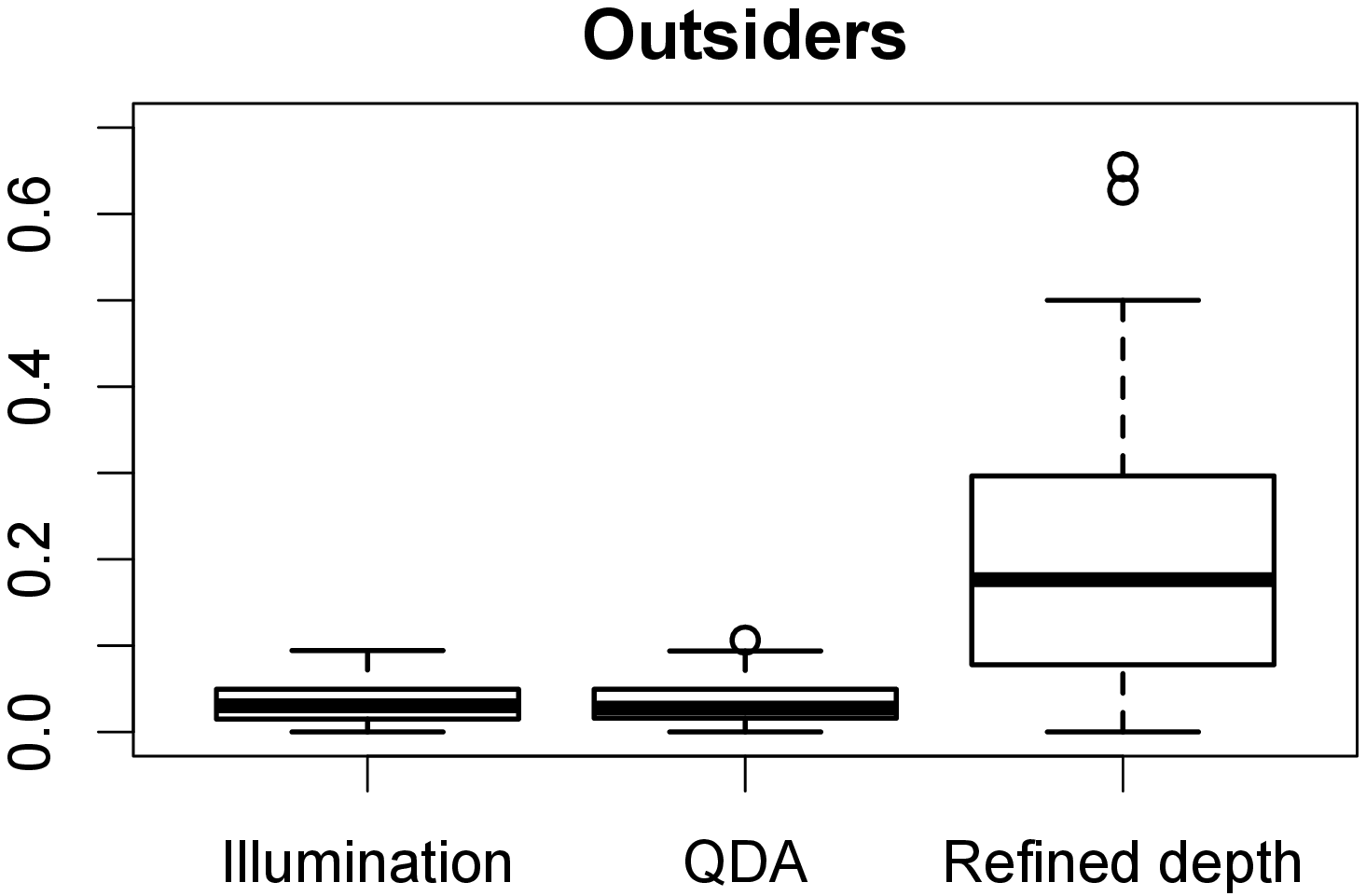}
\caption{Misclassification rates, based on 100 replications of the experiment with two bivariate normal distributions with different location and different scale. Based on all testing points (left panel) or the outsiders (right panel). Dashed horizontal line in the left panel corresponds to the theoretical Bayes error rate.}
\label{fig:classification normal location scale}
\end{figure}

To assess the performance of the classification methods in the extremes we consider another experiment where first $2500$ testing points are generated from each distribution, but only those outside the convex hull of both training sets are used for classification. This setup corresponds to the so-called outsider problem studied, among others, in \cite{Lange_etal2014}. The outsiders have all zero empirical halfspace depth w.r.t. both training sets, and hence cannot be classified based on $\HD$ only. Figure~\ref{fig:classification normal location scale} (right panel) shows boxplots of the misclassification rates of the considered methods for the outsiders. Both illumination-based approach and QDA still perform well. The method based on the refined depth suffers from much higher misclassification rates because of the incorrect estimation of the tail index.

\begin{table}[htpb]
\resizebox{\textwidth}{!}{
\begin{tabular}{c|c|c|c||c|c|c}
    & \multicolumn{3}{c||}{All points} & \multicolumn{3}{c}{Outsiders} \\ 
    & Illumination & QDA & Ref. depth & Illumination & QDA & Ref. depth \\ \hline
	0 \%   & 0.025 {\footnotesize (0.003)} & 0.024 {\footnotesize (0.003)} & 0.029 {\footnotesize (0.006)} & 0.034 {\footnotesize (0.023)} & 0.033 {\footnotesize (0.024)} & 0.202 {\footnotesize (0.148)} \\
  1 \%   & 0.025 {\footnotesize (0.003)} & 0.050 {\footnotesize (0.005)} & 0.045 {\footnotesize (0.010)} & 0.035 {\footnotesize (0.029)} & 0.055 {\footnotesize (0.038)} & 0.506 {\footnotesize (0.155)} \\
	5 \%   & 0.026 {\footnotesize (0.003)} & 0.044 {\footnotesize (0.005)} & 0.052 {\footnotesize (0.011)} & 0.033 {\footnotesize (0.026)} & 0.081 {\footnotesize (0.052)} & 0.539 {\footnotesize (0.146)} \\
	10 \%  & 0.034 {\footnotesize (0.004)} & 0.047 {\footnotesize (0.008)} & 0.059 {\footnotesize (0.059)} & 0.036 {\footnotesize (0.029)} & 0.092 {\footnotesize (0.058)} & 0.541 {\footnotesize (0.139)} \\ \hline
\end{tabular}}
\caption{Average misclassification rates and their standard deviations (in brackets), bivariate normal distributions with different location and different scale, level of contamination in one of the training samples ranging from 0 to 10~\%. Based on 100 replications of the experiment and all testing points (left part) and the outsiders (right part), respectively.}
\label{tab:classificationNormalContaminationLocationScale}
\end{table}

To study the robustness properties of the classification methods we consider contamination of the first training set by observations from $P^{(3)} = P_{X+(40,40)\tr}$. We set the extent of contamination to 1~\%, 5~\% and 10~\% of the data points, respectively. Table~\ref{tab:classificationNormalContaminationLocationScale} gives the average misclassification rates in these settings. Note that in the right part of the table different rows are not directly comparable since higher contamination implies larger span of the training points, hence on average fewer (more outlying) test points lie outside the convex hull of the training points. We observe that the illumination-based approach is very robust and performs well even under rather severe contamination. In contrast, the refined depth \cite{Einmahl_etal2015} is sensitive to the contamination through the estimation of the tail index.


\subsubsection{Bivariate elliptical distribution, location and scale difference}\label{subsubsec:elliptical location scale}

To study the classification performances for a heavy-tailed distribution, let $P_Y$ be the elliptical distribution from \cite{He_Einmahl2017} with the probability density function $f(x,y) = \frac{3(x^2/4+y^2)^2}{4\pi(1+(x^2/4+y^2)^3)^{3/2}}, (x,y)\tr \in \R^2$. Let $P^{(1)} = P_Y$, $P^{(2)} = P_{2Y+(4,4)\tr}$, $P^{(3)} = P_{X+(40,40)\tr}$ for $X$ from Section~\ref{subsubsec:normal location scale}. It is possible to adapt our robust QDA procedure by replacing $\Phi$ with the marginal distribution function $F$ of the first element of the spherically symmetric affine image of $Y$. The function $F$ is assumed to be known\footnote{To avoid lengthy numerical computations of multiple quantiles of $F$, here and in Section~\ref{subsubsec:elliptical location} in the appendix we slightly simplify the rule \eqref{classification rule}, and for $x \in P_{n,\delta}^{(1)} \cup P_{n,\delta}^{(2)}$ we assign $x$ to $P^{(1)}$ if and only if $\HD\left(x;P^{(1)}_n\right) > \HD\left(x;P^{(2)}_n\right)$.}. Here we choose $\delta \doteq 0.11205$ so that the probability content of $P_{Y,\delta}$ is $0.5$.

The results are summarized in Figure~\ref{fig:classification elliptical location scale} and Table~\ref{tab:classificationEllipticalContaminationLocationScale}. We observe that in the case with no contamination, the illumination-based approach performs the best from the three classification methods considered. Note that the very high misclassification rates for the method based on the refined depth and the group of outsiders under contamination (see the right part of Table~\ref{tab:classificationEllipticalContaminationLocationScale}), are caused by the poorly estimated tail index of the distribution $P^{(1)}$, and because approximately three times more testing points come from the distribution $P^{(2)}$ than from $P^{(1)}$ due to the greater spread of $P^{(2)}$.

\begin{figure}[tpb]
\includegraphics[width=0.45\textwidth]{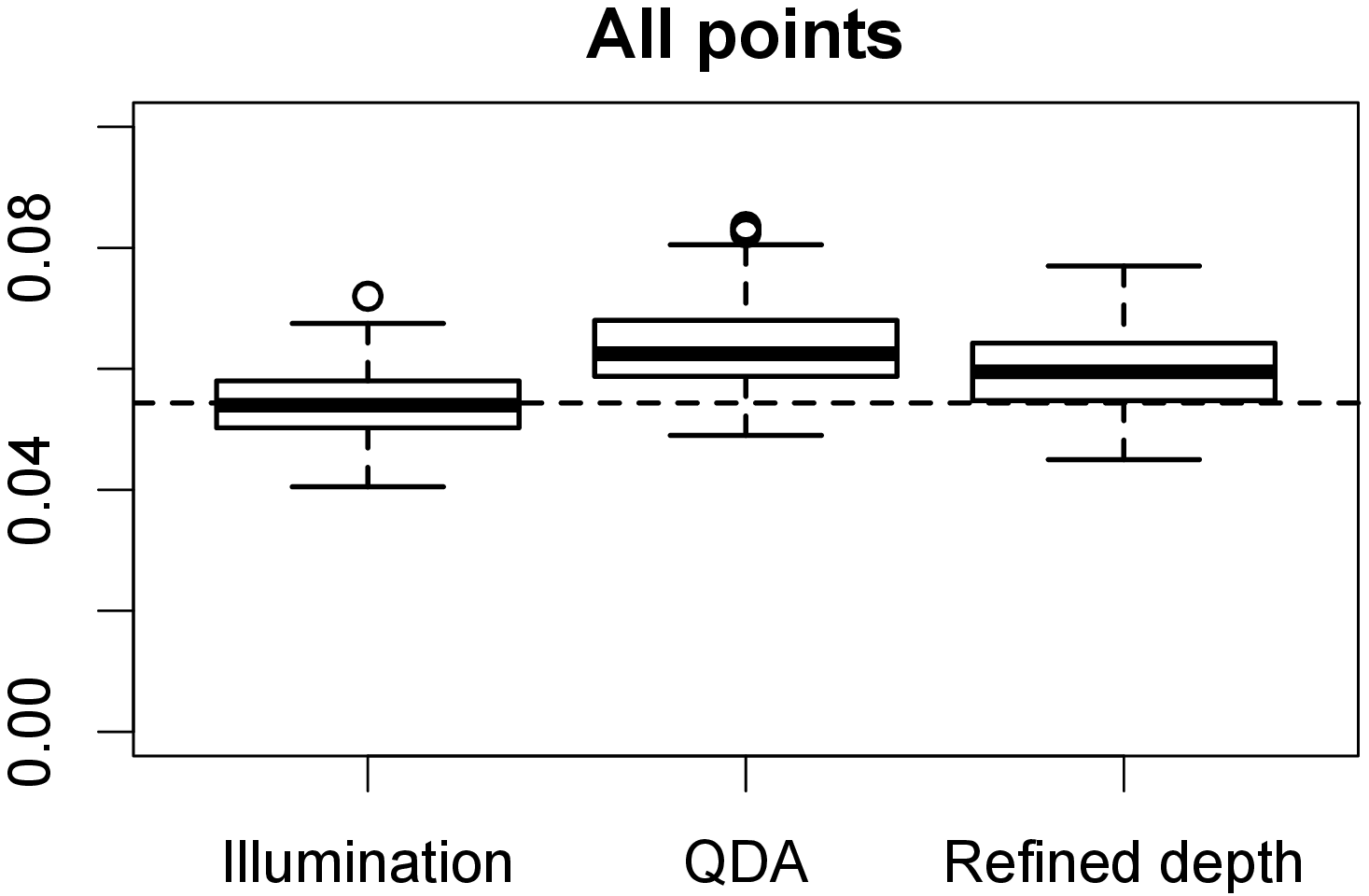}
\includegraphics[width=0.45\textwidth]{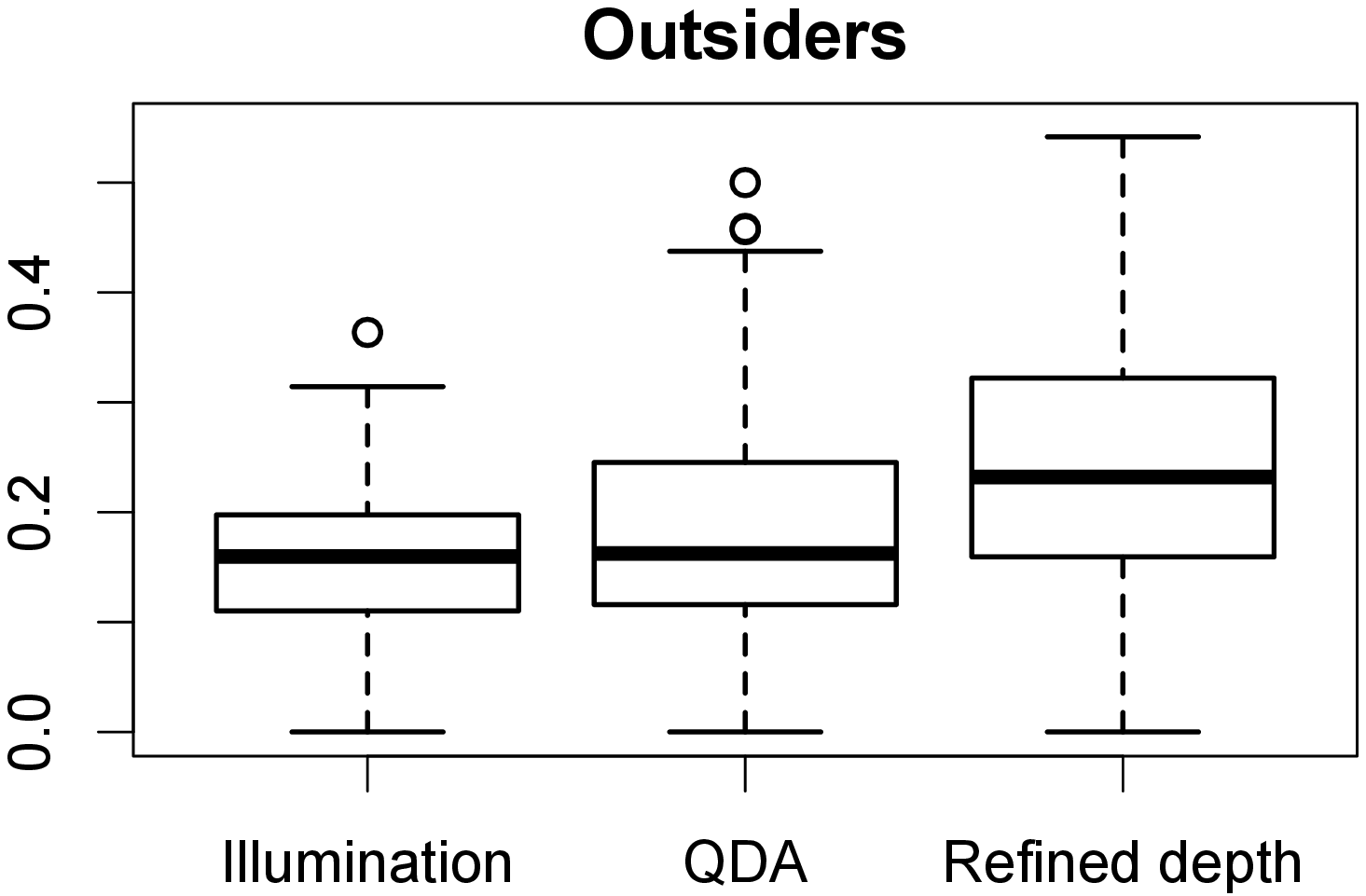}
\caption{Misclassification rates, based on $100$ replications of the experiment with two bivariate elliptical distributions with different location and different scale. Based on all testing points (left panel) or the outsiders (right panel). Dashed horizontal line in the left panel corresponds to the theoretical Bayes error rate.}
\label{fig:classification elliptical location scale}
\end{figure}

\begin{table}[tpb]
\resizebox{\textwidth}{!}{
\begin{tabular}{c|c|c|c||c|c|c}
    & \multicolumn{3}{c||}{All points} & \multicolumn{3}{c}{Outsiders} \\ 
    & Illumination & QDA & Ref. depth & Illumination & QDA & Ref. depth \\ \hline
	0 \%   & 0.055 {\footnotesize (0.006)} & 0.064 {\footnotesize (0.007)} & 0.060 {\footnotesize (0.007)} & 0.155 {\footnotesize (0.071)} & 0.184 {\footnotesize (0.105)} & 0.240 {\footnotesize (0.118)} \\
  1 \%   & 0.056 {\footnotesize (0.006)} & 0.096 {\footnotesize (0.014)} & 0.066 {\footnotesize (0.008)} & 0.179 {\footnotesize (0.081)} & 0.238 {\footnotesize (0.094)} & 0.527 {\footnotesize (0.183)} \\
	5 \%   & 0.057 {\footnotesize (0.005)} & 0.103 {\footnotesize (0.029)} & 0.093 {\footnotesize (0.016)} & 0.185 {\footnotesize (0.090)} & 0.239 {\footnotesize (0.102)} & 0.755 {\footnotesize (0.129)} \\
	10 \%  & 0.068 {\footnotesize (0.008)} & 0.154 {\footnotesize (0.038)} & 0.108 {\footnotesize (0.015)} & 0.202 {\footnotesize (0.084)} & 0.266 {\footnotesize (0.098)} & 0.742 {\footnotesize (0.128)} \\ \hline
\end{tabular}}
\caption{Misclassification rates and their standard deviations (in brackets), bivariate elliptical distributions with different location and different scale, level of contamination in one of the training samples ranging from $0$ to $10~\%$. Based on $100$ replications of the experiment and all testing points (left part) and the outsiders (right part), respectively.}
\label{tab:classificationEllipticalContaminationLocationScale}
\end{table}


Overall, our experiments demonstrate a great potential for varied illumination-based statistical methodology. All these results will be further elaborated on in the coming works of the authors.

%
%
%
%

\appendix

\section{Proofs of the theoretical results} \label{appendix:proofs}

\subsection{Proof of Lemma~\ref{lemma:ellipsoid}}

For $d=1$, $\Sigma = \sigma^2 > 0$ and $\left\vert x - \mu \right\vert > \sigma$, the formula reduces to $\Ill\left(x;\Ell\right) = \left\vert x - \mu \right\vert + \sigma$, which is the illumination of $x$ outside $\Ell = \sigma \B[1]$ on that ball. For $d>1$, let us first compute the illumination of a unit ball. Take $x \notin \B$. The set difference of the convex hull of $x$ and $\B$ minus $\B$ is a cone with height $\left\Vert x \right\Vert - 1/\left\Vert x \right\Vert$ and base a $(d-1)$-dimensional ball with radius $\sqrt{1 - 1/\left\Vert x \right\Vert^2}$, without a spherical cap of $\B$ of height $1-1/\left\Vert x \right\Vert$. Because $\vol{\B} = \frac{\pi^{d/2}}{\Gamma\left(\frac{d}{2}+1\right)}$, the volume of the cone is 
	\begin{equation}	\label{cone volume}
	\frac{1}{d} \frac{\pi^{\frac{d-1}{2}}}{\Gamma\left(\frac{d-1}{2}+1\right)} \left( 1 - 1/\left\Vert x \right\Vert^2 \right)^{\frac{d-1}{2}} \left( \left\Vert x \right\Vert - \frac{1}{\left\Vert x \right\Vert} \right),	
	\end{equation}
and the volume of the cap is 
	\begin{equation}	\label{cap volume}
	\frac{\pi^{\frac{d-1}{2}}}{\Gamma\left(\frac{d+1}{2}\right)} \int_0^{\arccos\left(1/\left\Vert x \right\Vert\right)} \sin^d (t) \dd t.	
	\end{equation}
Altogether, \eqref{cone volume} and \eqref{cap volume} give that
	\[	\Ill\left(x;\B\right) =  \frac{\pi^{\frac{d-1}{2}}}{\Gamma\left(\frac{d+1}{2}\right)} \left( \frac{\left\Vert x \right\Vert}{d} \left( 1 - 1/\left\Vert x \right\Vert^2 \right)^{\frac{d+1}{2}} - \int_0^{\arccos\left(1/\left\Vert x \right\Vert\right)} \sin^d (t) \dd t \right) + \vol{\B}.	\]
It is not difficult to see that $\Ell = \Sigma^{1/2} \B + \mu = \bigcup_{x \in \B} \left\{ \Sigma^{1/2} x + \mu \right\}$. Thus, by the affine equivariance of the illumination bodies \citep[Proposition~2]{Werner2006} we have $\Ill\left(x;\Ell\right) = \Ill\left(\Sigma^{-1/2}\left(x - \mu\right);\B\right) \sqrt{\left\vert \Sigma \right\vert}$. The general assertion then follows from $\left\Vert \Sigma^{-1/2}\left(x - \mu\right) \right\Vert = \sqrt{\left(x - \mu\right)\tr\Sigma^{-1}\left(x-\mu\right)}$.

\begin{figure}[htpb]
\includegraphics[width=\twowidth\textwidth]{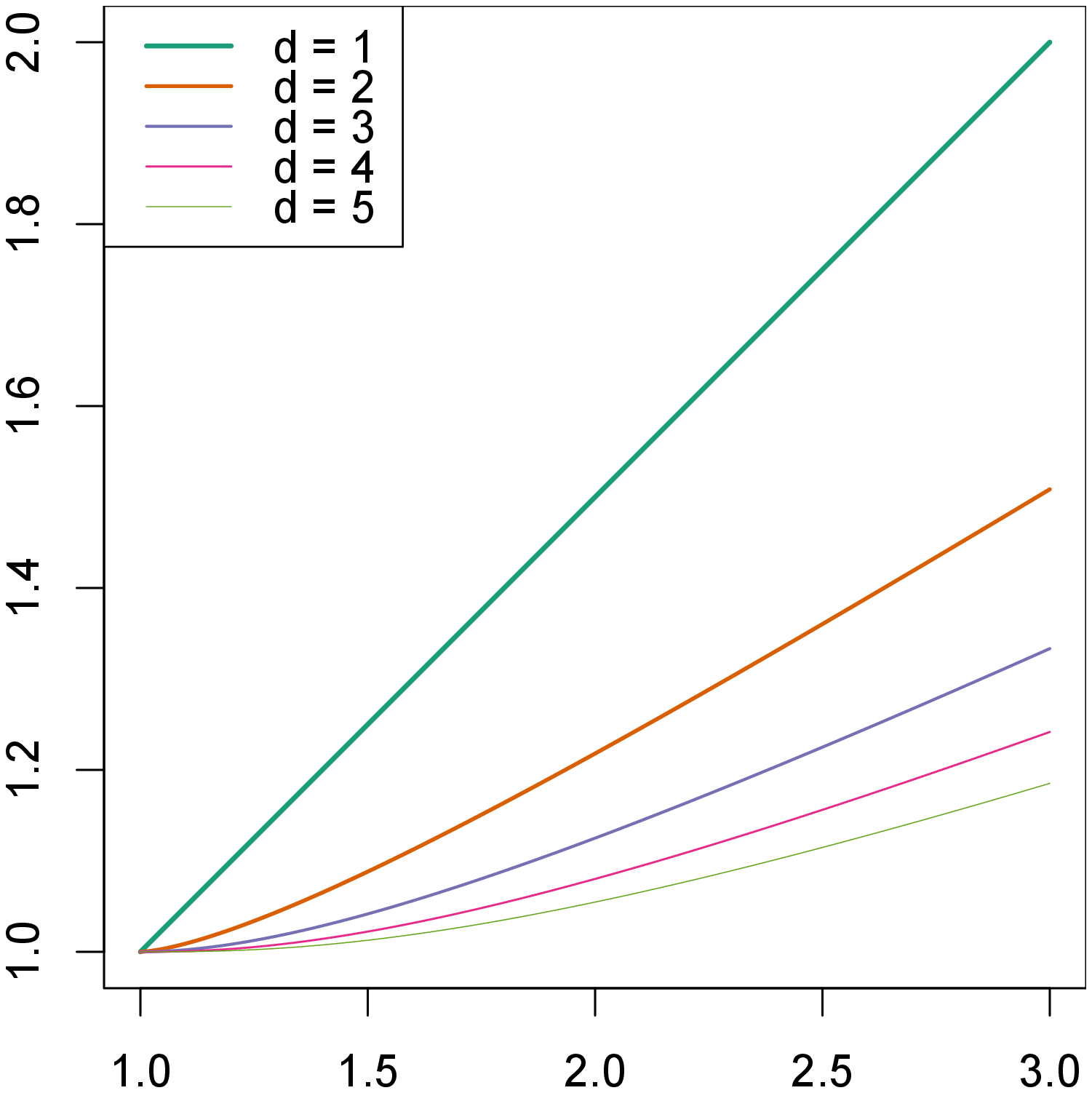} \quad \includegraphics[width=\twowidth\textwidth]{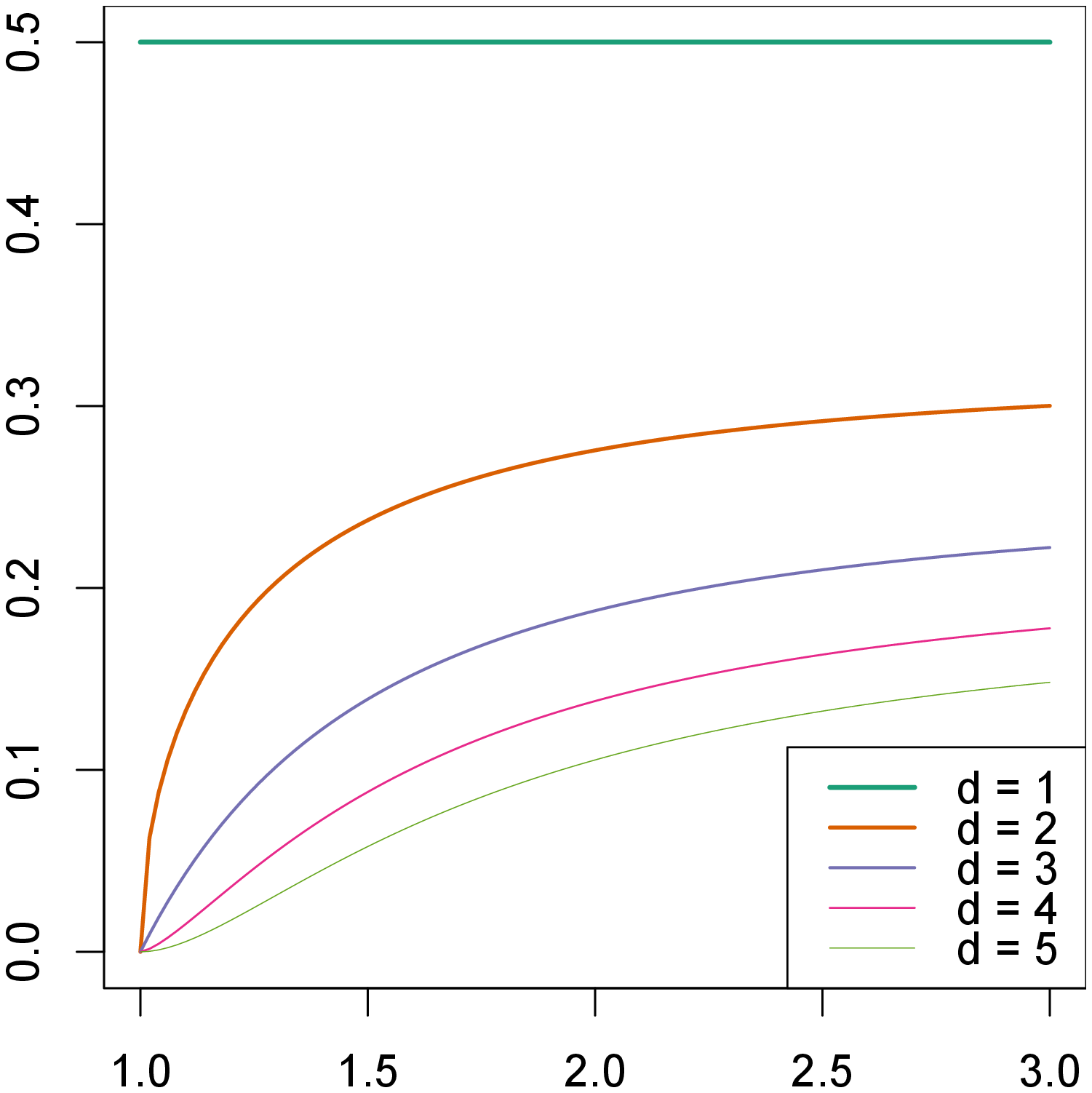}
\caption{First five functions $g_d$, $d=1,\dots,5$ (left panel) and their first derivatives (right panel).}
\end{figure}

\subsection{Lemma~\ref{lemma:g properties}}
The next lemma summarizes some analytical properties of the function $g_d$ defined in Section~\ref{section:illumination bodies}.

\begin{lemma} \label{lemma:g properties}
For all $d \geq 1$
	\begin{enumerate}[label=(\roman*), ref=(\roman*)]
	\item function $g_d \colon [1,\infty) \to [1,\infty)$ is uniformly continuous, strictly increasing, and convex;
	\item $g_d(1) = 1$, $\lim_{t \to \infty} g_d(t) = \infty$;
	\item $g_d$ is differentiable on $(1,\infty)$ and
		\begin{equation}	\label{g derivative}
		g_d^\prime(t) = \frac{\Gamma\left(\frac{d}{2}+1\right)}{\sqrt{\pi}\,\Gamma\left(\frac{d+1}{2}\right)} \frac{1}{d} \left(1 - \frac{1}{t^2}\right)^{(d-1)/2} \quad\mbox{for }t\in(1,\infty);	
		\end{equation}
	\item \label{O expression for g} $g_d(t) -1 = \mathcal O\left(\left(t-1\right)^{(d+1)/2}\right)$ as $t \to 1$ from the right;
	\item \label{modulus} the minimal modulus of continuity of the inverse function $g_d^{-1}$ takes the form
		\[	w_{g_d^{-1}}(h) = \sup_{\left\vert s - t \right\vert < h} \left\vert g_d^{-1}(s) - g_d^{-1}(t) \right\vert = g_d^{-1}(1 + h) - 1 \quad\mbox{for }h\geq 0;	\]
	\item 
		as $h \to 0$ from the right, $w_{g_d^{-1}}(h) = \mathcal O\left( h^{2/(d+1)} \right)$;
	\item as $t \to \infty$, $g_d^{-1}(t) = \mathcal O\left(t\right)$.
	\end{enumerate}
\end{lemma}

\begin{proof}
Using the Leibniz integral formula it is easy to see that the derivative of $g_d$ is \eqref{g derivative}. That function is positive, increasing, and bounded from above. Hence, $g_d$ is strictly increasing, convex, and Lipschitz continuous. Part~\ref{O expression for g} follows by an application of l'H\^{o}pital's rule
	\begin{equation}	\label{limit for g}
	\begin{aligned}
	\lim_{t \to 1} \frac{g_d(t) - 1}{\left(t - 1\right)^{(d+1)/2}} & = \lim_{t \to 1} \frac{\Gamma\left(\frac{d}{2}+1\right)}{\sqrt{\pi}\, \Gamma\left(\frac{d+1}{2}\right)} \frac{2}{d (d+1)} \frac{\left(\left( t - 1 \right) \left(t + 1\right)\right)^{(d-1)/2}}{t^{d-1} (t-1)^{(d-1)/2}} \\
	& = \frac{\Gamma\left(\frac{d}{2}+1\right)}{\sqrt{\pi}\, \Gamma\left(\frac{d+1}{2}\right)} \frac{2^{(d+1)/2}}{d(d+1)}.
	\end{aligned}
	\end{equation}
For Part~\ref{modulus} first note that because $g_d$ is smooth, strictly increasing and convex, its inverse $g_d^{-1}$ must be smooth, strictly increasing and concave. For such a function the mean value theorem asserts that the greatest difference $g_d^{-1}(s) - g_d^{-1}(t)$ subject to $1 \leq t \leq s < t+h$ must be attained at the left endpoint of its domain, i.e. for $t = 1$ and $s = 1+h$. To obtain the rate of the modulus of continuity, note that by \eqref{limit for g} there exists $c > 0$ such that
	\[	g_d(t) - 1 \geq c \left( t - 1 \right)^{(d+1)/2} \quad\mbox{for all $t>1$ close enough to $1$}.	\]
Apply $g_d^{-1}$ to both sides of this inequality and substitute $h = c \left(t - 1\right)^{(d+1)/2}$ to get
	\[	g_d^{-1}(1+h) - 1 \leq \left(\frac{h}{c}\right)^{2/(d+1)} \quad\mbox{for all $h>0$ small enough},	\]
and the conclusion follows. Finally, using substitution $t = g_d(s)$ and l'H\^{o}pital's rule again,
	\[	\lim_{t\to\infty} \frac{g_d^{-1}\left(t\right)}{t} = \lim_{s\to\infty} \frac{g_d^{-1}\left( g_d\left(s\right)\right)}{g_d\left(s\right)} = \lim_{s \to \infty} \frac{s}{g_d(s)} = \lim_{s \to \infty} \frac{1}{g_d^\prime\left(s\right)} = \frac{d\sqrt{\pi}\,\Gamma\left(\frac{d+1}{2}\right)}{\Gamma\left(\frac{d}{2}+1\right)}.	\]
Hence, $g_d^{-1}(t) = \mathcal O\left(t\right)$ as $t\to\infty$.
\end{proof}

%
%


\subsection{Proof of Theorem~\ref{theorem:affine invariance}}
We only prove the first part of the theorem. The remaining parts are straightforward, and follow directly from the essential properties of the halfspace depth \cite{Nagy_etal2018s}, and the properties of the illumination \cite{Werner2006}. 

By the affine invariance of the halfspace depth \cite[Lemma~2.1]{Donoho_Gasko1992} we know that $\left(P_{A X + b}\right)_\alpha = A (P_X)_\alpha + b$. For the illumination, it follows that
	\[
	\begin{aligned}
	\frac{\Ill\left(A x + b;\left(P_{A X + b}\right)_\alpha\right)}{\vol{\left(P_{A X + b}\right)_\alpha}} & =  \frac{\vol{\co{\left(A (P_X)_\alpha + b\right) \cup \left\{ A x + b \right\}}}}{\vol{A (P_X)_\alpha + b}} \\
	& = \frac{\vol{A \left( (P_X)_\alpha \cup \left\{ x \right\} \right) + b}}{\vol{A (P_X)_\alpha + b}} = \frac{\Ill\left(x;(P_X)_\alpha\right)}{\vol{(P_X)_\alpha}}.
	\end{aligned}
	\]	 

\subsection{Proof of Theorem~\ref{theorem:consistency}}

We start with the illumination. From \cite[Theorem~4.2]{Dyckerhoff2018} we know that under the assumptions of the theorem, the central regions $P_\alpha$ are consistent for $P$ in the Hausdorff distance, i.e. 
	\begin{equation}	\label{Dyckerhoff}
	\Haus\left(P_{n,\alpha},P_\alpha\right) \xrightarrow[n\to\infty]{\as}0.
	\end{equation}
For any $x \in K_n$ we know that almost surely for $n$ large
	\begin{equation}	\label{illumination bound}
	\begin{aligned}
	\left\vert \Ill\left(x;P_{n,\alpha}\right) - \Ill\left(x;P_\alpha\right) \right\vert & = \left\vert \vol{\co{P_{n,\alpha} \cup \left\{ x \right\}}} - \vol{\co{P_{\alpha} \cup \left\{ x \right\}}} \right\vert \\
	& \leq c_{d} \sum_{j = 0}^{d-1} \Haus\left(\co{P_{\alpha} \cup \left\{ x \right\}}, \co{P_{n,\alpha} \cup \left\{ x \right\}}\right)^{d-j} R_n^j \\
	& \leq c_{d} \sum_{j = 0}^{d-1} \Haus\left(P_{\alpha} \cup \left\{ x \right\}, P_{n,\alpha} \cup \left\{ x \right\}\right)^{d-j} R_n^j \\
	& \leq c_{d} \sum_{j = 0}^{d-1} \Haus\left(P_{\alpha}, P_{n,\alpha} \right)^{d-j} R_n^j \\
	& \leq d \, c_{d} \, \Haus\left(P_{\alpha}, P_{n,\alpha} \right) \max\left\{1,R_n^{d-1}\right\}.
	\end{aligned}
	\end{equation}
In the inequalities we used Lemma~\ref{lemma:Lipschitz} stated below for $\Haus\left(P_{n,\alpha},P_\alpha\right)<1$, and the properties of the Hausdorff distance \citep[p.~64]{Schneider2014}. Since for a fixed compact set $K = K_n$ for all $n$ the term $R_n$ is constant, the first part of the theorem is verified in view of \eqref{Dyckerhoff}. 

To derive the rates of convergence, by \cite[Theorem~2]{Brunel2019} we have that $\Haus\left(P_{n,\alpha}, P_\alpha\right) = \mathcal O_{\PP}\left( n^{-1/2} \right)$, and the last inequality in \eqref{illumination bound} is enough to conclude.

For the affine invariant version of the illumination, write
	\begin{equation}	\label{robust illumination bound}
	\begin{aligned}
	\left\vert \frac{\Ill\left(x;P_{n,\alpha}\right)}{\vol{P_{n,\alpha}}} - \frac{\Ill\left(x;P_\alpha\right)}{\vol{P_\alpha}} \right\vert & \leq \frac{\left\vert \Ill(x;P_{n,\alpha}) - \Ill(x;P_\alpha) \right\vert}{\vol{P_{n,\alpha}}} \\ 
	& \phantom{\leq} + \left\vert \Ill(x;P_\alpha) \right\vert \left\vert \frac{1}{\vol{P_{n,\alpha}}} - \frac{1}{\vol{P_\alpha}} \right\vert.	
	\end{aligned}
	\end{equation}
By the assumptions of the theorem we know that $\vol{P_\alpha}>0$. From \eqref{Dyckerhoff} and Lemma~\ref{lemma:Lipschitz} it thus follows that for $n$ large enough $\vol{P_{n,\alpha}} \geq \vol{P_\alpha}/2$ almost surely, and that for such $n$ it also holds true that
	\[	\left\vert \frac{1}{\vol{P_{n,\alpha}}} - \frac{1}{\vol{P_\alpha}} \right\vert \leq \frac{2 \, \left\vert \vol{P_{n,\alpha}} - \vol{P_\alpha} \right\vert}{\vol{P_\alpha}^2} \leq \frac{2\, d \, c_d \max\left\{1,R_1^{d-1}\right\} \, \Haus\left(P_\alpha, P_{n,\alpha}\right)}{\vol{P_\alpha}^2}	\]
almost surely, for $c_d > 0$ the constant from Lemma~\ref{lemma:Lipschitz}. By \cite[Theorem~2]{Brunel2019} the last formula can be written also as
	\[	\left\vert \frac{1}{\vol{P_{n,\alpha}}} - \frac{1}{\vol{P_\alpha}} \right\vert = \mathcal O_{\PP}\left(n^{-1/2}\right).	\]
Finally, because $P_\alpha$ is a fixed bounded set, a trivial upper bound for $\sup_{x \in K_n} \left\vert \Ill(x;P_\alpha) \right\vert$ is the maximum illumination of $x \in K_n$ w.r.t. the smallest enclosing ball of $P_\alpha$. By Lemmas~\ref{lemma:ellipsoid} and~\ref{lemma:g properties} this is of order $\mathcal O\left( R_n \right)$. Altogether, all the above bounds and the consistency result for $\Ill$ can be plugged into \eqref{robust illumination bound} to obtain the desired rate of convergence
	\[	
	\begin{aligned}
	\sup_{x \in K_n} \left\vert \frac{\Ill\left(x;P_{n,\alpha}\right)}{\vol{P_{n,\alpha}}} - \frac{\Ill\left(x;P_\alpha\right)}{\vol{P_\alpha}} \right\vert & = \mathcal O_{\PP}\left(\frac{\max\{1,R_n^{d-1}\}}{\sqrt{n}}\right) + \mathcal O\left(R_n\right) \mathcal O_{\PP}\left(\frac{1}{\sqrt{n}}\right) \\
	& = \mathcal O_{\PP}\left(\frac{\max\{1,R_n^{d-1}\}}{\sqrt{n}}\right).	
	\end{aligned}
	\]

\begin{lemma}	\label{lemma:Lipschitz}
Let $R>0$. There exists a constant $c_{d} > 0$ such that for all convex bodies $K, L \subset \R^d$ with $K \subset \B\left(x,R\right)$ for some $x \in \R^d$
	\[	\left\vert \vol{K} - \vol{L} \right\vert \leq c_{d} \sum_{j = 0}^{d-1} \Haus\left(K,L\right)^{d-j} R^j.	\]
\end{lemma}
	
\begin{proof}
Write $\delta = \Haus\left(K,L\right)$. From the definition of the Hausdorff distance \eqref{Hausdorff distance} we have that
	\begin{equation}	\label{inclusions}
	K \subset L + \delta \B \mbox{ and }L \subset K + \delta \B.	
	\end{equation}
If $\vol{K} \leq \vol{L}$, this gives $\vol{K} \leq \vol{L} \leq \vol{K + \delta \B}$; in the other case $\vol{L} < \vol{K}$ we get $\vol{L} < \vol{K} \leq \vol{L + \delta \B}$. This results in 
	\[	 \left\vert \vol{K} - \vol{L} \right\vert \leq \max\left\{ \vol{K + \delta \B} - \vol{K}, \vol{L + \delta \B} - \vol{L} \right\},	\]
and it is enough to bound the excess volume of the outer parallel body $K + \delta \B$ of a convex body $K$, and analogously for $L$. For this, use the Steiner formula \citep[Formula~(4.1)]{Schneider2014}
	\[	\vol{K + \delta \B} = \sum_{j=0}^d \delta^{d - j} \vol[d-j]{\B[d-j]} V_j(K),	\]
where $V_j(K)$ stands for the intrinsic volume of the convex body $K$ \citep[Chapter~4]{Schneider2014}. In particular, it holds true that $V_d(K) = \vol{K}$, $V_{d-1}(K)$ is proportional to the surface area measure of $K$, $V_1(K)$ is the so-called intrinsic width of $K$, and $V_0(K) = 1$. 

From the monotonicity of the intrinsic volumes that follows from formulas (5.25) and (5.31) in \cite{Schneider2014}, and $K \subset \B\left(x,R\right)$, we can use the expression for the intrinsic volumes of a ball~(4.64) from \cite{Schneider2014} and bound
	\begin{equation}	\label{K volume bound}	
	\begin{aligned}
	\vol{K + \delta \B} - \vol{K} & \leq \sum_{j=0}^{d-1} \delta^{d - j} \vol[d-j]{\B[d-j]} V_j(\B) \\
	& = \vol{\B} \sum_{j=0}^{d-1} \delta^{d - j} \binom{d}{j} R^j.	
	\end{aligned}
	\end{equation}
	
For a bound on the excess volume of $L + \delta \B$, first note that from \eqref{inclusions} we have
	\[	L \subset K + \delta \B \subset \B\left(x,R\right) + \B\left(0,\delta\right) = \B\left(x,R+\delta\right).	\]
Similarly as in \eqref{K volume bound} we can thus write
	\[	
	\begin{aligned}
	\vol{L + \delta \B} - \vol{L} & \leq \vol{\B} \sum_{j=0}^{d-1} \delta^{d - j} \binom{d}{j} \left(R + \delta\right)^j \\
	& = \vol{\B} \sum_{j=0}^{d-1} \delta^{d - j} \binom{d}{j} \sum_{k=0}^{j} \binom{j}{k} R^k \delta^{j-k} \\
	& = \vol{\B} \sum_{k=0}^{d-1} \delta^{d - k} R^k \sum_{j=k}^{d-1} \binom{d}{j} \binom{j}{k}.
	\end{aligned}
	\]
From \eqref{K volume bound} and the last inequality we see that our claim holds true for 
	\[	c_d = \vol{\B} \max_{k=0,\dots,d-1} \sum_{j=k}^{d-1} \binom{d}{j} \binom{j}{k},	\]
the maximum of all the terms that are constant in $R$ and $\delta$ in the sums on the right-hand sides of the two excess volume bounds.
\end{proof}

\subsection{Consistency of the illumination on unbounded sets}	\label{section:consistency}

Over unbounded subsets of $\R^d$ with $d>1$, neither illumination, nor the illumination depth are uniformly consistent. So see this take a convex body $K$ in $\R^d$, $y$ in the distance of $\varepsilon > 0$ from $K$, and let $K_y = \co{K \cup \{ y \}}$. Surely, $\Haus\left(K_y, K\right) = \varepsilon$. By the Hahn-Banach separation theorem \cite[Theorem~1.3.7]{Schneider2014}, $y$ and $K$ can be strongly separated by two parallel hyperplanes $H_1$, $H_2$ whose distance is at least $\varepsilon/2$ and $y \in H_1$. Take $x \in H_2$ far enough from $y$. The illumination $\Ill\left(x;K_y\right)$ and $\Ill\left(x;K\right)$ then differs by, at least, the illumination of $x$ onto the cone $K_y \cap H_2^+$ for $H_2^+$ the halfspace whose boundary is $H_2$ and $y \in H_2^+$. This illumination can be bounded from both below and above by the illumination of $x$ on any two balls $B_1$ and $B_2$ centred at some $z \in K_y \cap H_2^+$ such that $B_1 \subset K_y \cap H_2^+ \subset B_2$, respectively. By Lemmas~\ref{lemma:ellipsoid} and~\ref{lemma:g properties}, the latter two illuminations both grow with increasing $R = \left\Vert z - x \right\Vert$ at a rate $\mathcal O\left(R\right)$, i.e. $\Ill\left(x;K_y\right) - \Ill\left(x;K\right) = \mathcal O\left(\left\Vert z - x \right\Vert\right)$ with $H_2 \ni x \to \infty$. In other words, for any $\varepsilon > 0$ one can find $x$ far enough so that $\Ill\left(x;K_y\right) - \Ill\left(x;K\right) \geq 1$. Consequently, even if the distance $\Haus\left(K_n, K\right)$ converges to zero (almost surely), the illumination differences $\left\vert \Ill\left(x;K_n\right) - \Ill\left(x;K\right) \right\vert$ and $\left\vert \Ill\left(x;K_n\right)/\vol{K_n} - \Ill\left(x;K\right)/\vol{K} \right\vert$ cannot, in general, vanish uniformly over unbounded sets. The same example applies to the second component of $\ID$.

\subsection{Proof of Theorem~\ref{theorem:breakdown}}	\label{proof:breakdown}
For $x$ fixed, the illumination of $x$ tends to infinity if and only if the halfspace depth central region $P_{n,\alpha}$ breaks down. Hence, it is enough to evaluate the breakdown point of $P_{n,\alpha}$ with respect to the Hausdorff distance. We follow the derivations in the proofs of \cite[Proposition~2.2]{Donoho1982} and \cite[Proposition~3.2]{Donoho_Gasko1992}. Let $x_M \in \R^d$ be (any) halfspace median of $P_n$, that is let $\HD\left(x_M;P_n\right) = \Pi(P_n)$. By the argument used in the proof of \cite[Lemma~3.1]{Donoho_Gasko1992} to upset the set $P_{n,\alpha} = \left\{ y \in \R^d \colon \HD\left(y;P_n\right) n \geq \lceil \alpha n \rceil \right\}$ entirely, the smallest number of additional points that need to be added to the data is $m$, the smallest integer that satisfies $m \geq \lceil \alpha \left(m + n \right) \rceil$ (compare with formula~(6.19) in \cite{Donoho_Gasko1992}). This inequality is solved by $m = \lceil (\alpha/(1 - \alpha)) n \rceil$. The additional condition $\alpha \leq \Pi(P_n)/(1+\Pi(P_n))$ ensures that $m \leq \lceil \Pi(P_n)\, n \rceil = \Pi(P_n)\, n$. From this it follows that the depth of $x_M$ with respect to the contaminated dataset must be at least $\Pi(P_n)\, n/(n + m) \geq \Pi(P_n)/(1 + \Pi(P_n)) \geq \alpha$. Hence, after the contamination procedure, the central region of points whose depth is at least $\alpha$ must be non-empty.

In the situation when $\alpha > \Pi(P_n)/(1+\Pi(P_n))$, due to the nestedness of the central regions $P_{n,\alpha}$, by the previous part of the proof at least 
	\[	m = \left\lceil \frac{\Pi(P_n)/(1+\Pi(P_n))}{1 - \Pi(P_n)/(1+\Pi(P_n))} n \right\rceil = \left\lceil \Pi(P_n)\, n \right\rceil = \Pi(P_n)\, n	\]
contaminating points are needed. 

The corollary with the asymptotic value of the breakdown point follows the same argument as in the proof of \cite[Propositions~3.2 and~3.3]{Donoho_Gasko1992}.

\subsection{Proof of Theorem~\ref{theorem:M properties}}

The proofs of parts~\ref{M properties i}, \ref{M properties ii} and~\ref{M properties iii} are straightforward and analogous to the proof of Theorem~\ref{theorem:affine invariance}. For part~\ref{M properties iv} it is sufficient to realise that according to the non-degeneracy of $P$, and symmetry conditions imposed on the estimator $F_n$, the lower level set of $M_\alpha\left(\cdot;P_n\right)$ is large if and only if either \begin{enumerate*}[label=(\roman*)] \item the central region $P_{n,\alpha}$ is extremely large; or \item $F_n^{-1}\left(1-\alpha\right)$ is extremely small. \end{enumerate*} By Theorem~\ref{theorem:breakdown}, for the former case, asymptotically at least $m \approx n \min\{\alpha,1/3\}/(1 - \min\{\alpha,1/3\})$ contaminating points have to be added to the random sample to disrupt the central region entirely. In the latter case, unless there exists a configuration of $m$ points that make $F_n^{-1}\left( 1 - \alpha \right)$ arbitrarily small, the set $P_{n,\alpha}$ cannot be made arbitrarily large. By extension, no fixed lower level set \eqref{level set of M} can then be made too big. By the assumption on the breakdown point of $F_n^{-1}\left(1 - \alpha\right)$, in the second scenario it is even more difficult to break down the estimator \eqref{level set of M} than in the first one. Another option when the lower level set \eqref{level set of M} breaks down is when it is an empty set. But, that can happen only if for some $\delta > 0$ small enough, $F_n^{-1}\left( 1 - \Pi\left(P_n\right) \right) > \delta$. This is ruled out by the additional condition imposed on $\delta$. Thus, the resulting limiting breakdown point of the level set is the same as that of $P_{n,\alpha}$.

\subsection{Proof of Theorem~\ref{theorem:M consistency}}

By \eqref{elliptical depth} and \eqref{illumination for measure}, the Mahalanobis distance $\dist(x,\mu)$ can be written either as $F^{-1}\left( 1 - \HD\left(x;P\right)\right)$ for any $x \in \R^d$, or, in case when $x \notin P_{\alpha}$, also as $F^{-1}\left(1 - \alpha\right) g_d^{-1}\left(\Ill\left(x;P_\alpha\right)/\vol{P_\alpha}\right)$. It is thus sufficient to bound
	\[
	\begin{aligned}
	\sup_{x \in K_n} & \left\vert M_\alpha\left(x;P_n\right) - \dist(x,\mu) \right\vert \\
	& \leq \sup_{x \in K_n \cap P_{n,\alpha}} \left\vert M_\alpha\left(x;P_n\right) - \dist(x,\mu) \right\vert + \sup_{x \in K_n \setminus P_{n,\alpha}} \left\vert M_\alpha\left(x;P_n\right) - \dist(x,\mu) \right\vert \\
	& \leq \sup_{x \in K_n \cap P_{n,\alpha}} \left\vert F_n^{-1}\left(1 - \HD\left(x;P_n\right)\right) - F^{-1}\left(1 - \HD\left(x;P\right)\right) \right\vert \\
	& \phantom{\leq} + \sup_{x \in K_n \setminus \left(P_{n,\alpha} \cup P_\alpha\right)} \left\vert F_n^{-1}\left(1 - \alpha\right) g_d^{-1}\left(\frac{\Ill\left(x;P_{n,\alpha}\right)}{\vol{P_{n,\alpha}}} \right) - F^{-1}\left(1 - \alpha\right) g_d^{-1}\left(\frac{\Ill\left(x;P_{\alpha}\right)}{\vol{P_{\alpha}}} \right) \right\vert \\
	& \phantom{\leq} + \sup_{x \in K_n \cap \left( P_\alpha \setminus P_{n,\alpha} \right)} \left\vert M_\alpha\left(x;P_n\right) - \dist(x,\mu) \right\vert.
	\end{aligned}
	\]
The three suprema on the right hand side will be treated separately. Denote them by \textbf{I}, \textbf{II}, and \textbf{III}, respectively.

\subsubsection{Supremum \textbf{I}} The sample halfspace depth $\HD\left(\cdot;P_n\right)$ is known \cite[formula~(6.6)]{Donoho_Gasko1992} to be a uniformly consistent estimator of its population version
	\begin{equation}	\label{HD consistency}
	\sup_{x \in \R^d} \left\vert \HD\left(x;P_n\right) - \HD\left(x;P\right) \right\vert \xrightarrow[n\to\infty]{\as}0.	
	\end{equation}
Because $P$ is halfspace symmetric, yet its centre of symmetry has zero probability mass, for any $x \in K_n \cap P_{n,\alpha}$ we have $\alpha \leq \HD\left(x;P_n\right) \leq 1/2$, with the second inequality almost surely for all $n$ large enough due to \eqref{HD consistency}. We may use the consistency \eqref{HD consistency} again to get that for any $\varepsilon > 0$ small, $\alpha - \varepsilon \leq \HD\left(x;P\right) \leq 1/2$ for all $x \in K_n \cap P_{n,\alpha}$ and $n$ large enough.

Function $F$ is strictly increasing in a neighbourhood of $[0,F^{-1}\left(1 - \alpha\right)]$. Thus, $F^{-1}$ is (uniformly) continuous on $I = [1/2, 1 - \alpha]$. Its approximating sequence $\left\{ F_n^{-1} \right\}_{n=1}^\infty$ is a sequence of functions that are non-decreasing, and converge to $F^{-1}$ at each $t \in I$ by the uniform consistency of $F_n$ from \eqref{F consistency}, and \cite[Lemma~21.2]{Vandervaart1998}. A lemma of P\'olya \cite[Problem~127, part~II]{Polya_Szergo1998} gives that this convergence is uniform on $I$. We can thus write for $n$ large enough
	\[	
	\begin{aligned}
	\textbf{I} & \leq \sup_{x \in K_n \cap P_{n,\alpha}} \left\vert F_n^{-1}\left(1 - \HD\left(x;P_n\right)\right) - F^{-1}\left(1 - \HD\left(x;P_n\right)\right) \right\vert \\
	& \phantom{\leq} + \sup_{x \in K_n \cap P_{n,\alpha}} \left\vert F^{-1}\left(1 - \HD\left(x;P_n\right)\right) - F^{-1}\left(1 - \HD\left(x;P\right)\right) \right\vert \\
	& \leq \sup_{t \in I} \left\vert F_n^{-1}\left(t\right) - F^{-1}\left(t\right) \right\vert + w_{F^{-1}}\left(\sup_{x \in \R^d} \left\vert \HD\left(x;P_n\right) - \HD\left(x;P\right) \right\vert \right),
	\end{aligned}
	\]
where $w_{F^{-1}}$ is the minimal modulus of continuity of $F^{-1}$ restricted to the interval $I$. The first supremum on the right hand is small almost surely for $n$ large by the uniform convergence of the quantile functions established above. The second will vanish almost surely because of \eqref{HD consistency} and the uniform continuity of $F^{-1}$ on $I$.

\subsubsection{Supremum \textbf{II}} Let us first introduce the notation
	\begin{equation}	\label{a notation}
	\begin{aligned}
	a_x & = - g_d^{-1}\left(\frac{\Ill\left(x;P_{\alpha}\right)}{\vol{P_{\alpha}}}\right), \\
	a_{n,x} & = - g_d^{-1}\left(\frac{\Ill\left(x;P_{n,\alpha}\right)}{\vol{P_{n,\alpha}}}\right),
	\end{aligned} \qquad 
	\begin{aligned}
	b & = F^{-1}\left(1 - \alpha \right), \\
	b_n & = F_n^{-1}\left(1 - \alpha \right).
	\end{aligned}
	\end{equation} 
In supremum \textbf{II} we bound for $L_n^{II} = K_n \setminus \left(P_{n,\alpha} \cup P_\alpha\right)$
	\begin{equation}	\label{a b bound}
	\sup_{x \in L_n^{II}} \left\vert a_{n,x} b_n - a_x b\right\vert \leq \left\vert b_n - b \right\vert \sup_{x \in L_n^{II}} \left\vert a_x \right\vert + \left\vert b_n \right\vert \sup_{x \in L_n^{II}} \left\vert a_{n,x} - a_x \right\vert.	
	\end{equation}
For the supremum in the first summand in \eqref{a b bound} we know from \eqref{illumination for measure} that
	\begin{equation}	\label{a bound}	
	\begin{aligned}
	\sup_{x \in L_n^{II}} \left\vert a_x \right\vert & \leq \sup_{x \in \B\left(\mu,R_n\right)} \frac{\dist(x,\mu)}{F^{-1}\left(1 - \alpha\right)} = \frac{R_n}{F^{-1}\left(1 - \alpha\right)} \sup_{x \in \B\left(\mu,1\right)} \dist(x,\mu) \\
	& = \frac{R_n}{F^{-1}\left(1 - \alpha\right)} \sqrt{\sup_{\left\Vert x \right\Vert = 1} x\tr \Sigma^{-1} x}	= \frac{R_n}{F^{-1}\left(1 - \alpha\right)} \sqrt{1/\lambda} = \mathcal O\left(R_n\right),
	\end{aligned}
	\end{equation}
where $\lambda > 0$ is the smallest eigenvalue of $\Sigma$. Using the assumption~\eqref{quantile rate} we see that the first summand on the right hand side of \eqref{a b bound} vanishes in probability as $n\to\infty$. Furthermore, by \eqref{quantile rate} we also have that $\left\vert b_n \right\vert = \mathcal O_{\PP}(1)$, and by Lemma~\ref{lemma:g properties} together with Theorem~\ref{theorem:consistency}
	\begin{equation}	\label{an a bound}	
	\begin{aligned}
	\sup_{x \in L_n^{II}} \left\vert a_{n,x} - a_x \right\vert & \leq w_{g_d^{-1}}\left( \sup_{x \in L_n^{II}} \left\vert \frac{\Ill\left(x;P_{n,\alpha}\right)}{\vol{P_{n,\alpha}}} - \frac{\Ill\left(x;P_\alpha\right)}{\vol{P_\alpha}} \right\vert \right) = w_{g_d^{-1}} \left( \mathcal O_{\PP}\left( \frac{\max\left\{1,R_n^{d-1}\right\}}{\sqrt{n}} \right) \right) \\
	& = \mathcal O_{\PP}\left( \left( \frac{\max\left\{1,R_n^{d-1}\right\}}{\sqrt{n}} \right)^{2/(d+1)} \right) = o_{\PP}(1),
	\end{aligned}
	\end{equation}
where $w_{g_d^{-1}}$ is the minimal modulus of continuity of $g_d^{-1}$ from Lemma~\ref{lemma:g properties}. Together, we have verified that
	\begin{equation}	\label{a b rate}	
	\sup_{x \in L_n^{II}} \left\vert a_{n,x} b_n - a_x b\right\vert = \mathcal O\left(R_n\right) o_{\PP}\left(1/R_n\right) + \mathcal O_{\PP}(1) \mathcal O_{\PP}\left( \left( \frac{\max\left\{1,R_n^{d-1}\right\}}{\sqrt{n}} \right)^{2/(d+1)} \right) = o_{\PP}\left(1\right).	
	\end{equation}
	
\subsubsection{Supremum \textbf{III}} Here it will be crucial that under the conditions of the theorem, the set $P_\alpha \setminus P_{n,\alpha}$ is negligible as $n\to\infty$ by the consistency of the halfspace depth contours \eqref{Dyckerhoff}. First, without loss of generality, suppose that both $P_\alpha$ and $P_{n,\alpha}$ are contained in $K_n$. This is possible, because $P_\alpha$ is a fixed set, and the sequence $P_{n,\alpha}$ is convergent almost surely by \eqref{Dyckerhoff}. Thus, possible enlargement of $K_n$ by a fixed set does not affect any results in this proof. Take $x \in L_n^{III} = K_n \cap \left( P_\alpha \setminus P_{n,\alpha} \right)$. As $x \notin P_{n,\alpha}$, 
	\[	M_\alpha\left(x;P_n\right) = F_n^{-1}\left(1 - \alpha\right) g_d^{-1}\left( \frac{\Ill\left(x;P_{n,\alpha}\right)}{\vol{P_{n,\alpha}}} \right).	\]
In terms of $x$, this expression varies monotonically with $\Ill\left(x;P_{n,\alpha}\right)$. Note that for $K$ a convex body, the illumination $\Ill\left(\cdot;K\right)$ strictly increases on any straight halfline $L$ that starts from $x \in \partial K$ (the boundary of $K$) and does not intersect $K$ elsewhere, i.e. $K \cap L = \{ x \}$. Thus, in our situation, if one considers any halfline that starts at a boundary point of $P_{n,\alpha}$ and passes through $x$,
	\[	\inf_{y \in \partial P_{n,\alpha}} M_\alpha\left(y;P_n\right) \leq M_\alpha\left(x;P_n\right) \leq \sup_{y \in \partial P_\alpha} M_\alpha\left(y;P_n\right).	\]
On the boundary of $P_{n,\alpha}$ we are in the situation dealt with in supremum \textbf{I}, and by that part of the proof we know that for $\varepsilon > 0$ given, almost surely for any $n$ large enough, 
	\[	\inf_{y \in \partial P_{n,\alpha}} \dist(y,\mu) - \varepsilon \leq \inf_{y \in \partial P_{n,\alpha}} M_\alpha\left(y;P_n\right).	\]
Likewise, for the upper bound, by part \textbf{II} of this proof, and the continuity of $\dist(x,\mu)$, we have an analogous restriction, and with high probability, for $n$ large enough,
	\[	\sup_{y \in \partial P_\alpha} M_\alpha\left(y;P_n\right) \leq \sup_{y \in \partial P_\alpha} \dist(y,\mu) + \varepsilon = F^{-1}\left(1 - \alpha\right) + \varepsilon.	\]	
Finally, we use \eqref{Dyckerhoff} and the fact that the Hausdorff distances of convex bodies, and of their boundaries, are the same \cite[Lemma~1.8.1]{Schneider2014}. This gets that almost surely, for any $\delta > 0$, for all $n$ large enough, and any $y \in \partial P_{n,\alpha}$, there exists $z \in \partial P_\alpha$ such that $\left\Vert y - z \right\Vert < \delta$. Now, because $\dist(x,\mu)$ is in $x$ (uniformly) continuous in a uniform neighbourhood of $P_\alpha$, this means that almost surely, for $n$ large enough,
	\[	
	\begin{aligned}
	\inf_{y \in \partial P_{n,\alpha}} \dist(y,\mu) & \geq \inf_{y \in \partial P_\alpha} \dist(y,\mu) - \varepsilon = F^{-1}\left(1 - \alpha\right) - \varepsilon,
	\end{aligned}
	\]
and for any $x \in L_n^{III}$
	\[	F^{-1}\left(1 - \alpha\right) - \varepsilon \leq \dist(x,\mu) \leq F^{-1}\left( 1 - \alpha \right).	\]
Altogether, collect all the bounds in this part of the proof to get that for any $\varepsilon > 0$, with high probability, for $n$ large enough, 
	\[	\sup_{x \in L_n^{III}} \left\vert M_\alpha\left(x;P_n\right) - \dist(x,\mu) \right\vert \leq 2 \varepsilon,	\]
which finishes the proof.

\subsection{Proof of Theorem~\ref{theorem:F consistency}}

In view of the uniform consistency of the halfspace depth \eqref{HD consistency} it suffices to show that
	\[	\sup_{x \in K_n \setminus P_{n,\alpha}} \left\vert F_n\left(-g_d^{-1}\left(\frac{\Ill\left(x;P_{n,\alpha}\right)}{\vol{P_{n,\alpha}}} \right) F_n^{-1}\left(1 - \alpha \right) \right) - \HD\left(x;P\right) \right\vert \xrightarrow[n\to\infty]{\PP}0.	\]
Proceed analogously as in the proof of Theorem~\ref{theorem:M consistency}, and consider two situations --- the supremum above over $x \in L_n^{II} = K_n \setminus \left( P_{n,\alpha} \cup P_\alpha \right)$, and the the supremum over $x \in L_n^{III} = K_n \cap \left(P_\alpha \setminus P_{n,\alpha}\right)$. 

Suppose first that $x \in L_n^{II}$. By \eqref{elliptical depth} and \eqref{illumination for measure}, in the notation from \eqref{a notation} we have that
	\[
	\begin{aligned}
	\left\vert F_n\left(a_{n,x} b_n\right) - \HD\left(x;P\right) \right\vert & = \left\vert F_n\left( a_{n,x} b_n \right) - F\left(-\dist(x,\mu)\right) \right\vert \\
	& = \left\vert F_n\left( a_{n,x} b_n \right) - F\left(a_x b\right) \right\vert \\
	& \leq \left\vert F_n(a_{n,x} b_n) - F(a_{n,x} b_n) \right\vert + \left\vert F(a_{n,x} b_n) - F(a_x b) \right\vert.
	\end{aligned}
	\]
Therefore, 
	\[	\sup_{x \in L_n^{II}} \left\vert F_n\left(a_{n,x} b_n\right) - \HD\left(x;P\right) \right\vert \leq \sup_{t \in \R} \left\vert F_n(t) - F(t) \right\vert + \sup_{x \in L_n^{II}} \left\vert F(a_{n,x} b_n) - F(a_x b) \right\vert.	\]
The first summand on the right hand side above vanishes almost surely as $n\to\infty$ by \eqref{F consistency}. For the second summand we already have a bound from \eqref{a b rate} from the proof of Theorem~\ref{theorem:M consistency}. Since $F$ has a density, it must be uniformly continuous on $\R$. Denote by $w_F \colon (0,\infty) \to \R$ its minimal modulus of continuity. We obtain 
	\[	\sup_{x \in L_n^{II}} \left\vert F(a_{n,x} b_n) - F(a_x b) \right\vert \leq w_F\left( \sup_{x \in L_n^{II}} \left\vert a_{n,x}b_n - a_x b \right\vert \right) = w_F\left( o_{\PP}(1) \right) = o_{\PP}(1).	\]
This completes the part of the proof with $L_n^{II}$.

For the second part, consider $x \in L_n^{III}$. Note that thanks to \eqref{elliptical depth} and the continuity of $F$ in a neighbourhood of $F^{-1}\left(1 - \alpha\right)$, the halfspace depth $\HD\left(\cdot;P\right)$ must be (uniformly) continuous in a uniform neighbourhood of $P_\alpha$. Furthermore, for $x \notin P_{n,\alpha}$, $\RHD\left(x;P_n\right)$ varies monotonically with $\Ill\left(x;P_{n,\alpha}\right)$. Thus, derivation analogous to that from part \textbf{III} in the proof of Theorem~\ref{theorem:M consistency} gives that the convergence of the halfspace depth contours \eqref{Dyckerhoff} implies that with $n\to\infty$
	\[	\sup_{x \in L_n^{III}} \left\vert \RHD\left(x;P_n\right) - \HD\left(x;P\right) \right\vert = o_{\PP}(1),	\]
and the proof is finished.

\subsection{Proof of Theorem~\ref{theorem:F multiplicative consistency}}

By the uniform consistency of the halfspace depth \eqref{HD consistency} we can bound for $n$ large enough
	\[
	\begin{aligned}
	\sup_{x \in \R^d \colon \HD\left(x;P_n\right) \geq \alpha} \left\vert \frac{\RHD\left(x;P_n\right)}{\HD\left(x;P\right)} - 1 \right\vert & \leq \sup_{x \in \R^d \colon \HD\left(x;P_n\right) \geq \alpha} 2 \left\vert \frac{\RHD\left(x;P_n\right) - \HD\left(x;P\right)}{\HD\left(x;P_n\right)}\right\vert	\\
	& \leq \frac{2}{\alpha} \sup_{x \in \R^d} \left\vert \HD\left(x;P_n\right) - \HD\left(x;P\right) \right\vert, 
	\end{aligned}
	\]
where the last term vanishes almost surely as $n\to\infty$. Thus, in the notation established in \eqref{a notation} in the proof of Theorem~\ref{theorem:M consistency}, it suffices to show that also the right hand size of
	\begin{equation*}	
	\sup_{x \in K_n \setminus P_{n,\alpha}} \left\vert \frac{F_n\left(a_{n,x} b_n\right)}{\HD\left(x;P\right)} - 1 \right\vert \leq \sup_{x \in L_n^{II}} \left\vert \frac{F_n\left(a_{n,x}b_n\right)}{F\left( a_x b \right)} - 1 \right\vert + \sup_{x \in L_n^{III}} \left\vert \frac{F_n\left(a_{n,x}b_n\right)}{\HD\left(x;P\right)} - 1 \right\vert	
	\end{equation*}
is asymptotically negligible, where $L_n^{II} = K_n \setminus \left( P_{n,\alpha} \cup P_\alpha \right)$ and $L_n^{III} = K_n \cap \left(P_\alpha \setminus P_{n,\alpha}\right)$. We used \eqref{elliptical depth} and \eqref{illumination for measure} to obtain the expression on the right hand side. We already have everything prepared to bound the second summand above. Indeed, by Theorem~\ref{theorem:F consistency}
	\[
	\begin{aligned}
	\sup_{x \in L_n^{III}} \left\vert \frac{F_n\left(a_{n,x}b_n\right)}{\HD\left(x;P\right)} - 1 \right\vert & \leq \sup_{x \in L_n^{III}} \frac{\left\vert \RHD\left(x;P_n\right) - \HD\left(x;P\right) \right\vert}{\alpha} \\
	& \leq \frac{1}{\alpha} \sup_{x \in K_n} \left\vert \RHD\left(x;P_n\right) - \HD\left(x;P\right) \right\vert = o_{\PP}\left(1\right).
	\end{aligned}
	\]  

Let us now focus on the supremum over $L_n^{II}$. For $x \in L_n^{II}$ we can write
	\begin{equation}	\label{multiplicative consistency expansion}
	\begin{aligned}
	\left\vert \frac{F_n\left( a_{n,x} b_n \right)}{F(a_x b)} - 1 \right\vert & = \left\vert \frac{F(a_{n,x} b_n)}{F(a_x b)} \right\vert \left\vert \frac{F_n\left(a_{n,x} b_n\right)}{F\left(a_{n,x} b_n\right)} - \frac{F(a_x b)}{F(a_{n,x} b_n)} \right\vert \\
	& \leq \left\vert \frac{F(a_{n,x} b_n)}{F(a_x b)} \right\vert \left( \left\vert \frac{F_n\left(a_{n,x} b_n\right)}{F\left(a_{n,x} b_n\right)} - 1 \right\vert + \left\vert \frac{F(a_x b)}{F(a_{n,x} b_n)} - 1 \right\vert \right).
	\end{aligned}
	\end{equation}
In the same way as in \eqref{a b bound}, \eqref{a bound}, \eqref{an a bound} and \eqref{a b rate} in the proof of Theorem~\ref{theorem:M consistency} we have, using \eqref{quantile rate2}, that
	\begin{equation}	\label{a and a b rates}
	\begin{aligned}
	\sup_{x \in L_n^{II}} \left\vert a_{n,x} - a_x \right\vert & = \mathcal O_{\PP} \left( \omega_n \right),	\\
	\sup_{x \in L_n^{II}} \left\vert a_{n,x} b_n - a_x b\right\vert & = \mathcal O\left(R_n\right) \mathcal O_{\PP}\left(\xi_n\right) + \mathcal O_{\PP}\left( \omega_n \right) = \mathcal O_{\PP} \left(\max\left\{R_n \xi_n, \omega_n \right\}\right),	\\
	b \sup_{x \in L_n^{II}} \left\vert a_x \right\vert & = R_n/\sqrt{\lambda} = \mathcal O\left(R_n\right).
	\end{aligned}
	\end{equation}
By the definition of the refined depth \eqref{refined depth} we also see that $a_{n,x} < -1$ for any $x\in L_n^{II}$. Combine this with \eqref{a and a b rates} to obtain that there exists $c>0$ such that for all $n \geq 1$ and $x \in L_n^{II}$ we can write $\left(1 - c\, \omega_n \right) b < \left\vert a_x b \right\vert$, which means that for $n$ large enough $b/2 < \left\vert a_x b \right\vert < R_n/\sqrt{\lambda}$ for all $x \in L_n^{II}$. Formulas \eqref{a and a b rates} therefore allow us to write for some $c>0$ large enough for all $\varepsilon>0$ and $n$ large
	\begin{equation*}	
	\begin{aligned}
	\PP & \left( \sup_{x \in L_n^{II}} \left\vert \frac{F(a_x b)}{F(a_{n,x} b_n)} - 1 \right\vert > \varepsilon \right) \leq \PP\left( \sup_{x \in L_n^{II}} \left\vert a_{n,x} b_n - a_x b \right\vert > c\, \max\left\{R_n \xi_n, \omega_n \right\} \right) \\
	& + \PP\left( \sup_{x \in L_n^{II}} \left\vert a_{n,x} - a_x \right\vert > c\, \omega_n \right) + \PP\left( \sup_{\substack{b/2 < \left\vert s \right\vert < R_n/\sqrt{\lambda} \\ \left\vert s - t \right\vert < c\, \max\left\{R_n \xi_n, \, \omega_n \right\}}} \left\vert \frac{F(s)}{F(t)} - 1 \right\vert > \varepsilon\right). \\
	\end{aligned}
	\end{equation*}
The first two summands on the right hand side vanish with $n\to\infty$ because of the first two formulas in \eqref{a and a b rates}. The argument in the last summand is non-random, and the probability is therefore equal to zero for $n$ large by \eqref{F multiplicative continuity}. Thus,
	\begin{equation}	\label{convergence continuity}
	\sup_{x \in L_n^{II}} \left\vert \frac{F(a_x b)}{F(a_{n,x} b_n)} - 1 \right\vert \xrightarrow[n\to\infty]{\PP}0.	
	\end{equation}
	
Using similar argumentation we have that there is $c>0$ with the property that for all $\varepsilon>0$ and $n$ large enough
	\[	
	\begin{aligned}
	\PP & \left( \sup_{x \in L_n^{II}} \left\vert \frac{F_n\left(a_{n,x} b_n\right)}{F\left(a_{n,x} b_n\right)} - 1 \right\vert > \varepsilon \right) \\
	& \leq \PP\left( \sup_{x \in L_n^{II}} \left\vert a_{n,x} b_n - a_x b \right\vert > c\, \omega_n \right) + \PP\left( \sup_{\left\vert a_{n,x} b_n \right\vert < R_n/\sqrt{\lambda} + c\, \omega_n} \left\vert \frac{F_n\left(a_{n,x} b_n \right)}{F\left(a_{n,x} b_n\right)} - 1 \right\vert > \varepsilon \right) \\	
	& \leq \PP\left( \sup_{x \in L_n^{II}} \left\vert a_{n,x} b_n - a_x b \right\vert > c\, \omega_n \right) + \PP\left( \sup_{\left\vert t \right\vert < 2 R_n /\sqrt{\lambda}} \left\vert \frac{F_n\left(t\right)}{F\left(t\right)} - 1 \right\vert > \varepsilon \right),
	\end{aligned}
	\]
and the last expression tends to zero in as $n\to\infty$ thanks to the second rate in \eqref{a and a b rates} and \eqref{F multiplicative consistency}. Thus,
	\begin{equation}	\label{convergence consistency}
	\sup_{x \in L_n^{II}} \left\vert \frac{F_n\left(a_{n,x} b_n\right)}{F\left(a_{n,x} b\right)} - 1 \right\vert \xrightarrow[n\to\infty]{\PP}0.	
	\end{equation}

Altogether, we can start from \eqref{multiplicative consistency expansion} and bound
	\[	\sup_{x \in L_n^{II}} \left\vert \frac{F_n\left(a_{n,x} b_n\right)}{F\left(a_x b\right)} - 1 \right\vert \leq \sup_{x \in L_n^{II}} \left\vert \frac{F(a_{n,x} b_n)}{F(a_x b)} \right\vert \left( \sup_{x \in L_n^{II}} \left\vert \frac{F_n\left(a_{n,x} b_n\right)}{F\left(a_{n,x} b_n\right)} - 1 \right\vert + \sup_{x \in L_n^{II}} \left\vert \frac{F(a_x b)}{F(a_{n,x} b_n)} - 1 \right\vert \right). 	\]
The first term on the right hand side is bounded in probability due to \eqref{convergence continuity}. The summands vanish in probability thanks to \eqref{convergence consistency} and \eqref{convergence continuity}, respectively. The theorem is proved.

\subsection{Proof of Remark~\ref{R7}}

Suppose first that $F(t) \geq c\, \left\vert t \right\vert^\gamma$ for some $c>0$, $t$ small enough and $\gamma < 0$. Consider $R_n = \mathcal O(n^\alpha)$ for $\alpha > 0$. We have
	\[	\frac{\omega_n}{F\left(-R_n/\sqrt{\lambda}\right)} \leq \frac{\left(R_n^{d-1}/\sqrt{n}\right)^{2/(d+1)}}{c \, \left\vert R_n/\sqrt{\lambda} \right\vert^\gamma} = \mathcal O \left( n^{2 \alpha (d-1)/(d+1) - 1/(d+1) - \alpha \gamma} \right).	\]
For the right hand side to be $o(1)$, it is sufficient that $\alpha < \left(2 (d-1) - \gamma (d+1) \right)^{-1}$.
	
If $F(t) \geq c\, e^{-\left\vert t \right\vert^\gamma}$ for some $c>0$, $\gamma > 0$ and all $t$ small enough, we get for $R_n  \leq \left( \left(\frac{1}{d+1} - \varepsilon\right) \log(n) \right)^{1/\gamma} \sqrt{\lambda}$ and for $\varepsilon > 0$ 
	\[	\frac{\omega_n}{F\left(-R_n/\sqrt{\lambda}\right)} \leq \frac{\left(R_n^{d-1}/\sqrt{n}\right)^{2/(d+1)}}{c \, e^{-\left\vert R_n/\sqrt{\lambda} \right\vert^\gamma}} \leq \frac{R_n^{2(d-1)/(d+1)}}{c\,n^\varepsilon} = o(1).	\]
	
For $F(t) \geq c\, \exp\left(-e^{\left\vert t \right\vert^\gamma}\right)$ for some $c>0$, $\gamma > 0$ and all $t$ small enough, $R_n \leq \left( \log\left( \frac{1}{d+1} - \varepsilon\right) + \log \log n \right)^{1/\gamma} \sqrt{\lambda}$ with $\varepsilon > 0$ gives
	\[	\frac{\omega_n}{F\left(-R_n/\sqrt{\lambda}\right)} \leq \frac{\left(R_n^{d-1}/\sqrt{n}\right)^{2/(d+1)}}{c\, \exp\left(-e^{\left\vert R_n/\sqrt{\lambda} \right\vert^\gamma}\right)} \leq \frac{R_n^{2(d-1)/(d+1)}}{c\, n^{\varepsilon}} = o(1).	\]

\subsection{Proof of Theorem~\ref{theorem:QDA}} 	\label{appendix:classification}

The logarithm of the density of $P^{(j)}$ at $x \in \R^d$ can be written as
	\[	\log\left(f_j(x)\right) = - \log\left(\sqrt{(2\pi)^d}\right) - \log\left( \sqrt{\left\vert \Sigma_j \right\vert} \right) -\frac{1}{2}\dist[\Sigma_j]\left(x,\mu_j\right)^2.	\]
By \eqref{elliptical depth} we know that for any $0 < \delta < 1/2$
	\[	
	\begin{aligned}
	\vol{P_\delta^{(j)}} & = \vol{\left\{y \in \R^d \colon \Phi\left(-\dist[\Sigma_j]\left(y,\mu_j\right)\right) \geq \delta \right\}} = \vol{\left\{ y \in \R^d \colon \dist[\widetilde{\Sigma}_j]\left(y,\mu_j\right) \leq 1 \right\}} \\
	& = \vol{\widetilde{\Sigma}_j^{1/2} \B + \mu_j} = \left\vert \widetilde{\Sigma}_j^{1/2} \right\vert \vol{\B} = \Phi^{-1}\left( 1 - \delta \right)^d \sqrt{\left\vert \Sigma_j \right\vert} \vol{\B},	
	\end{aligned}
	\]
where $\widetilde{\Sigma}_j = \Phi^{-1}\left( 1 - \delta \right)^2 \Sigma_j$. Thus, 
	\[	2 \log\left(\pi_j f_j(x)\right) = 2 \log\left(\frac{\Phi^{-1}\left( 1 - \delta \right)^d \vol{\B}}{\sqrt{(2\pi)^d}}\right) + 2 \log\left(\frac{\pi_j}{\vol{P_\delta^{(j)}}} \right) -\dist[\Sigma_j]\left(x,\mu_j\right)^2,	\]
and $\pi_1 f_1(x) > \pi_2 f_2(x)$ if and only if \eqref{classification rule} is true.

The uniform consistency follows from Theorem~\ref{theorem:M consistency}, formula \eqref{Dyckerhoff}, and Lemma~\ref{lemma:Lipschitz}.

\renewcommand{\theequation}{B.\arabic{equation}}
\renewcommand{\thetable}{B.\arabic{table}}
\renewcommand{\thefigure}{B.\arabic{figure}}

\section{Additional simulations and results}	\label{appendix:simulations}

\subsection{Robust classification}

\subsubsection{Bivariate normal distribution, location difference}\label{subsubsec:normal location}

We repeat the same classification experiment as in Section~\ref{subsubsec:normal location scale}, with $P^{(1)} = P_X, P^{(2)} = P_{X+(2,2)\tr}, P^{(3)} = P_{X+(20,20)^T}$. This accounts for classification in presence of only location difference. The results are summarized in Figure~\ref{fig:classification normal location} and Table~\ref{tab:classificationNormalContaminationLocation}. We observe similar results as in Section~\ref{subsubsec:normal location scale}: the optimal (Bayes) error rate is nearly achieved by the illumination-based approach and the classical QDA. The approach based on the refined depth performs worse, especially in the extremes, and it is very sensitive to possible contamination.

\begin{figure}[htpb]
\includegraphics[width=0.45\textwidth]{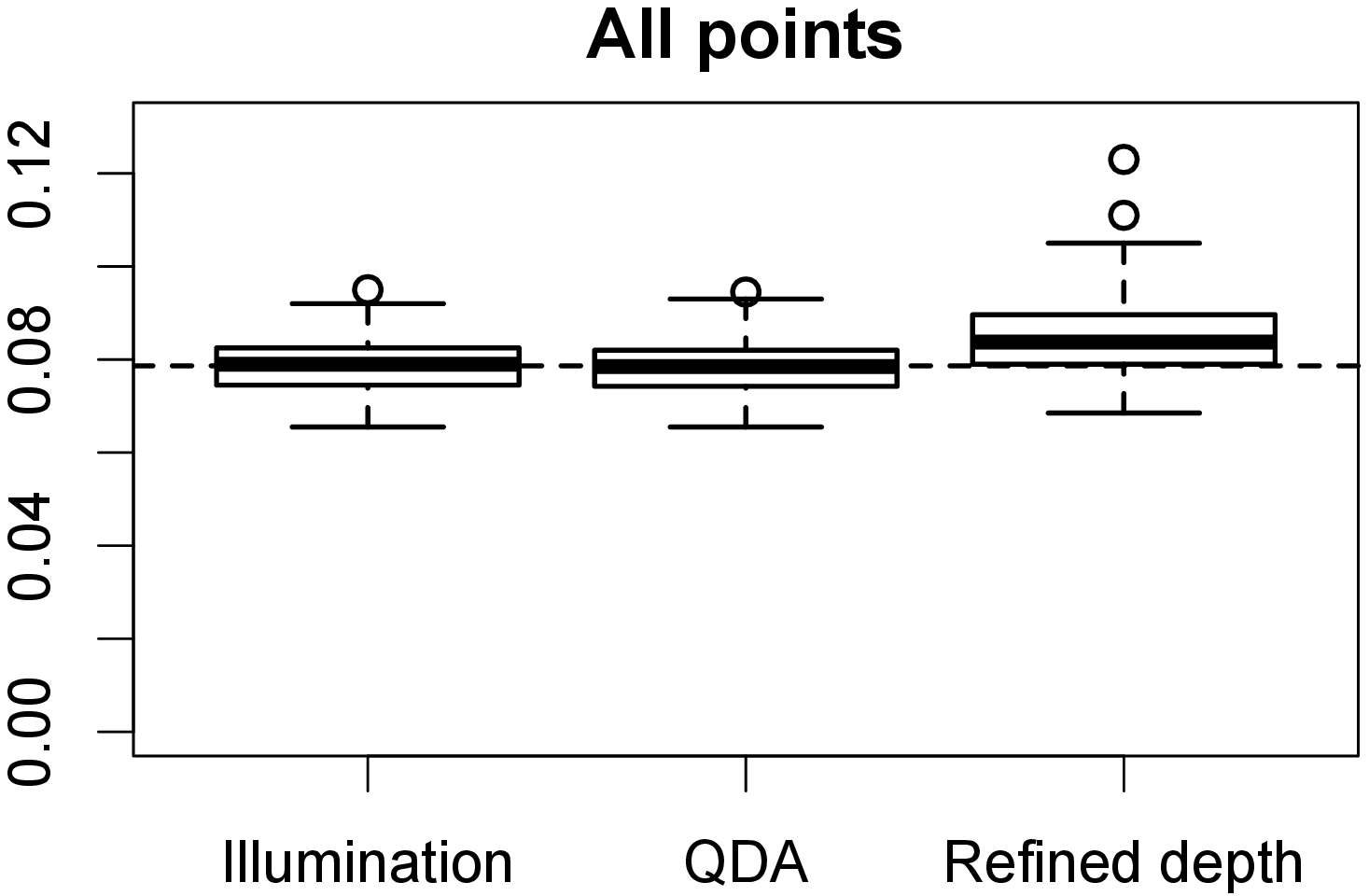}
\includegraphics[width=0.45\textwidth]{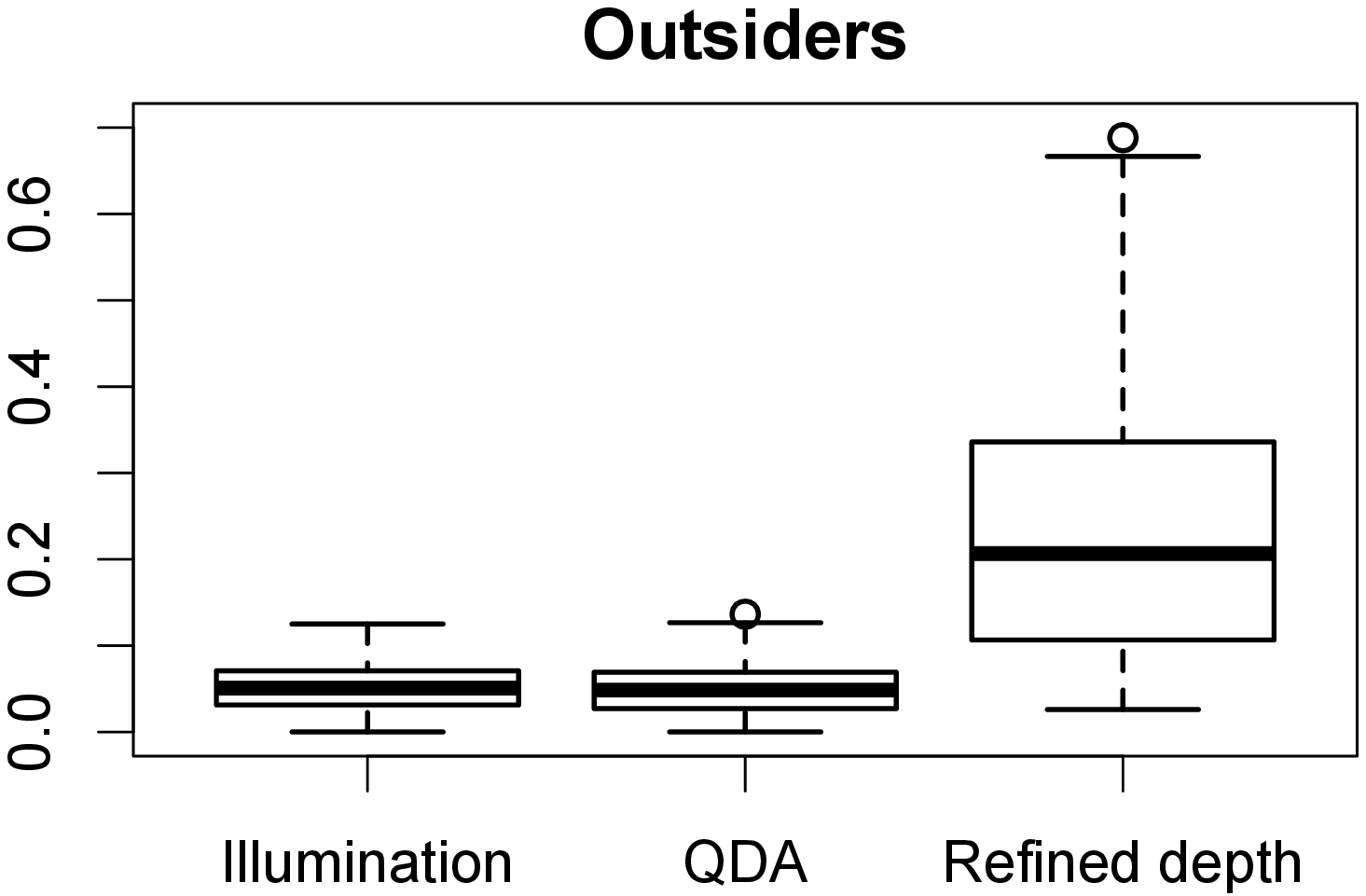}
\caption{Misclassification rates, based on 100 replications of the experiment with two bivariate normal distributions with different location and same scale. Based on all testing points (left panel) or the outsiders (right panel). Dashed horizontal line in the left panel corresponds to the theoretical Bayes error rate.}
\label{fig:classification normal location}
\end{figure}

\begin{table}[htpb]
\resizebox{\textwidth}{!}{
\begin{tabular}{c|c|c|c||c|c|c}
    & \multicolumn{3}{c||}{All points} & \multicolumn{3}{c}{Outsiders} \\ 
    & Illumination & QDA & Ref. depth & Illumination & QDA & Ref. depth \\ \hline
	0 \%   & 0.079 {\footnotesize (0.006)} & 0.079 {\footnotesize (0.006)} & 0.085 {\footnotesize (0.009)} & 0.054 {\footnotesize (0.030)} & 0.051 {\footnotesize (0.031)} & 0.236 {\footnotesize (0.165)} \\
  1 \%   & 0.079 {\footnotesize (0.006)} & 0.089 {\footnotesize (0.007)} & 0.101 {\footnotesize (0.010)} & 0.039 {\footnotesize (0.032)} & 0.059 {\footnotesize (0.056)} & 0.236 {\footnotesize (0.235)} \\
	5 \%   & 0.081 {\footnotesize (0.006)} & 0.118 {\footnotesize (0.010)} & 0.113 {\footnotesize (0.011)} & 0.047 {\footnotesize (0.038)} & 0.065 {\footnotesize (0.052)} & 0.216 {\footnotesize (0.109)} \\
	10 \%  & 0.087 {\footnotesize (0.006)} & 0.135 {\footnotesize (0.012)} & 0.120 {\footnotesize (0.011)} & 0.066 {\footnotesize (0.045)} & 0.069 {\footnotesize (0.054)} & 0.210 {\footnotesize (0.117)} \\ \hline
\end{tabular}}
\caption{Average misclassification rates and their standard deviations (in brackets), bivariate normal distributions with different location and same scale, level of contamination in one of the training samples ranging from 0 to 10~\%. Based on 100 replications of the experiment and all testing points (left part) and outsiders (right part), respectively.}
\label{tab:classificationNormalContaminationLocation}
\end{table}


\subsubsection{Bivariate elliptical distribution, location difference}\label{subsubsec:elliptical location}

Finally, consider the experiment from Section~\ref{subsubsec:elliptical location scale} with $P^{(1)} = P_Y, P^{(2)} = P_{Y+(2,2)\tr}, P^{(3)} = P_{X+(20,20)\tr}$. Our results are summarized in Figure~\ref{fig:classification elliptical location} and Table~\ref{tab:classificationEllipticalContaminationLocation}. We observe that in the case with no contamination, the illumination-based approach and the classical QDA, used only as a reference method here, perform slightly better than the method based on the refined depth. If some contamination is present, the robust QDA appears to outperform both competitors.

\begin{figure}[htpb]
\includegraphics[width=0.45\textwidth]{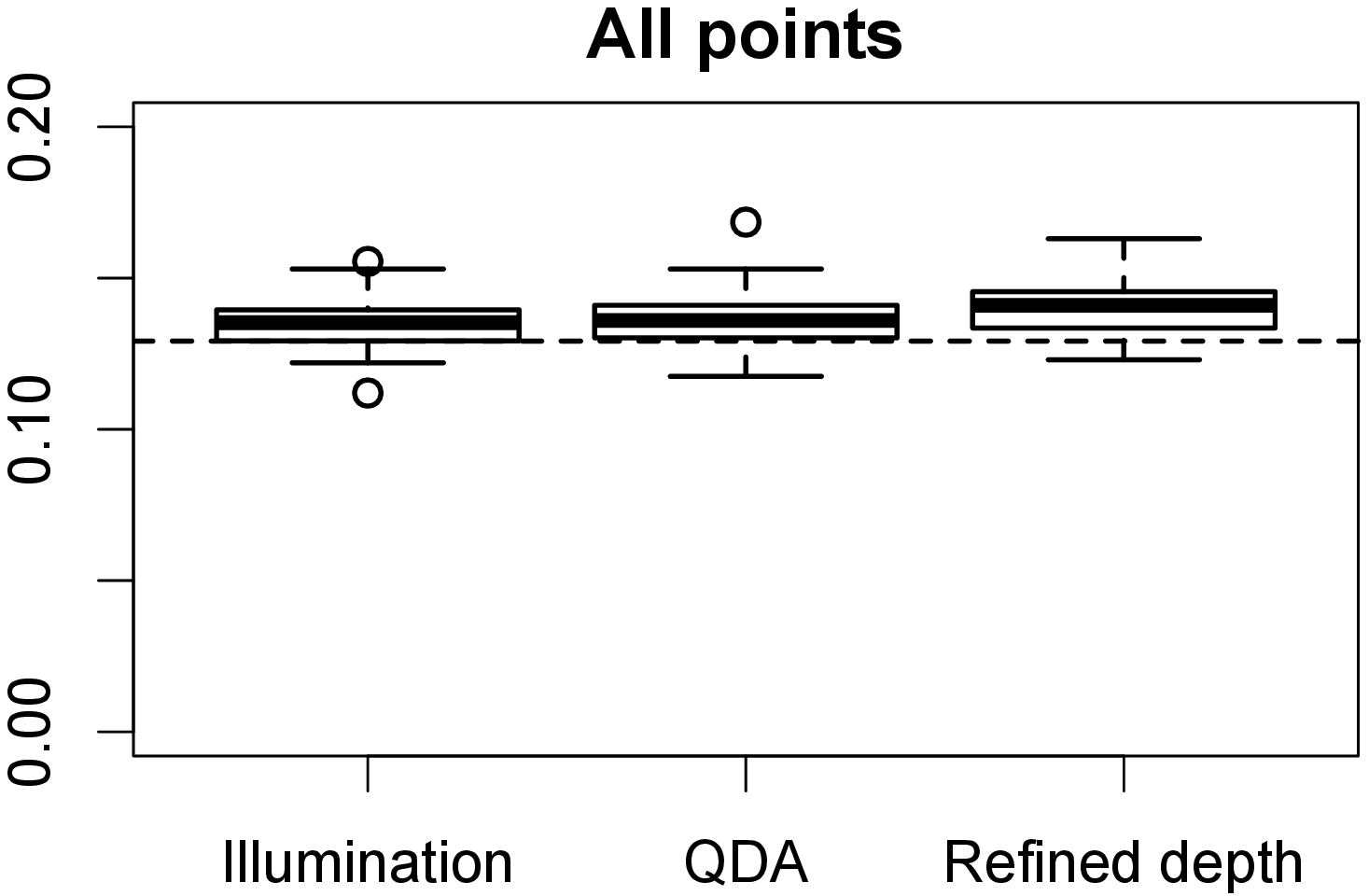}
\includegraphics[width=0.45\textwidth]{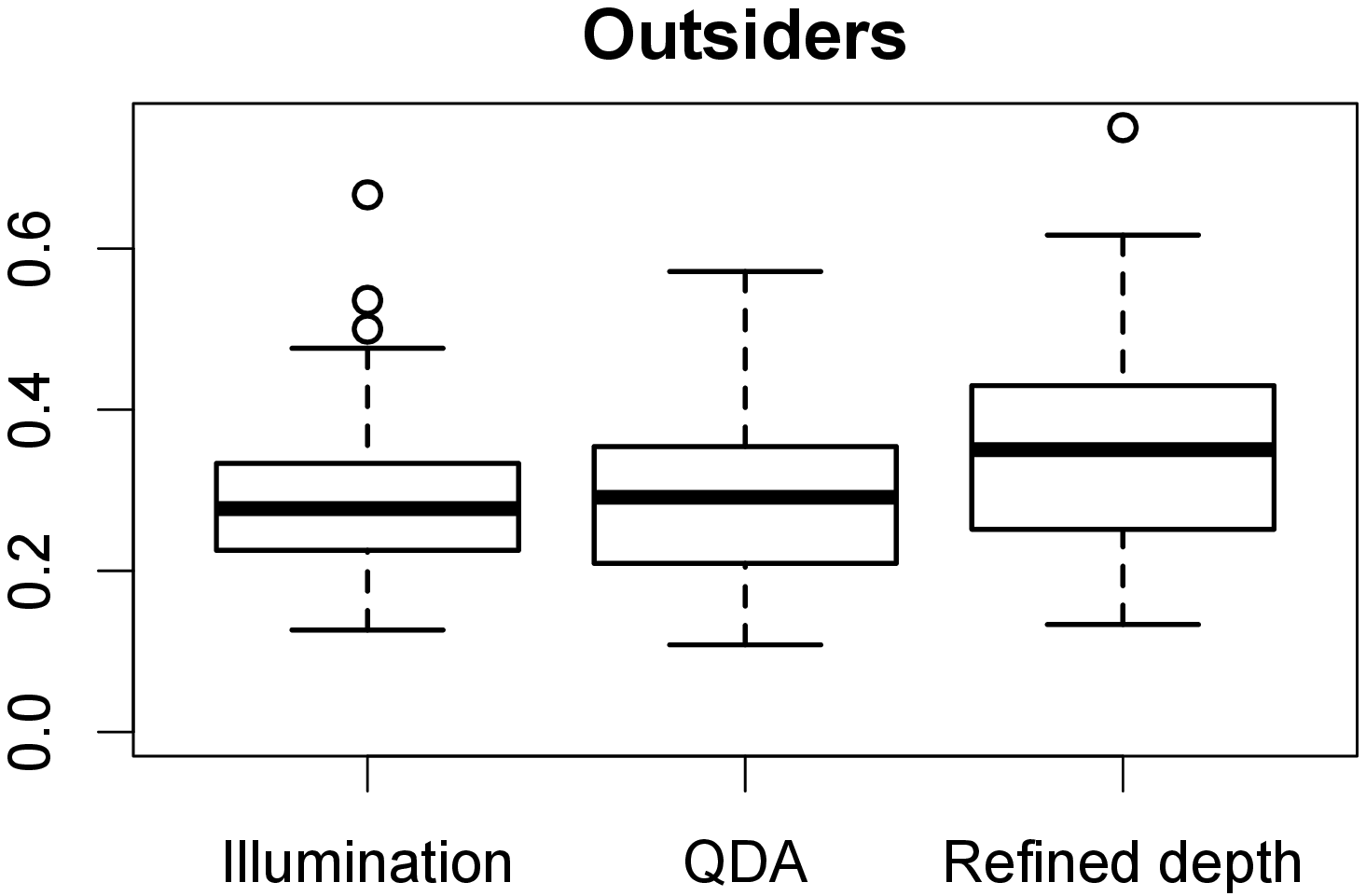}
\caption{Misclassification rates, based on 100 replications of the experiment with two bivariate elliptical distributions with different location and same scale. Based on all testing points (left panel) or the outsiders (right panel). Dashed horizontal line in the left panel corresponds to the theoretical Bayes error rate.}
\label{fig:classification elliptical location}
\end{figure}

\begin{table}[htpb]
\resizebox{\textwidth}{!}{
\begin{tabular}{c|c|c|c||c|c|c}
    & \multicolumn{3}{c||}{All points} & \multicolumn{3}{c}{Outsiders} \\ 
    & Illumination & QDA & Ref. depth & Illumination & QDA & Ref. depth \\ \hline
	0 \%   & 0.135 {\footnotesize (0.007)} & 0.136 {\footnotesize (0.008)} & 0.140 {\footnotesize (0.009)} & 0.286 {\footnotesize (0.089)} & 0.296 {\footnotesize (0.109)} & 0.353 {\footnotesize (0.127)} \\
  1 \%   & 0.134 {\footnotesize (0.008)} & 0.168 {\footnotesize (0.015)} & 0.142 {\footnotesize (0.010)} & 0.317 {\footnotesize (0.111)} & 0.374 {\footnotesize (0.119)} & 0.382 {\footnotesize (0.136)} \\
	5 \%   & 0.140 {\footnotesize (0.007)} & 0.242 {\footnotesize (0.017)} & 0.162 {\footnotesize (0.013)} & 0.285 {\footnotesize (0.092)} & 0.354 {\footnotesize (0.116)} & 0.384 {\footnotesize (0.107)} \\
	10 \%  & 0.165 {\footnotesize (0.014)} & 0.265 {\footnotesize (0.020)} & 0.175 {\footnotesize (0.014)} & 0.311 {\footnotesize (0.110)} & 0.336 {\footnotesize (0.114)} & 0.365 {\footnotesize (0.120)} \\ \hline
\end{tabular}}
\caption{Average misclassification rates and their standard deviations (in brackets), bivariate elliptical distributions with different location and same scale, level of contamination in one of the training samples ranging from 0 to 10~\%. Based on 100 replications of the experiment and all testing points (left part) and the outsiders (right part), respectively.}
\label{tab:classificationEllipticalContaminationLocation}
\end{table}

\section{R source code}	\label{appendix:code}

\begin{verbatim}
library(TukeyRegion)
library(geometry)

Illumination = function(X,x,alpha){
  # X: n-times-d matrix of the sample points (n points in d dimensions)
  # x: vector of length d whose illumination is computed
  # alpha: cut-off value for the illumination
  # returns 
  # I: the illumination of x onto the depth central region
  # (volume of the convex hull of points with hD at least alpha, and x)
  # volPa: volume of the depth central region
  # (volume of the region of points whose hD is at least alpha)
  Pa = TukeyRegion(X,depth=alpha*nrow(X),retVertices=TRUE,retVolume=TRUE)
  volPax = convhulln(rbind(Pa$vertices,x),options="FA")$vol
  return(list(I=volPax,volPa=Pa$volume))
}
\end{verbatim}

\section*{Acknowledgements}
We would like to thank John H. J. Einmahl and Jun Li for sharing the source code of the refined halfspace depth. This work was supported by the grant 19-16097Y of the Czech Science Foundation, and by the PRIMUS/17/SCI/3 project of Charles University.


\begin{thebibliography}{}

\bibitem[Brunel, 2019]{Brunel2019}
Brunel, V.-E. (2019).
\newblock Concentration of the empirical level sets of {T}ukey's halfspace
  depth.
\newblock {\em Probab. Theory Related Fields}, 173(3--4):1165--1196.

\bibitem[B\"{u}eler et~al., 2000]{Bueler_etal2000}
B\"{u}eler, B., Enge, A., and Fukuda, K. (2000).
\newblock Exact volume computation for polytopes: a practical study.
\newblock In {\em Polytopes---combinatorics and computation ({O}berwolfach,
  1997)}, volume~29 of {\em DMV Sem.}, pages 131--154. Birkh\"{a}user, Basel.

\bibitem[Donoho, 1982]{Donoho1982}
Donoho, D.~L. (1982).
\newblock Breakdown properties of multivariate location estimators.
\newblock Qualifying paper, Harvard University.

\bibitem[Donoho and Gasko, 1992]{Donoho_Gasko1992}
Donoho, D.~L. and Gasko, M. (1992).
\newblock Breakdown properties of location estimates based on halfspace depth
  and projected outlyingness.
\newblock {\em Ann. Statist.}, 20(4):1803--1827.

\bibitem[Dyckerhoff, 2004]{Dyckerhoff2004}
Dyckerhoff, R. (2004).
\newblock Data depths satisfying the projection property.
\newblock {\em Allg. Stat. Arch.}, 88(2):163--190.

\bibitem[Dyckerhoff, 2018]{Dyckerhoff2018}
Dyckerhoff, R. (2018).
\newblock Convergence of depths and depth-trimmed regions.
\newblock {\em arXiv preprint arXiv:1611.08721}.

\bibitem[Einmahl et~al., 2015]{Einmahl_etal2015}
Einmahl, J. H.~J., Li, J., and Liu, R.~Y. (2015).
\newblock Bridging centrality and extremity: refining empirical data depth
  using extreme value statistics.
\newblock {\em Ann. Statist.}, 43(6):2738--2765.

\bibitem[Emiris and Fisikopoulos, 2018]{Emiris_Fisikopoulos2018}
Emiris, I.~Z. and Fisikopoulos, V. (2018).
\newblock Practical polytope volume approximation.
\newblock {\em ACM Trans. Math. Software}, 44(4):Art. 38, 21.

\bibitem[Fang et~al., 1990]{Fang_etal1990}
Fang, K.~T., Kotz, S., and Ng, K.~W. (1990).
\newblock {\em Symmetric multivariate and related distributions}, volume~36 of
  {\em Monographs on Statistics and Applied Probability}.
\newblock Chapman and Hall, Ltd., London.

\bibitem[He and Einmahl, 2017]{He_Einmahl2017}
He, Y. and Einmahl, J. H.~J. (2017).
\newblock Estimation of extreme depth-based quantile regions.
\newblock {\em J. R. Stat. Soc. Ser. B. Stat. Methodol.}, 79(2):449--461.

\bibitem[Lange et~al., 2014]{Lange_etal2014}
Lange, T., Mosler, K., and Mozharovskyi, P. (2014).
\newblock Fast nonparametric classification based on data depth.
\newblock {\em Statist. Papers}, 55(1):49--69.

\bibitem[Liebscher, 2005]{Liebscher2005}
Liebscher, E. (2005).
\newblock A semiparametric density estimator based on elliptical distributions.
\newblock {\em J. Multivariate Anal.}, 92(1):205--225.

\bibitem[Liu et~al., 1999]{Liu_etal1999}
Liu, R.~Y., Parelius, J.~M., and Singh, K. (1999).
\newblock Multivariate analysis by data depth: descriptive statistics, graphics
  and inference.
\newblock {\em Ann. Statist.}, 27(3):783--858.

\bibitem[Liu et~al., 2006]{Liu_etal2006}
Liu, R.~Y., Serfling, R., and Souvaine, D.~L., editors (2006).
\newblock {\em Data depth: robust multivariate analysis, computational geometry
  and applications}, volume~72 of {\em DIMACS Series in Discrete Mathematics
  and Theoretical Computer Science}.
\newblock American Mathematical Society, Providence, RI.
\newblock Papers from the workshop held at Rutgers University, New Brunswick,
  NJ, May 14--16, 2003.

\bibitem[Liu et~al., 2019]{Liu_etal2019}
Liu, X., Mosler, K., and Mozharovskyi, P. (2019).
\newblock Fast computation of {T}ukey trimmed regions and median in dimension
  $p>2$.
\newblock {\em J. Comput. Graph. Statist.}
\newblock To appear.

\bibitem[Milman and Pajor, 1989]{Milman_Pajor1989}
Milman, V.~D. and Pajor, A. (1989).
\newblock Isotropic position and inertia ellipsoids and zonoids of the unit
  ball of a normed {$n$}-dimensional space.
\newblock In {\em Geometric aspects of functional analysis (1987--88)}, volume
  1376 of {\em Lecture Notes in Math.}, pages 64--104. Springer, Berlin.

\bibitem[Mordhorst and Werner, 2017a]{Mordhorst_Werner2017}
Mordhorst, O. and Werner, E.~M. (2017a).
\newblock Duality of floating and illumination bodies.
\newblock {\em arXiv preprint arXiv:1709.02424}.

\bibitem[Mordhorst and Werner, 2017b]{Mordhorst_Werner2017b}
Mordhorst, O. and Werner, E.~M. (2017b).
\newblock Duality of floating and illumination bodies for polytopes.
\newblock {\em arXiv preprint arXiv:1709.02429}.

\bibitem[Nagy et~al., 2019]{Nagy_etal2018s}
Nagy, S., Sch\"utt, C., and Werner, E.~M. (2019).
\newblock Halfspace depth and floating body.
\newblock {\em Stat. Surv.}
\newblock To appear.

\bibitem[Paindaveine and Van~Bever, 2018]{Paindaveine_VanBever2018}
Paindaveine, D. and Van~Bever, G. (2018).
\newblock Halfspace depths for scatter, concentration and shape matrices.
\newblock {\em Ann. Statist.}, 46(6B):3276--3307.

\bibitem[P\'{o}lya and Szeg\H{o}, 1998]{Polya_Szergo1998}
P\'{o}lya, G. and Szeg\H{o}, G. (1998).
\newblock {\em Problems and theorems in analysis. {I}}.
\newblock Classics in Mathematics. Springer-Verlag, Berlin.
\newblock Series, integral calculus, theory of functions, Translated from the
  German by Dorothee Aeppli, Reprint of the 1978 English translation.

\bibitem[Schneider, 2014]{Schneider2014}
Schneider, R. (2014).
\newblock {\em Convex bodies: the {B}runn-{M}inkowski theory}, volume 151 of
  {\em Encyclopedia of Mathematics and its Applications}.
\newblock Cambridge University Press, Cambridge, expanded edition.

\bibitem[Schuster, 1973]{Schuster1973}
Schuster, E.~F. (1973).
\newblock On the goodness-of-fit problem for continuous symmetric
  distributions.
\newblock {\em J. Amer. Statist. Assoc.}, 68:713--715.

\bibitem[Sch\"utt, 1997]{Schutt1997}
Sch\"utt, C. (1997).
\newblock Floating body, illumination body, and polytopal approximation.
\newblock {\em C. R. Acad. Sci. Paris S\'er. I Math.}, 324(2):201--203.

\bibitem[Sch{\"u}tt and Werner, 1990]{Schutt_Werner1990}
Sch{\"u}tt, C. and Werner, E.~M. (1990).
\newblock The convex floating body.
\newblock {\em Math. Scand.}, 66(2):275--290.

\bibitem[Sch\"utt and Werner, 1992]{Schutt_Werner1992}
Sch\"utt, C. and Werner, E.~M. (1992).
\newblock The convex floating body of almost polygonal bodies.
\newblock {\em Geom. Dedicata}, 44(2):169--188.

\bibitem[Serfling, 2006a]{Serfling2006}
Serfling, R. (2006a).
\newblock Depth functions in nonparametric multivariate inference.
\newblock In {\em Data depth: robust multivariate analysis, computational
  geometry and applications}, volume~72 of {\em DIMACS Ser. Discrete Math.
  Theoret. Comput. Sci.}, pages 1--16. Amer. Math. Soc., Providence, RI.

\bibitem[Serfling, 2006b]{Serfling_symmetry}
Serfling, R. (2006b).
\newblock Multivariate symmetry and asymmetry.
\newblock {\em Encyclopedia of Statistical Sciences, Second Edition},
  8:5338--5345.

\bibitem[Tukey, 1975]{Tukey1975}
Tukey, J.~W. (1975).
\newblock Mathematics and the picturing of data.
\newblock In {\em Proceedings of the {I}nternational {C}ongress of
  {M}athematicians ({V}ancouver, {B}. {C}., 1974), {V}ol. 2}, pages 523--531.
  Canad. Math. Congress, Montreal, Que.

\bibitem[van~der Vaart, 1998]{Vandervaart1998}
van~der Vaart, A.~W. (1998).
\newblock {\em Asymptotic statistics}, volume~3 of {\em Cambridge Series in
  Statistical and Probabilistic Mathematics}.
\newblock Cambridge University Press, Cambridge.

\bibitem[Werner, 1996]{Werner1996}
Werner, E. (1996).
\newblock The illumination bodies of a simplex.
\newblock {\em Discrete Comput. Geom.}, 15(3):297--306.

\bibitem[Werner, 1994]{Werner1994}
Werner, E.~M. (1994).
\newblock Illumination bodies and affine surface area.
\newblock {\em Studia Math.}, 110(3):257--269.

\bibitem[Werner, 1997]{Werner1997}
Werner, E.~M. (1997).
\newblock The illumination body of almost polygonal bodies.
\newblock {\em Geom. Dedicata}, 64(3):343--354.

\bibitem[Werner, 2006]{Werner2006}
Werner, E.~M. (2006).
\newblock Floating bodies and illumination bodies.
\newblock In {\em Integral geometry and convexity}, pages 129--140. World Sci.
  Publ., Hackensack, NJ.

\bibitem[Zuo and Serfling, 2000a]{Zuo_Serfling2000}
Zuo, Y. and Serfling, R. (2000a).
\newblock General notions of statistical depth function.
\newblock {\em Ann. Statist.}, 28(2):461--482.

\bibitem[Zuo and Serfling, 2000b]{Zuo_Serfling2000b}
Zuo, Y. and Serfling, R. (2000b).
\newblock On the performance of some robust nonparametric location measures
  relative to a general notion of multivariate symmetry.
\newblock {\em J. Stat. Plan. Inference}, 84(1-2):55--79.

\end{thebibliography}

\end{document}